\documentclass[11pt,letterpaper]{amsart}

\usepackage{url}

\usepackage{hyperref}

\usepackage[final]{showlabels}

\input{macros.sty}

\newif\iffinalrun

\iffinalrun
\else

\fi

\iffinalrun
	\newcommand{\need}[1]{}
	\newcommand{\mar}[1]{}
\else
	\newcommand{\need}[1]{{ \tiny *** #1}}
	\newcommand{\mar}[1]{\marginpar{\raggedright\tiny #1}}
\fi

\title{Modularity of trianguline Galois representations}
\author[R. Bellovin]{Rebecca Bellovin}

\begin{document}

    \begin{abstract}
	    We use the theory of trianguline $(\varphi,\Gamma)$-modules over pseudorigid spaces to prove a modularity lifting theorem for certain Galois representations which are trianguline at $p$, including those with characteristic $p$ coefficients.  The use of pseudorigid spaces lets us construct integral models of the trianguline varieties of ~\cite{breuil-hellmann-schraen}, ~\cite{chenevier2013} after bounding the slope, and we carry out a Taylor--Wiles patching argument for families of overconvergent modular forms.  This permits us to construct a patched quaternionic eigenvariety and deduce our modularity results.
    \end{abstract}

    \maketitle
    
\section{Introduction}

The Fontaine--Mazur conjecture predicts that representations of Galois groups of number fields which are sufficiently nice should come from geometry.  In practice, the way one proves this is by proving so-called automorphy lifting theorems, relating the Galois representations of interest to Galois representations already known to have the desired properties.

In this context, if $\rho:\Gal_{F}\rightarrow \GL_n(\overline\Q_p)$ is the representation, ``sufficiently nice'' includes a condition on the local Galois group at $p$ called being \emph{geometric}.  In the present paper, motivated by a question of Andreatta--Iovita--Pilloni~\cite{aip2018}, we consider a characteristic $p$ analogue of this conjecture.  There is no definition of ``geometric'' for a Galois representation with positive characteristic coefficients, but we replace it with the condition \emph{trianguline}:
\begin{thm-intro}\label{thm: main thm}
	Assume $p\geq 5$, and let $L$ be a finite extension of $\F_p(\!(u)\!)$.  Let $\rho:\Gal_{\Q}\rightarrow \GL_2(\O_L)$ be an odd continuous Galois representation unramified away from $p$ such that the $(\varphi,\Gamma)$-module $D_{\rig}(\rho|_{\Gal_{Q_p}})$ is trianguline with regular parameters.  Assume moreover that the reduction $\overline\rho$ is modular and satisfies certain additional technical hypotheses.  Then $\rho$ is the twist of the Galois representation corresponding to a point on the extended eigencurve $\mathscr{X}_{\GL_2}$.
\end{thm-intro}

The eigencurve $\mathscr{X}_{\GL_2}^{\rig}$ was originally constructed by Coleman--Mazur, and it is a rigid analytic space whose points correspond to \emph{overconvergent modular forms}.  Points corresponding to classical eigenforms (of varying weight and level) are dense, so we can think of it as a moduli space of $p$-adic modular forms.  Each point of the eigencurve has a Galois representation attached, but Kisin~\cite{kisin03} showed that the Galois representations at non-classical points are not geometric at $p$.  Instead, they are trianguline (though he did not use this terminology; it was introduced subsequently by Colmez).  A converse was proved by Emerton ~\cite[Theorem 1.2.4]{emerton2011} when the coefficients are $p$-adic.

Given a $p$-adic Galois representation $\rho$, there is an associated object $D_{\rig}(\rho)$ called a $(\varphi,\Gamma)$-module; at the expense of making the coefficients more complicated, the Galois representation can be captured as the action of a semi-linear operator $\varphi$ together with the action of a $1$-dimensional $p$-adic Lie group $\Gamma$.  Then even if $\rho$ is irreducible, it is possible for $D_{\rig}(\rho)$ to be reducible.  Kisin showed that this happens in small neighborhoods of classical points on the eigencurve; if $\rho_x$ is the Galois representation attached to a point $x$, there is an exact sequence
\[	0\rightarrow D_1\rightarrow D_{\rig}(\rho_x)\rightarrow D_2\rightarrow 0	\]
where $D_1$ and $D_2$ are rank-$1$ $(\varphi,\Gamma)$-modules. There is a basis element $\mathbf{e}_1$ of $D_1$ such that $\varphi$ acts on $\mathbf{e}_1$ by the $U_p$-eigenvalue at $x$ and $\Gamma$ acts on $\mathbf{e}_1$ trivially.  This construction was extended over (a normalization of) the eigencurve in separate work of ~\cite{kpx} and ~\cite{liu}.

The eigencurve is equipped with a map $\mathrm{wt}:\mathscr{X}_{\GL_2}^{\rig}\rightarrow \mathscr{W}^{\rig}$ to \emph{weight space}, which we may view as the disjoint union of $p-1$ rigid analytic open unit disks.  The existence of Galois representations attached to eigenforms means it is also equipped with a morphism $\mathscr{X}_{\GL_2}^{\rig}\rightarrow \G_m^{\rig}\times\coprod_{\overline\rho}R_{\overline\rho}$, where the $R_{\overline\rho}$ are Galois deformation rings (more precisely, deformation rings of pseudocharacters), and $\G_m^{\rig}$ corresponds to the eigenvalue of the Hecke operator $U_p$.  The triangulation results of ~\cite{kisin03},  ~\cite{kpx}, and ~\cite{liu} mean that we can combine these two maps to get a morphism
\[	\mathscr{X}_{\GL_2}\rightarrow \coprod_{\overline\rho}X_{{\tri},\overline\rho}^{\psi,\kappa,\rig}	\]
to a moduli space of trianguline Galois representations (here the decorations $\psi$ and $\kappa$ simply mean we are fixing the determinant and the parameters of the triangulation).  The result of ~\cite{emerton2011} then shows that this morphism surjects onto certain components.

More recently, the construction of the eigencurve has been extended to mixed characteristic by Andreatta--Pilloni--Iovita~\cite{aip2018},~\cite{aip2016} and Johansson--Newton~\cite{johansson-newton}, using Huber's theory of adic spaces instead of Tate's theory of rigid analytic spaces.  These authors construct \emph{pseudorigid spaces} containing characteristic $0$ eigenvarieties as open subspaces, with non-empty characteristic $p$ loci.

In previous work, we generalized the construction of $(\varphi,\Gamma)$-module to families of Galois representations with pseudorigid coefficients ~\cite{bellovin2020} and showed that the triangulation of the eigencurve extends to the boundary characteristic $p$ points ~\cite{bellovin2021}.  This yields an analogous morphism $\mathscr{X}_{\GL_2}\rightarrow\coprod X_{{\tri},\overline\rho}^{\psi,\kappa}$ of pseudorigid spaces.  In the present paper, we use that machinery to prove a modularity result for Galois representations trianguline at $p$, characterizing the image in many components.

The proof rests on the Taylor--Wiles patching method, as reformulated in ~\cite{scholze-lubin-tate}.  This is the source of the aforementioned technical hypotheses on $\overline\rho$ (which amount to assumptions about the image of $\overline\rho$ being sufficiently big).  However, there are a number of technical complications.  For example, to carry out some preliminary reductions, we first prove a version of the Jacquet--Langlands correspondence on eigenvarieties extending the construction of ~\cite{birkbeck2019}, and we characterize the image of the cyclic base change morphism $\mathscr{X}_{\GL_2/\Q}\rightarrow \mathscr{X}_{\GL_2/F}$ of ~\cite{johansson-newton17}.  The latter uses the construction of an auxiliary ``$\Gal(F/\Q)$-fixed'' eigenvariety, which may be of independent interest.  This permits us to transfer the problem to overconvergent quaternionic modular forms over a cyclic totally real extension of $\Q$.

The modules of quaternionic automorphic forms we patch are those constructed in ~\cite{johansson-newton}.  We construct trianguline deformation rings which act on them, and we patch by introducing ramification at additional primes.  But the construction of trianguline deformation rings is delicate, because in general triangulations of $(\varphi,\Gamma)$-modules do not interact well with integral structures on the corresponding Galois representation.  Thus, we crucially use the pseudorigid theory of triangulations (and not just the rigid analytic theory) to ensure that we can construct an integral quotient of a Galois deformation ring whose analytic points are trianguline, with Frobenius eigenvalues bounded by a fixed slope.

This leads to a further difficulty, which is that it is difficult to study the components of the trianguline deformation ring directly.  Instead, we patch families of overconvergent automorphic forms, which lets us compare the Galois representation we are interested in with ``nearby'' potentially Barsotti--Tate representations which are known to be automorphic.  Along the way, we construct local pieces of a patched quaternionic eigenvariety $\mathscr{X}_{\underline D^\times}^{\infty}$, together with a morphism to a trianguline variety and a patched module of overconvergent modular forms.  We note that it is only possible to patch families of overconvergent automorphic forms because we constructed an integral model of the trianguline variety; we know almost nothing about its structure away from nice points in the analytic locus, but understanding it better would be very interesting.  We also hope to glue these local patched modules in future work.

We have not attempted to work in maximum generality.  In particular, it should be possible to relax the ramification condition and prove an overconvergent modularity lifting theorem for certain totally real fields.  However, this would require constructing and studying a cyclic base change morphism for more general extensions of number fields.  We expect that it is possible to construct these morphisms for the middle-degree eigenvariety over a totally real field, which would lead to stronger trianguline modularity theorems in characteristic $0$.  But we were forced to assume the degree of the cyclic extension was prime to $p$ to characterize the image of a base change morphism in positive characteristic, so additional work would be required to strengthen our results in positive characteristic.

We further remark that our ``big image'' condition on the residual Galois representation is stronger than the standard one.  This is to ensure that we have access to the necessary cohomological vanishing theorems, to permit us to work with middle-degree eigenvarieties.

The work of Breuil--Hellmann--Schraen ~\cite{breuil-hellmann-schraen} constructs a similar patched eigenvariety for unitary groups, using completed cohomology rather than overconvergent cohomology.  It would be extremely interesting to relate these two constructions.

We now describe the structure of this paper.  We begin by recalling the theory of trianguline $(\varphi,\Gamma)$-modules and their deformations; this permits us to construct and study pseudorigid trianguline varieties (generalizing those of ~\cite{chenevier2013} and ~\cite{breuil-hellmann-schraen}).  We compute the dimension of these pseudorigid trianguline varieties with fixed determinant and weight, and we show that they have an integral model after bounding the slopes of the rank-$1$ constituents.

We then turn to the automorphic theory we will need.  We prove that so-called \emph{twist classical} points are very Zariski dense in the eigenvariety $\mathscr{X}_{\underline D^\times}$, which permits us to interpolate the Jacquet--Langlands correspondence to extended eigenvarieties and permits us to conclude that $\mathscr{X}_{\underline D^\times}$ is reduced (extending the results of ~\cite{birkbeck2019} and ~\cite{chenevier2005}).  We also study the cyclic base change morphism $\mathscr{X}_{\GL_2/\Q}\rightarrow \mathscr{X}_{\GL_2/F}$ of ~\cite{johansson-newton17}; when $F$ is totally real and $[F:\Q]$ is prime to $p$, we show that $x\in \mathscr{X}_{\GL_2/F}$ is in the image if and only if it is fixed by $\Gal(F/\Q)$.  To do this, we study the ``$\Gal(F/\Q)$-fixed locus'' in the Hilbert eigenvariety, and show that classical points are dense in it.

Finally, we turn to the patching argument.  We show that our modules of integral overconvergent automorphic forms are projective at maximal points of weight space, and we show that we can add certain kinds of level structure.  Then using the standard Taylor--Wiles patching construction, we construct a patched module with the support we expect.  This permits us to deduce the desired modularity statement, by deformation from potentially Barsotti--Tate points in characteristic $0$.  This last step requires the results of ~\cite{kisin2009moduli}.

\subsection*{Notation}

We fix some running notation and hypotheses.  In section ~\ref{section: trianguline varieties} we assume that $p\geq 3$, because we only developed the theory of $(\varphi,\Gamma)$-modules over pseudorigid spaces in that situation.  In sections ~\ref{section: extended eigenvarieties} and ~\ref{section: patching}, we assume $p\geq 5$; we need this hypothesis to construct eigenvarieties (and the Jacquet--Langlands and cyclic base change morphisms between them) at tame level $1$, and later to apply Taylor--Wiles patching.

We normalize class field theory so that it sends uniformizers to geometric Frobenius, and we normalize Hodge--Tate weights so that the cyclotomic character has Hodge--Tate weight $-1$.

If $X$ is a group isomorphic to $X_0\times \Z_p^{\oplus r}\times \Z^{\oplus s}$, where $X_0$ is a finite group, we let $\widehat X:=\underline\Hom(X,\G_m^{\ad})$ denote the functor $R\mapsto \Hom_{\mathrm{cts}}(X,R^\times)$.

\subsection*{Acknowledgments}

I would like to thank A. Caraiani, T. Gee, J. Newton, and V. Pilloni for many helpful conversations, as well as useful comments on earlier versions of this paper.  I would also like to thank the anonymous referee for reading this paper extremely carefully and making many helpful comments.

\section{Trianguline varieties and Galois deformation rings}\label{section: trianguline varieties}

\subsection{Galois deformation rings}

Let $\E/\Q_p$ be a finite extension, with ring of integers $\O_E$, uniformizer $\varpi_E$, and residue field $\F$, and let $G$ be a profinite group satisfying Mazur's condition $\Phi_p$.  The two cases we will be most interested in are $G=\Gal_K$ and $G=\Gal_{F,S}$, where $K$ is a finite extension of $\Q_p$, and $F$ is a number field, and $S$ is a set of places of $F$.

Suppose we have a continuous homomorphism $\overline\rho:G\rightarrow \GL_d(\F)$.  Then we may construct the univeral framed deformation ring $R_{\overline\rho}^\square$, which pro-represents the functor
\[	A\rightsquigarrow \{\rho:G\rightarrow \GL_d(A)\mid \rho\equiv \overline\rho\pmod{\mathfrak{m}_A}\}	\]
on the category of complete local noetherian $\O_E$-algebras with residue field $\F$, of lifts of $\overline\rho$, that is, deformations of $\overline\rho$ together with a basis.  If $\End_G(\overline\rho)=\F$ (for example, if $\overline\rho$ is absolutely irreducible), we additionally have the universal (unframed) deformation ring $R_{\overline\rho}$ parametrizing deformations of $\rho$.

If $R$ is a complete local noetherian $\O_E$-algebra with maximal ideal $\mathfrak{m}_R$ and finite residue field, and $\psi:\Gal_K\rightarrow R^\times$ is a continuous character such that $\det\overline\rho=\psi\mod \mathfrak{m}_R$, there is a quotient $R\htimes R_{\overline\rho}^\square\twoheadrightarrow R_{\overline\rho}^{\square,\psi}$ parametrizing lifts of $\overline\rho$ with determinant $\psi$.  Indeed, there is a homomorphism $R_{\det\overline\rho}\rightarrow R_{\overline\rho}^\square$ given by the determinant map, and the choice of $\psi$ defines a homomorphism $R_{\det\overline\rho}\rightarrow R$; then $R_{\overline\rho}^{\square,\psi}=R\htimes_{R_{\det\overline\rho}}R_{\overline\rho}^\square$.  If $\End_G(\overline\rho)=\F$, there is similarly a quotient $R\htimes R_{\overline\rho}\twoheadrightarrow R_{\overline\rho}^\psi$ parametrizing deformations of $\overline\rho$ with determinant $\psi$.

Now we specialize to the arithmetic situations of interest.  Let $K/\Q_p$ be a finite extension, and assume that $\Hom(K,E)$ has cardinality $[K:\Q_p]$.  Then by ~\cite[Corollary 3.37]{boeckle-iyengar-paskunas}, $R_{\overline\rho}^\square$ is a complete intersection, and by ~\cite[Corollary 4.21]{boeckle-iyengar-paskunas} the irreducible components of $\Spec R_{\overline\rho}^\square$ are in bijection with the irreducible components of $\Spec R_{\det\overline\rho}$.  More precisely, if $\mu:=\mu_{p^\infty}(K)$ denotes the $p$-power roots of unity in $K^\times$, local class field theory identifies it with a subgroup of $\Gal_K^{\mathrm{ab}}$; by ~\cite[Lemma 4.1]{boeckle-iyengar-paskunas} $R_{\det\overline\rho}$ is a power series ring over $\O_E[\mu]$, so its irreducible components are in bijection with characters $\chi:\mu\rightarrow \O_E^\times$.  There are quotients $R_{\overline\rho}^\square\twoheadrightarrow R_{\overline\rho}^{\square,\chi}$ parametrizing lifts of $\overline\rho$ whose determinant restricted to $\mu$ (via the Artin map) agrees with $\chi$, and by ~\cite[Corollary 4.5, Corollary 4.19]{boeckle-iyengar-paskunas} the rings $R_{\overline\rho}^{\square,\chi}$ are normal domains and complete intersections.  In particular, $R_{\overline\rho}^{\square}$ is reduced.

Let $F$ be a number field and let $\Sigma_p:=\{v\mid p\}$.  If $\rho:\Gal_F\rightarrow \GL_d(\F)$ is a continuous representation and $v$ is a place of $F$, we let $\rho_v$ denote $\rho|_{\Gal_{F_v}}$.  Suppose that $\overline\rho$ is absolutely irreducible, and let $S$ be a finite set of places of $F$ containing $\Sigma_p$ and the infinite places such that $\overline\rho$ is unramified outside $S$.  Then we let $R_{\overline\rho,S}$ denote the universal deformation ring parametrizing deformations unramified outside of $S$, and we let $R_{\overline\rho,S}^{\square}$ denote the universal deformation ring whose $A$-points are deformations $\rho_A$ of $\overline\rho$ unramified outside of $S$, together with bases for $\rho_A|_{\Gal_{F_v}}$ for each $v\in\Sigma_p$.  We also let $R_{\overline\rho,\loc}^{\square}:=\otimes_{v\in\Sigma_p}R_{\overline\rho_v}^{\square}$.  

If $\psi:\Gal_F\rightarrow R^\times$ is a continuous character as above, we let 
\begin{align*}
	R_{\overline\rho,S}^\psi&:=R\htimes_{R_{\det\overline\rho,S}}R_{\overline\rho,S}	\\
	R_{\overline\rho,S}^{\square,\psi}&:=R\htimes_{R_{\det\overline\rho,S}}R_{\overline\rho,S}^{\square}	\\
	R_{\overline\rho,\loc}^{\square,\psi}&:=R\htimes_{R_{\det\overline\rho,\loc}}R_{\overline\rho,\loc}^{\square}
\end{align*}

For any place $v\in\Sigma_p$, restriction from $\Gal_{F,S}$ to $\Gal_{F_v}$ defines a homomorphism $R_{\overline\rho_v}^\square\rightarrow R_{\overline\rho,S}^\square$, and so we obtain homomorphisms
\[	R_{\overline\rho,\loc}^{\square}\rightarrow R_{\overline\rho,S}^\square	\]
and
\[	R_{\overline\rho,\loc}^{\square,\psi}\rightarrow R_{\overline\rho,S}^{\square,\psi}	\]
We can relate our local and global deformation rings more precisely:
\begin{lemma}\label{lemma: local global def ring top gens}
	Suppose that $p\nmid d$.  Let $h^1$ denote the dimension (as an $\F$-vector space) of
	\[	\ker\left(H^1(\Gal_{F,S},\ad^0(\overline\rho))\rightarrow \prod_{v\in \Sigma_p}H^1(\Gal_{F_v},\ad^0(\overline\rho_v))\right)	\]
	let $\delta_F:=\dim_{\F}H^0(\Gal_{F,S},\ad\overline\rho)$, and for $v\in\Sigma_p$ let $\delta_v:=\dim_{\F}H^0(\Gal_{F_v},\ad\overline\rho_v)$.  Then $R_{\overline\rho,S}^{\square,\psi}$ can be topologically generated over $R_{\overline\rho,\loc}^{\square,\psi}$ by $g:=h^1+\sum_{v\in\Sigma_p}\delta_v-\delta_F$ elements.
	\label{lemma: def ring global local}
\end{lemma}
\begin{proof}
	Let $\mathfrak{m}_{\loc}$ denote the maximal ideal of $R_{\overline\rho,\loc}^{\square,\psi}$ and let $\mathfrak{m}_S$ denote the maximal ideal of $R_{\overline\rho,S}^{\square,\psi}$.  We need to compute the relative tangent space $\left(\mathfrak{m}_S/(\mathfrak{m}_S^2,\mathfrak{m}_{\loc})\right)^\ast$ of $R_{\overline\rho,S}^{\square,\psi}/\mathfrak{m}_{\loc}$.  But the maximal ideal of $R$ is contained in $\mathfrak{m}_{\loc}$, so we may assume that $\psi$ is constant, and the result follows from ~\cite[Lemma 3.2.2]{kisin2009moduli}.
\end{proof}

\subsection{Deformations of trianguline $(\varphi,\Gamma)$-modules}

Trianguline $(\varphi,\Gamma)$-modules are those which are extensions of $(\varphi,\Gamma)$-modules of character type.  More precisely,
\begin{definition}
        Let $X$ be a pseudorigid space over $\mathscr{O}_E$ for some finite extension $E/\Q_p$, let $K/\Q_p$ be a finite extension, and let $\underline\delta=(\delta_1,\ldots,\delta_d):(K^\times)^d\rightarrow \Gamma(X,\mathscr{O}_X^\times)$ be a $d$-tuple of continuous characters.  A $(\varphi,\Gamma_{K})$-module $D$ is \emph{trianguline with parameter $\underline\delta$} if (possibly after enlarging $E$) there is an increasing filtration $\Fil^\bullet D$ by $(\varphi,\Gamma_{K})$-modules and a set of line bundles $\mathscr{L}_1,\ldots,\mathscr{L}_d$ such that $\gr^iD\cong \Lambda_{X,\rig,K}(\delta_i)\otimes \mathscr{L}_i$ for all $i$.

        If $X=\Spa R$ where $R$ is a field, we say that $D$ is \emph{strictly trianguline with parameter $\underline\delta$} if for each $i$, $\Fil^{i+1}D$ is the unique sub-$(\varphi,\Gamma_{K})$-module of $D$ containing $\Fil^iD$ such that $\gr^{i+1}D\cong \Lambda_{R,\rig,K}(\delta_{i+1})$.
\end{definition}

As in the characteristic $0$ situation treated in ~\cite[\textsection 2.3]{bellaiche-chenevier}, we may define and study deformations of trianguline $(\varphi,\Gamma)$-modules:
\begin{definition}
        Let $R$ be a finite extension of $\F_p(\!(u)\!)$ and let $D$ be a fixed $(\varphi,\Gamma_{K})$-module of rank $d$ over $\Lambda_{R,\rig,K}$ equipped with a triangulation $\Fil^\bullet D$ with parameter $\underline\delta$.  Let $\mathcal{C}_R$ denote the category of artin local $\Z_p$-algebras $R'$ equipped with an isomorphism $R'/\mathfrak{m}_{R'}\xrightarrow{\sim}R$.  The trianguline deformation functor $\mathrm{Def}_{D,\Fil^\bullet}:\mathcal{C}_R\rightarrow\underline{\mathrm{Set}}$ is defined to be the set of isomorphism classes 
        \[      \mathrm{Def}_{D,\Fil^\bullet}(R'):=\{(D_{R'},\Fil^\bullet D_{R'}, \iota)\}/\sim \]
where $D_{R'}$ is a $(\varphi,\Gamma_{K})$-module over $\Lambda_{R',\rig,K}$, $\Fil^\bullet D_{R'}$ is a triangulation, and $\iota:R\otimes_{R'}D_{R'}\xrightarrow{\sim}D$ is an isomorphism which also defines isomorphisms $R\otimes_{R'}\Fil^iD_{R'}\xrightarrow{\sim}\Fil^iD$.
\end{definition}

One of the consequences of the proof of ~\cite[Proposition 5.1]{bellovin2020} is that when $d=1$, $\mathrm{Def}_{D,\Fil^\bullet}$ is formally smooth.  As in the characteristic $0$ situation, the same is true for general $d$, so long as the parameter satisfies a certain regularity condition.  Note that the regularity condition in here is slightly different than in characteristic $0$; the additional characters avoided in the statement of ~\cite[Proposition 2.3.10]{bellaiche-chenevier} do not make sense in characteristic $p$.
\begin{prop}
	Suppose the parameter $\underline\delta$ of $\Fil^\bullet D$ satisfies the property that $\delta_i\delta_j^{-1}\neq \chi_{\mathrm{cyc}}\circ\Nm_{K/\Q_p}$ for any $i<j$.  Then $\mathrm{Def}_{D,\Fil^\bullet}$ is formally smooth.
\end{prop}
\begin{proof}
	The proof is essentially identical to that of ~\cite[Proposition 2.3.10]{bellaiche-chenevier}, but we sketch it here for the convenience of the reader.  We proceed by induction on $d$; the case $d=1$ follows from the proof of ~\cite[Proposition 5.1]{bellovin2020}, so we assume the result for trianguline deformations of $(\varphi,\Gamma)$-modules of rank $d-1$.  Let $I\subset R'$ be a square-zero ideal.  We need to prove that $\mathrm{Def}_{D,\Fil^\bullet}(R')\rightarrow \mathrm{Def}_{D,\Fil^\bullet}(R'/I)$ is surjective, so we may factor $R'\twoheadrightarrow R'/I$ into a series of small extensions and assume that $I$ is principal and $I\mathfrak{m}_{R'}=0$.  By the inductive hypothesis, we may find a trianguline deformation $D'$ of $\Fil^{d-1}D$ over $\Lambda_{R',\rig,L}$.  By twisting, we may assume that $\delta_d$ is trivial.  Then we need to show that the natural map $H_{\varphi,\Gamma}^1(D')\rightarrow H_{\varphi,\Gamma}^1(\Fil^{d-1})$ is surjective.  But the cokernel of this map is $H_{\varphi,\Gamma}^2(I\otimes_{R'/\mathfrak{m}_{R'}}\Fil^{d-1}D(\delta_d^{-1}))=I\otimes_{R'/\mathfrak{m}_{R'}}H_{\varphi,\Gamma}^2(\Fil^{d-1}D(\delta_d^{-1}))$, which is $0$ by assumption and ~\cite[Corollary 4.11]{bellovin2020}.
\end{proof}

In order to build moduli spaces of trianguline $(\varphi,\Gamma)$-modules, we will use moduli spaces of characters, as in ~\cite[\textsection 2.3]{bellovin2021}.  If $G$ is a commutative $p$-adic Lie group and $G'\subset G$ is a compact subgroup such that $G/G'$ is free and finitely generated, then we have $\widehat{G'}:=\Spa \Z_p[\![G']\!]$ and the pseudorigid spaces $\widehat{G'}^{\an}$ and $\widehat{G}^{\an}:=\Spa(\Z[G/G'],\Z)\times_{\Z}\widehat{G'}^{\an}$.  If $X$ is a pseudorigid space, we also have the pseudorigid space $\widehat G_X$, which represents the functor 
\[	Y\rightsquigarrow \Hom_{\cts}(G,\O(Y))	\]
on the category of adic spaces over $X$.

In particular, if $K$ is a finite extension of $\Q_p$, we will be interested in $\widehat{K^\times}^{\an}$ and $\widehat{(K^\times)^d}^{\an}$ for $d\geq 1$:
\begin{definition}
	We let $\mathcal{T}:=\widehat{K^\times}^{\an}$, and for any $d\geq 1$, we write $\mathcal{T}^d:=\widehat{(K^\times)^d}^{\an}$.
        \label{def:G hat}
\end{definition}
We see that $\widehat{K^\times}^{\an}\cong \G_m^{\ad}\times_{\Z}\Spa \Z_p[\![\O_K^\times]\!]^{\an}$, and $\mathcal{T}^d\cong \G_m^{\ad,d}\times_{\Z}\Spa \Z_p[\![(\O_K^\times)^d]\!]^{\an}$.  Since $\O_K^\times$ is compact, $\Spa \Z_p[\![\O_K^\times]\!]^{\an}$ is a quasi-compact pseudorigid space; it has a finite cover $\{U_i:=\Spa R_i\}$ by affinoid subspaces, and $\G_{m,U_i}$ is a rising union of relative annuli $C_{U_i,h}:=\Spa R_i\left\langle u^hT, u^hT^{-1}\right\rangle$.

If $K=\Q_p$, then $\widehat{\Q_p^\times}^{\an}$ has connected components indexed by the elements of $\mu_{p-1}$, each of which is isomorphic to $\left(\Spa \Z_p[\![\Z_p]\!]\right)^{\an}\times\G_m^{\ad}$.

\begin{remark}
	In the pseudorigid setting (unlike the classical rigid analytic setting), it is not true that $\widehat{G_1\times G_2}^{\an}\cong \widehat{G_1}^{\an}\times \widehat{G_2}^{\an}$.  Indeed, $\Spa \Z_p[\![T_1,T_2]\!]^{\an}$ consists of all valuations which do not vanish on all three of $p, T_1, T_2$.  But 
	\[	\Spa \Z_p[\![T_1]\!]^{\an}\times_{\Z_p}\Spa \Z_p[\![T_2]\!]^{\an}	\]
	also excludes valuations vanishing at both $p$ and $T_1$ (or both $p$ and $T_2$).  In particular, $\mathcal{T}^d$ is \emph{not} a product of copies of $\mathcal{T}$.
        \label{rmk: products char varieties}
\end{remark}

\begin{definition}
        We say that a continuous character $\kappa:K^\times\rightarrow \O(X)^\times$ is \emph{regular} if for all maximal points $x\in X$, the residual character $\kappa_x:K^\times\rightarrow k(x)^\times$ is not of the form 
	\begin{itemize}
		\item	$\alpha\mapsto \alpha^{-\mathbf{i}}$ or $\alpha\mapsto \alpha^{\mathbf{i}+\mathbf{1}}\lvert \alpha\rvert$ for $\mathbf{i}\in \Z_{\geq 0}^{\Hom(K,k(x))}$ (if $x$ is a characteristic $0$ point), or 
		\item	trivial or $\chi_{\mathrm{cyc}}\circ\Nm_{K/\Q_p}$ (if $x$ is a characteristic $p$ point).
	\end{itemize}

	The space of \emph{regular parameters} $\mathcal{T}_{\mathrm{reg}}^d\subset\mathcal{T}^d$ is the Zariski-open subspace whose $X$-points are given by parameters $\underline\delta:(K^\times)^d\rightarrow \O(X)^\times$ such that $\delta_i\delta_j^{-1}:K^\times\rightarrow \O(X)^\times$ is regular for all $j>i$.
        \label{def: regular parameters}
\end{definition}

Consider the functor $\mathcal{S}_d^\square$ on pseudorigid spaces defined via
\[      X\rightsquigarrow \{(D,\Fil^\bullet D,\underline\delta,\underline\nu)\}/\sim    \]
where $D$ is a trianguline $(\varphi,\Gamma_K)$-module with filtration $\Fil^\bullet D$ and regular parameter $\underline\delta\in\mathcal{T}_{\mathrm{reg}}^d$, and $\underline\nu$ is a sequence of trivializations $\nu_i:\gr^iD\xrightarrow{\sim}\Lambda_{X,\rig,K}$.  There is a natural transformation $\mathcal{S}_d^\square\rightarrow\mathcal{T}_{\mathrm{reg}}^d$ given on $X$-points by
\[      (D,\Fil^\bullet D,\underline\delta,\underline\nu)\rightsquigarrow \underline\delta      \]
Exactly as in ~\cite[Th\'eor\`eme 3.3]{chenevier2013} and ~\cite[Theorem 2.4]{hellmann-schraen2016}, we have the following:
\begin{prop}\label{prop: dim S_d}
        The functor $\mathcal{S}_d^\square$ is representable by a pseudorigid space, which we also denote $\mathcal{S}_d^\square$, and the morphism $\mathcal{S}_d^\square\rightarrow \mathcal{T}_{\mathrm{reg}}^d$ is smooth of relative dimension $\frac{d(d-1)}{2}[K:\Q_p]$.
\end{prop}
One proves by induction on $d$ that if $D$ is a trianguline $(\varphi,\Gamma_K)$-module over $X$ with parameter $\underline\delta\in (\mathcal{T}_{\mathrm{reg}})^d$, then $H_{\varphi,\Gamma_K}^1(D)$ is a vector bundle over $X$ of rank $d[K:\Q_p]$ (the regularity assumption ensures that $H_{\varphi,\Gamma_K}^0(D)=H_{\varphi,\Gamma_K}^2(D)=0$).  Now $\mathcal{S}_1^\square=\mathcal{T}=\mathcal{T}_{\mathrm{reg}}^1$, so $\mathcal{S}_1^\square$ is representable and is smooth of the correct dimension over $\mathcal{T}_{\mathrm{reg}}^1$.  Then one may proceed by induction on $d$ again, and construct $\mathcal{S}_d^\square$ as the moduli space of extensions of the universal $(\varphi,\Gamma_K)$-module of character type $\Lambda_{\mathcal{T},\rig,K}(\delta_{\mathrm{univ}})$ by the universal object $D_{d-1,\mathrm{univ}}$ over $\mathcal{S}_{d-1}^\square$.  For a specified regular parameter $\underline\delta=(\delta_1,\ldots,\delta_d)\in\mathcal{T}_{\mathrm{reg}}^d(X)$, the fiber $\mathcal{S}_d^\square|_{\underline\delta}$ is equal to $\Ext^1(\Lambda_{X,\rig,K}(\delta_d),D_{d-1,\mathrm{univ}}|_{(\delta_1,\ldots,\delta_{d-1})})=H_{\varphi,\Gamma_K}^1(D_{d-1,\mathrm{univ}}|_{(\delta_1,\ldots,\delta_{d-1})}(\delta_d^{-1}))$.  This is a rank-$(d-1)$ vector bundle over $X$, and the claim follows.

We also introduce variants of $\mathcal{S}_d^\square$ with families of fixed determinant and weights.  More precisely, suppose $X$ is a pseudorigid space and we have a continuous character $\delta_{\det}:K^\times\rightarrow\O(X)^\times$ and a $d$-tuple of continuous characters $\underline\kappa:=(\kappa_1,\ldots,\kappa_d):\O_K^\times\rightarrow \O(X)^\times$.  We say that $\delta_{\det}$ and $\underline\kappa$ are \emph{compatible} if $\delta_{\det}|_{\O_K^\times}=\kappa_1\cdots\kappa_d$.  If $\delta_{\det}$ and $\underline\kappa$ are compatible, we consider the functors $\mathcal{S}_d^{\square,\delta_{\det}}$ and $\mathcal{S}_d^{\square,\delta_{\det},\underline\kappa}$ on pseudorigid spaces over $X$ defined via
	\[	Y\rightsquigarrow \{(D,\Fil^\bullet D,\underline\delta)\in \mathcal{S}_d^\square(Y) \mid \delta_1\cdots\delta_d=\delta_{\det}\}/\sim	\]
	and
	\[	Y\rightsquigarrow \{(D,\Fil^\bullet D,\underline\delta,\underline\nu)\in \mathcal{S}_d^\square(Y) \mid \delta_i|_{\O_K^\times}=\kappa_i\text{ for all }i, \delta_1\cdots\delta_d=\delta_{\det}\}/\sim	\]
	\begin{prop}\label{prop: dim S_d det wts}
		The functor $\mathcal{S}_d^{\square,\delta_{\det},\underline\kappa}$ is representable by a pseudorigid space over $X$, which we also denote $\mathcal{S}_d^{\square,\delta_{\det},\underline\kappa}$, and the morphism $\mathcal{S}_d^{\square,\delta_{\det},\underline\kappa}\rightarrow X$ is smooth and surjective of relative dimension $\frac{d(d-1)}{2}[K:\Q_p]+d-1$.
	\end{prop}
	\begin{proof}
		Set $Y:=\widehat{(\O_K^\times)^d}^{\an}$.  Then there is a morphism $\mathcal{T}^d\rightarrow \G_{m,Y}$, given by $\underline\delta\mapsto \left(\delta_1|_{\O_K^\times},\ldots,\delta_d|_{\O_K^\times},\delta_1(\varpi_K)\cdots\delta_d(\varpi_K)\right)$, and it is smooth of relative dimension $d-1$.  The choice of $\delta_{\det}$ and $\underline\kappa$ define a morphism $X\rightarrow \G_{m,Y}$, and we have a pullback square
		\[
			\begin{tikzcd}
				\mathcal{S}_d^{\square,\delta_{\det},\underline\kappa} \ar[r] \ar[d] & \mathcal{S}_d^\square \ar[d] \\
				X \ar[r] & \G_{m,Y}
			\end{tikzcd}
		\]
		Then the result follows from Proposition~\ref{prop: dim S_d}.
	\end{proof}

	\begin{example}\label{example: tri var over wt space}
		In the example of most interest to us, we will take $K=\Q_p$, $d=2$, and $R=\Z_p[\![T_0]\!]$, where $T_0:=\T(\Z_p)$ for a split maximal torus $\T\subset \GL_2/\Z_p$.  Fix an unramified character $\psi_0:\Gal_{\Q_p}\rightarrow R^\times$.  There is a universal pair of characters $\lambda_1,\lambda_2:\Z_p^\times\rightrightarrows R^\times$, and we set $\psi:=\left( \lambda_1\lambda_2 \chi_{\cyc}\right)^{-1}\psi_0$ and $\underline\kappa:(\lambda_2^{-1},\left(\lambda_1\chi_{\cyc}\right)^{-1})$.  Then the morphism $\mathcal{S}_2^{\square}\rightarrow \Spa R^{\an}$ is the natural projection $\mathcal{S}_d^{\square}\rightarrow \widehat{(\Z_p^\times)^2}$, composed with taking inverses and swapping factors.  Furthermore, $\mathcal{T}$ is $2$-dimensional and irreducible (corresponding to a choice of $\delta_1$); fixing the determinant means the remaining degrees of freedom are the $1$-dimensional irreducible space $\widehat{\Z_p^\times}$ (corresponding to the choice of $\delta_2|_{\Z_p^\times}$), and the generically $1$-dimensional space of extensions between them.  We see that in this case, $\mathcal{S}_2^{\square,\delta_{\psi},\underline\kappa}$ is $4$-dimensional, and an $\A^1$-torsor over a dense open subspace of $\G_m^{\ad}\times\widehat{(\Z_p^\times)^2}$.  Hence it is irreducible.
	\end{example}

	\subsection{Structure of trianguline varieties}

	Let $K/\Q_p$ be a finite extension, and let $\overline\rho:\Gal_K\rightarrow \GL_d(k)$ be a continuous representation, where $k$ is a finite field containing the residue field of $K$.  The fiber product $(\Spa R_{\overline\rho}^{\square})^{\an}\times_{\Spa \Z_p}\mathcal{T}^d$ exists as a pseudorigid space, and it is contained in the fiber product
	\[	\G_m^{\ad,d}\times_\Z \widehat{(\O_K^\times)^d} \times \Spa(R_{\overline\rho}^{\square})^{\an}	\]
	(with complement of codimension $\geq 2$ if $d\geq 2$).  Let $X_{{\tri},\overline\rho}^{\square}$ be the Zariski closure in the latter of the set of maximal points $x=\{(\rho_x,\underline\delta_x)\}$, where $\rho_x$ is a (framed) lift of $\overline\rho$ and $\underline\delta_x\in \mathcal{T}_{\mathrm{reg}}^d(L)$ is a regular parameter of $D_{\rig}(\rho_x)$.

	Let $R$ be a complete local noetherian $\Z_p$-algebra with finite residue field.  Fix an $d$-tuple of characters $\underline\kappa:=(\kappa_1,\ldots,\kappa_d)$, where $\kappa_i:\O_K^\times\rightarrow \O(X)^\times$ and $X:=\left(\Spa R\right)^{\an}$, and fix a character $\psi:\Gal_K\rightarrow R^\times$.   Over the pseudorigid space $X$, a character $\psi:\Gal_K\rightarrow \O(X)^\times$ corresponds to a rank-$1$ $(\varphi,\Gamma)$-module of the form $D_{\rig}(\delta_{\psi})$, for some character $\delta_{\psi}:K^\times\rightarrow \O(X)^\times$.  If $\delta_{\psi}$ and $\underline\kappa$ are compatible, we may define 
	\[	X_{ {\tri},\overline\rho}^{\square,\psi,\underline\kappa}\subset \G_m^{\ad,d}\times_{\Z}\widehat{(\O_K^\times)^d}\times (\Spa R_{\overline\rho}^{\square,\psi})^{\an}	\]
	to be the Zariski closure of the set of maximal points $x=\{(\rho_x,\underline\delta_x)\}$, where $\rho_x$ is a framed lift of $\overline\rho$ with determinant $\psi$ and $\underline\delta_x\in \mathcal{T}_{\mathrm{reg}}^d(L)$ is a regular parameter of $D_{\rig}(\rho_x)$ such that $\delta_i|_{\O_K^\times}=\kappa_i$.

	In order to study the structure of $X_{ {\tri},\overline\rho}^{\square}$ and $X_{ {\tri},\overline\rho}^{\square,\psi,\underline\kappa}$, we will need to know something about the essential image of the functor from Galois representations to $(\varphi,\Gamma)$-modules.  We refer the reader to ~\cite{bellovin2020} for details on definitions of pseudorigid overconvergent period rings and the construction of $(\varphi,\Gamma)$-modules in the pseudorigid setting.  However, we note here that $\Lambda_{R,[0,b],K}$ is the coordinate ring of a closed annulus over $\Spa R$, $\Lambda_{R,(0,b],K}$ is the ring of global functions on a half-open annulus over $\Spa R$, and $\Lambda_{R,{\rig},K}:=\varprojlim_{b\rightarrow 0}\Lambda_{R,(0,b],K}$.  As in the work of ~\cite{cc} and ~\cite{berger-colmez}, $(\varphi,\Gamma)$-modules attached to Galois representations are constructed over $\Lambda_{R,[0,b],K}$ for some $b>0$ (which depends in subtle ways on the representation).
\begin{lemma}
	The functor $M\rightsquigarrow D_{\rig,K}(M)$ from $\Gal_K$-representations to their associated $(\varphi,\Gamma)$-modules is formally smooth.
	\label{lemma: phi-gamma formally smooth}
\end{lemma}
\begin{proof}
	We need to show that if $D$ is a projective $(\varphi,\Gamma_K)$-module over a pseudoaffinoid algebra $R'$, and $I\subset R'$ is a square-zero ideal such that $(R'/I)\otimes_{R'}D$ arises from a family of Galois representations, then $D$ also arises from a family of Galois representations.  Indeed, we have a short exact sequence
	\[      0\rightarrow ID\rightarrow D\rightarrow (R'/I)\otimes_{R'}D\rightarrow 0    \]
	By assumption, $D':=(R'/I)\otimes_{R'}D$ arises from a family of $\Gal_K$ representations $M'$ over $R'/I$, and since 
	\[	D'':=ID \cong I\otimes_{R'} D	\cong (R'/\ann_{R'}I)\otimes_{R'/I}D	\]
	it arises from a family of $\Gal_K$ representations $M''$ over $R'/\ann_{R'}I$.  Since $D$ has a model $D_b$ over $\Lambda_{R',(0,b],K}$, we have a commutative diagram
	\[
                \begin{tikzcd}
			0 \ar[r] & \widetilde\Lambda_{R',(0,b/p]}\otimes_{R'}D'' \ar[r] & \widetilde\Lambda_{R',(0,b/p]}\otimes_{R'}D \ar[r] & \widetilde\Lambda_{R',(0,b/p]}\otimes_{R'}D' \ar[r] & 0      \\      
			0 \arrow{r} & \widetilde\Lambda_{R',(0,b]}\otimes_{R'}D'' \arrow{r}\arrow{u}{\varphi-1} & \widetilde\Lambda_{R',(0,b]}\otimes_{R'}D \arrow{r}\arrow{u}{\varphi-1} & \widetilde\Lambda_{R',(0,b]}\otimes_{R'}D' \arrow{r}\arrow{u}{\varphi-1} & 0
                \end{tikzcd}
        \]
	By construction, $\widetilde\Lambda_{R',(0,b]}\otimes_{R'}D''\cong \widetilde\Lambda_{R',(0,b]}\otimes \left(\widetilde\Lambda_{R_0',[0,b]}\otimes_{R_0'}M_0''\right)$ and $\widetilde\Lambda_{R',(0,b]}\otimes_{R'}D'\cong \widetilde\Lambda_{R',(0,b]}\otimes \left(\widetilde\Lambda_{R_0',[0,b]}\otimes_{R_0'}M_0'\right)$, for some integral models $M_0''$ and $M_0'$ (perhaps after localizing on $\Spa R'$ and shrinking $b$).  Therefore, we have quasi-isomorphisms
	\begin{align*}
		[M'']&\xrightarrow\sim [\widetilde\Lambda_{R',[0,b]}\otimes_{R_0'}M_0''\xrightarrow{\varphi-1}\widetilde\Lambda_{R',[0,b/p]}\otimes_{R_0'}M_0'']	\\
		&\xrightarrow\sim [\widetilde\Lambda_{R',(0,b]}\otimes_{R'}D'' \xrightarrow{\varphi-1}\widetilde\Lambda_{R',(0,b/p]}\otimes_{R'}D'']
	\end{align*}
	and
	\begin{align*}
	[M']&\xrightarrow\sim [\widetilde\Lambda_{R',[0,b]}\otimes_{R_0'}M_0'\xrightarrow{\varphi-1}\widetilde\Lambda_{R',[0,b/p]}\otimes_{R_0'}M_0']	\\
	&\xrightarrow\sim [\widetilde\Lambda_{R',(0,b]}\otimes_{R'}D'\xrightarrow{\varphi-1}\widetilde\Lambda_{R',(0,b/p]}\otimes_{R'}D']
	\end{align*}
	Then the snake lemma implies that we have an exact sequence
	\[      0\rightarrow M''\rightarrow \left(\widetilde\Lambda_{R',\rig,K}\otimes D\right)^{\varphi=1}\rightarrow M'\rightarrow 0   \]
	of $R'$-modules equipped with continuous $R'$-linear actions of $\Gal_K$, with $M'$ finite projective over $R'/I$ and $M''\cong(R'/\ann_{R'}I)\otimes_{R'/I}M'$.  It follows that $M:=\left(\widetilde\Lambda_{R',\rig,K}\otimes D\right)^{\varphi=1}$ is a projective $R'$-module of the same rank and $D_{\rig,K}(M)=D$.
\end{proof}

In ~\cite[\textsection 2.2]{breuil-hellmann-schraen}, the authors show that the characteristic $0$ locus $X_{ {\tri},\overline\rho}^{\square,\rig}$ of the trianguline variety is equidimensional of dimension $d^2+[K:\Q_p]\frac{d(d+1)}{2}$, and generically smooth.  We note that if $\psi:\Gal_K\rightarrow \O_E^\times$ is a crystalline character, where $E/\Q_p$ is a finite extension and $\O_E$ is its ring of integers, then an identical argument shows that the rigid analytic locus $X_{ {\tri},\overline\rho}^{\square,\psi,\rig}\subset X_{ {\tri},\overline\rho}^{\square,\psi}$ is equidimensional of dimension $d^2-1+[K:\Q_p]\frac{(d+2)(d-1)}{2}$ (indeed, ~\cite[Theorem 1.2]{boeckle-iyengar-paskunas2022} provides the necessary crystalline lifts with fixed determinant).

Unfortunately, we cannot rule out components of $X_{ {\tri},\overline\rho}^{\square}$ or $X_{ {\tri},\overline\rho}^{\square,\psi}$ supported entirely in characteristic $p$, and so to deduce the same result in the pseudorigid setting, we need to repeat a large part of the argument in a neighborhood of the characteristic $p$ fiber.
\begin{prop}\label{prop: xtri equidim}
	\begin{enumerate}
		\item	The space $X_{ {\tri},\overline\rho}^{\square}$ (equipped with its underlying reduced structure) is equidimensional of dimension $d^2+[K:\Q_p]\frac{d(d+1)}{2}$.
		\item	If $X_{ {\tri},\overline\rho}^{\square,\psi,\underline\kappa}$ is non-empty, it is equidimensional of dimension $d^2-1+[K:\Q_p]\frac{d(d-1)}{2}+\dim \Spa R^{\an}$.
		\item	There is an open subspace $Z\subset\Spa R^{\an}$ such that  morphism $X_{ {\tri},\overline\rho}^{\square,\psi,\underline\kappa}|_Z\rightarrow Z$ is equidimensional of dimension $d^2-1+[K:\Q_p]\frac{d(d-1)}{2}$.
	\end{enumerate}
\end{prop}
\begin{proof}
		The proofs of the first two parts are very similar to that of ~\cite[Th\'eor\`eme 2.6]{breuil-hellmann-schraen}, and we will prove them simultaneously.  By construction, there is a universal framed deformation $\rho_{\mathrm{univ}}:\Gal_K\rightarrow\GL_d(R_{\overline\rho}^\square)$ of $\overline\rho$, and we may pull it back to $X_{ {\tri},\overline\rho}^\square$ (resp. $X_{ {\tri},\overline\rho}^{\square,\psi,\underline\kappa}$).  Then for any irreducible open affinoid $X\subset X\rightarrow X_{ {\tri},\overline\rho}^{\square}$ (resp. $X_{ {\tri},\overline\rho}^{\square,\psi,\underline\kappa}$), by ~\cite[Corollary 5.10]{bellovin2021} there is a sequence of blow-ups and normalizations $f:\widetilde X\rightarrow X$ and an open subspace $U\subset \widetilde X$ containing the characteristic $p$ locus such that $f^\ast \rho_{\mathrm{univ}}|_U$ is trianguline with parameters $f^\ast\underline\delta$.  Shrinking $U$ if necessary, we may assume that $f^\ast\underline\delta$ is regular (indeed, the pre-image of $\mathcal{T}_{\mathrm{reg}}^d$ in $U$ is open, and by construction $U$ contains a Zariski dense set of points corresponding to trianguline representations with regular parameters).  Furthermore, there is a Zariski-dense and open subspace $V\subset X_{ {\tri},\overline\rho}^{\square}$ (resp. $X_{ {\tri},\overline\rho}^{\square,\psi,\underline\kappa}$) such that $f^{-1}(V)\subset U$ and $f$ defines an isomorphism $f^{-1}(V)\xrightarrow\sim V$.

			Over $U$, the $(\varphi,\Gamma_K)$-module $D:=D_{\rig,K}(f^\ast \rho_{\mathrm{univ}})$ is equipped with an increasing filtration $\Fil^\bullet D$ such that $\gr^iD\cong \Lambda_{U,\rig,K}(f^\ast\delta_i)\otimes \mathscr{L}_i$ for some line bundle $\mathscr{L}_i$ on $U$.  We may therefore construct a $\G_{m,U}^d$-torsor $U^\square\rightarrow U$ trivializing each of the $\mathscr{L}_i$; since $U^\square$ carries the data $(D,\Fil^\bullet D, f^\ast\underline\delta, \underline\nu)$, where $\underline\nu$ is the set of trivializations $\nu_i:\gr^iD\xrightarrow\sim  \Lambda_{U,\rig,K}(f^\ast\delta_i)$, there is a morphism $U^\square\rightarrow \mathcal{S}_d^\square$.

			Let $V^\square\subset U^\square$ denote the pullback of $U^\square\rightarrow U$ to $V$.  We claim that $V^\square\rightarrow \mathcal{S}_d^\square$ is smooth of relative dimension $d^2$.  To see this, suppose we have a pseudoaffinoid algebra $R'$, a morphism $\Spa R'\rightarrow \mathcal{S}_d^\square$, and a square-zero ideal $I\subset R'$ such that the composition $\Spa R'/I\hookrightarrow \Spa R'\rightarrow \mathcal{S}_d^\square$ is in the image of $V^\square$.  Then there is a ring of definition $R_0'\subset R'/I$ such that the homomorphism $R_{\overline\rho}^\square\rightarrow R'/I$ factors through $R_0'$; we let $M_0'\cong R_0'^{\oplus d}$ be the pullback of the universal framed deformation to $R_0'$ and we let $M':=R'/I\otimes_{R_0'}M_0'$.
			
			By Lemma~\ref{lemma: phi-gamma formally smooth}, there is a $\Gal_K$-representation $M$ over $R'$ such that $(R'/I)\otimes_{R'}M\xrightarrow{\sim}M'$.  It follows that $M_0'$ and its basis lift to a free module $M_0$ over some ring of definition $R_0'\subset R'$, such that $R'\otimes_{R_0'}M_0=M$.  Moreover, $M'$ is residually a lift of $\overline\rho$ at every maximal point of $\Spa R'$, so $M$ is as well.  
			By~\cite[Theorem 3.8]{wang-erickson}, $M_0$ corresponds to a $\Spa R_0'$-point of $\Spf R_{\overline\rho}^\square$, and by construction $M$ corresponds to a $\Spa R'$-point of $X_{ {\tri},\overline\rho}^{\square}$ deforming $M'$.  Since $M'$ corresponds to a $\Spa (R'/I)$-point of the Zariski-open subspace $V\subset X_{ {\tri},\overline\rho}^{\square}$, the image of the morphism $\Spa R'\rightarrow X_{ {\tri},\overline\rho}^{\square}$ also lands in $V$.  Since $D$ is trianguline with regular parameters and trivialized quotients, the morphism $\Spa R'\rightarrow V$ lifts to a morphism $\Spa R'\rightarrow V^\square$.

			The claim that $V^\square\rightarrow \mathcal{S}_d^\square$ has relative dimension $d^2$ follows because ``changing the framing'' makes $V^\square$ a $\left(\GL_d\right)^{\an}$-torsor over its image in $\mathcal{S}_d^\square$.

			Now we can compute the dimension.  By Proposition~\ref{prop: dim S_d}, we see that $V^\square$ is equidimensional of dimension $d^2 + \frac{d(d-1)}{2}[K:\Q_p] + d[K:\Q_p] + d$ (resp. $d^2 + \frac{d(d-1)}{2}[K:\Q_p] + d[K:\Q_p] + d + \dim \Spa R^{\an}$).  Since $V^\square\rightarrow V$ is a $\G_{m,V}^d$-torsor, it follows that $V$ is equidimensional of dimension $d^2+\frac{d(d+1)}{2}[K:\Q_p]$ (resp.  $d^2 + \frac{d(d-1)}{2}[K:\Q_p] + d[K:\Q_p] +\dim \Spa R^{\an}$.  Finally, $V\subset X$ is Zariski-dense, so we are done.

		For the last part, we define $V^{\square,\psi,\underline\kappa}$ via the pullback
			\[
				\begin{tikzcd}
					V^{\square,\psi,\underline\kappa} \ar[r]\ar[d] & V^\square \ar[d]	\\
					\mathcal{S}_d^{\square,\delta_\psi,\underline\kappa} \ar[r]\ar[d] & \mathcal{S}_d^\square \ar[d]	\\
					\Spa R^{\an} \ar[r] & \G_{m,Y}
				\end{tikzcd}
			\]
			where $Y:=\widehat{(\O_K^\times)^d}$ and the morphism $\Spa R^{\an}\rightarrow\G_{m,Y}$ is given by $\underline\kappa$ and $\delta_\psi$.  Since $V^\square\rightarrow \G_{m,Y}$ is smooth, its image is open, and the pre-image in $\Spa R^{\an}$ is open, as well.
\end{proof}

\begin{remark}
Suppose that $x\in\Spa R^{\an}$ is a maximal point such that the fiber of $X_{ {\tri},\overline\rho}^{\square,\psi,\underline\kappa}$ contains a point $(\rho_x,\underline\delta_x)$ such that $\underline\delta_x$ is a regular parameter for $D_{\rig}(\rho_x)$.  Then if we apply Proposition~\ref{prop: xtri equidim} with $R=k(x)^+$, we see that every irreducible component of the fiber containing $(\rho_x,\underline\delta_x)$ has dimension $d^2 + \frac{d(d-1)}{2}[K:\Q_p] + d[K:\Q_p]$.
\end{remark}

\begin{example}\label{example: tri var over wt space residual representation}
	We return to the setting of Example ~\ref{example: tri var over wt space}, where $K=\Q_p$, $d=2$, $R=\Z_p[\![T_0]\!]$ corresponds to integral weight space for a split maximal torus of $\GL_2/\Z_p$, $\psi_0:\Gal_{\Q_p}\rightarrow R^\times$ is an unramified character, and there is a universal pair of characters $\lambda_1,\lambda_2:\Z_p^\times\rightrightarrows R^\times$. We again set $\psi:=\psi_0\left( \lambda_1\lambda_2 \chi_{\cyc}\right)^{-1}$ and $\underline\kappa:(\lambda_2^{-1},\left(\lambda_1\chi_{\cyc}\right)^{-1})$.  Then if $X_{ {\tri},\overline\rho}^{\square,\psi,\underline\kappa}$ is non-empty, each irreducible component is $6$-dimensional.  

	Moreover, suppose there is a characteristic-$p$ point $(\rho_x,\underline\delta_x)$ with specified weight and determinant, such that $\rho_x$ is trianguline with regular parameter $\delta_x$.  Then the fiber over $\underline\delta_x|_{(\Z_p^\times)^2}$ is $4$-dimensional; since this is one of $p-1$ disjoint characteristic-$p$ fibers, we see that the irreducible component containing $(\rho_x,\underline\delta_x)$ contains a dense open characteristic-$0$ subspace, consisting of points in $U_{ {\tri}}^\square(\overline\rho)^{\mathrm{reg}}$ (in the notation of ~\cite[D\'efinition 2.4]{breuil-hellmann-schraen}).
\end{example}

Now we consider a global setup.  Let $F$ be a number field, and suppose that $\overline\rho:\Gal_{F}\rightarrow\GL_d(\F)$ is an absolutely irreducible representation, unramified outside a finite set of primes $S$.

Then the homomorphisms
\[	R_{\overline\rho_v}^\square\rightarrow R_{\overline\rho,S}^\square	\]
for each $v\mid p$ induce a morphism
\[	\left(\Spa R_{\overline\rho,S}^\square\right)^{\an}\times \prod_{v\mid p}\mathcal{T}^d\rightarrow \prod_{v\mid p}\left(\left(\Spa R_{\overline\rho_v}^\square\right)^{\an}\times\mathcal{T}^d\right)	\]
and we define $X_{ {\tri},\overline\rho,S}^\square$ to be the pre-image of $\prod_{v\mid p}X_{ {\tri},\overline\rho_v}^\square$.

If $R$ is a complete local noetherian $\Z_p$-algebra with maximal ideal $\mathfrak{m}_R$ and finite residue field, and $\psi:\Gal_F\rightarrow R^\times$ is a continuous character such that $\det\overline\rho = \psi\mod\mathfrak{m}_R$, the homomorphisms
\[	R_{\overline\rho_v}^{\square,\psi_v}\rightarrow R_{\overline\rho,S}^{\square,\psi}	\]
and
\[	R_{\overline\rho,\loc}^{\square,\psi}\rightarrow R_{\overline\rho,S}^{\square,\psi}	\]
induce a sequence of morphisms
\begin{equation*}
	\resizebox{\displaywidth}{!}{
	\begin{tikzcd}[ampersand replacement=\&]
		\left(\Spa R_{\overline\rho,S}^{\square,\psi}\right)^{\an}\times \prod_{v\mid p}\mathcal{T}^d\ar[r] \& \left(\Spa R_{\overline\rho,\loc}^{\square,\psi}\right)^{\an}\times\prod_{v\mid p}\mathcal{T}^d \arrow[r,hook] \& \prod_{v\mid p}\left(\left(\Spa R_{\overline\rho_v}^{\square,\psi_v}\right)^{\an}\times\mathcal{T}^d\right)      
	\end{tikzcd}
}
\end{equation*}
where $\psi_v:=\psi|_{\Gal_{F_v}}$.  We define $X_{ {\tri},\overline\rho,S}^{\square,\psi}$ and $X_{ {\tri},\overline\rho,\loc}^{\square,\psi}$ to be the pre-images of $\prod_{v\mid p}X_{ {\tri},\overline\rho_v}^{\square,\psi_v}$ in $\left(\Spa R_{\overline\rho,S}^{\square,\psi}\right)^{\an}\times \prod_{v\mid p}\mathcal{T}^d$ and $\left(\Spa R_{\overline\rho,\loc}^{\square,\psi}\right)^{\an}\times\prod_{v\mid p}\mathcal{T}^d$, respectively.

If we additionally have $d$-tuples of characters $\underline\kappa_v:=(\kappa_{v,1},\ldots,\kappa_{v,d})$, where $\kappa_{v,i}:\O_{F_v}^\times\rightarrow\O(X)^\times$ is a continuous character, and we set $X:=\left(\Spa R\right)^{\an}$, we may form the spaces
\begin{equation*}
	\resizebox{\displaywidth}{!}{
	\begin{tikzcd}[ampersand replacement=\&]
		X_{ {\tri},\overline\rho,S}^{\square,\psi,\underline\kappa} \ar[r] \arrow[d,symbol=\subset] \& X_{ {\tri},\overline\rho,\loc}^{\square,\psi,\underline\kappa} \arrow[r,hook] \arrow[d,symbol=\subset] \& \prod_{v\mid p}X_{ {\tri},\overline\rho_v}^{\square,\psi_v,\underline\kappa_v} \arrow[d,symbol=\subset]	\\
		\left(\Spa R_{\overline\rho,S}^{\square,\psi}\right)^{\an}\times \prod_{v\mid p}\mathcal{T}^d \ar[r] \& \left(\Spa R_{\overline\rho,\loc}^{\square,\psi}\right)^{\an}\times\prod_{v\mid p}\mathcal{T}^d \arrow[r,hook] \& \prod_{v\mid p}\left(\left(\Spa R_{\overline\rho_v}^{\square,\psi_v}\right)^{\an}\times\mathcal{T}^d\right)
\end{tikzcd}}
\end{equation*}

In particular, suppose we have fixed a neat level $K=K^pI$, as in sections~\ref{section: extended eigenvarieties} and ~\ref{section: modular forms}, and consider the ring $R=\Z_p[\![T_0/\overline{Z(K)}]\!]$ corresponding to integral weight space.  Since $T_0=\prod_{v\mid p}(\Res_{\O_{F_v}/\Z_p}T_v)(\Z_p)$, we have homomorphisms $\Z_p[\![T_v(\O_{F_v})]\!]\rightarrow R$, and hence morphisms $\Spa R\rightarrow \Spa \Z_p[\![T_v(\O_{F_v})]\!]$.  Suppose we have a determinant character $\psi:\Gal_F\rightarrow R^\times$ and a set of weights $\underline\kappa_v:=(\kappa_{v,1},\ldots,\kappa_{v,d}):\O_{F_v}^\times\rightarrow \O(\mathscr{W}_F)^\times$ for each $v\mid p$, such that $\psi|_{\Gal_{F_v}}$ and $\underline\kappa_v$ are compatible for all $v$, and such that $\psi_v$ and $\underline\kappa_v$ factor through $\Z_p[\![T_v(\O_{F_v})]\!]\rightarrow R$ for all $v$, i.e., they depend only on the projection to $\Spa \Z_p[\![T_v(\O_{F_v})]\!]$.
\begin{prop}\label{prop: rel dim loc tri}
	Under the above assumptions, there is an open subspace $Z\subset \mathscr{W}_F$ such that $X_{ {\tri},\overline\rho,\loc}^{\square,\psi,\underline\kappa}|_Z\rightarrow Z$ is equidimensional of dimension $\lvert\Sigma_p\rvert(d^2-1)+[F:\Q]\frac{d(d-1)}{2}$. 
\end{prop}
\begin{proof}
	Viewing $\psi_v$ as a character $\Gal_{F_v}\rightarrow \Z_p[\![T_v(\O_{F_v})]\!]^\times$ and viewing $\underline\kappa_v=(\kappa_{v,1},\ldots,\kappa_{v,d})$ as a $d$-tuple of characters $\O_{F_v}^\times\rightarrow \Z_p[\![T_v(\O_{F_v})]\!]^\times$, we have a pullback diagram
	\[
		\begin{tikzcd}
			X_{ {\tri},\overline\rho,\loc}^{\square,\psi,\underline\kappa} \ar[r]\ar[d] & \prod_{v\mid p}X_{ {\tri},\overline\rho_v}^{\square,\psi_v,\underline\kappa_v} \ar[d]	\\
			\mathscr{W}_F \ar[r] & \prod_{v\mid p}\left(\Spa \Z_p[\![T_v(\O_{F_v})]\!]\right)^{\an}
		\end{tikzcd}
	\]
	The right vertical morphism has relative dimension 
	\[	\sum_{v\mid p}\left(d^2-1 + [F_v:\Q_p]\frac{d(d-1)}{2}\right) = \lvert\Sigma_p\rvert(d^2-1)+[F:\Q]\frac{d(d-1)}{2}	\]
	over an open subspace of $\prod_{v\mid p}\left(\Spa \Z_p[\![T_v(\O_{F_v})]\!]\right)^{\an}$, so the morphism $X_{ {\tri},\overline\rho,\loc}^{\square,\psi,\underline\kappa}\rightarrow \mathscr{W}_F$ does, as well. 
\end{proof}

The case we will be most interested in is the case where $F/\Q$ is cyclic and totally split at $p$, and $d=2$.  In that case, $X_{ {\tri},\overline\rho_v}^{\square,\psi_v,\underline\kappa_v}\rightarrow \Spa \Z_p[\![T_v(\O_{F_v})]\!]^{\an}$ has relative dimension $4$ over an open subspace of $\Spa \Z_p[\![T_v(\O_{F_v})]\!]^{\an}$ for each $v\mid p$, and hence $X_{ {\tri},\overline\rho,\loc}^{\square,\psi,\underline\kappa}\rightarrow \mathscr{W}_F$ has relative dimension $4[F:\Q]$ over an open subspace of $\mathscr{W}_F$.

\subsection{Trianguline deformation rings}

We have constructed the trianguline varieties $X_{ {\tri},\overline\rho}^\square$ and $X_{ {\tri},\overline\rho}^{\square,\psi,\underline\kappa}$ as subspaces of the (non-quasicompact) pseudorigid space $\G_m^{\ad,d}\times_{\Z}\widehat{(\O_K^\times)^d}\times\left(\Spa R_{\overline\rho}^\square\right)^{\an}$.  However, the advantage of working with general pseudorigid spaces is that we can construct integral models, so long  as we bound the slope.

We will apply this to find formal models for pieces of our trianguline varieties.
Recall that when $K$ is a finite extension of $\Q_p$ and $\overline\rho$ is a representation of $\Gal_K$, we defined $X_{ {\tri},\overline\rho}^\square$ and $X_{{\tri},\overline\rho}^{\square,\psi,\underline\kappa}$ as analytic subspaces of $\G_m^{\ad}\times\Spa\Z_p[\![(\O_K^\times)^d]\!]\times\left(\Spa R_{\overline\rho}^{\square}\right)^{\an}$ and $\G_m^{\ad}\times\Spa\Z_p[\![(\O_K^\times)^d]\!]\times\left(\Spa R\htimes R_{\overline\rho}^{\square}\right)^{\an}$, respectively.  By construction, $\Spa\Z_p[\![(\O_K^\times)^d]\!]\times\left(\Spa R_{\overline\rho}^{\square}\right)^{\an}$ has a quasicompact integral model, but $\G_m^{\ad}\times\Spa\Z_p[\![(\O_K^\times)^d]\!]\times\left(\Spa R_{\overline\rho}^{\square}\right)^{\an}\times\left(\Spa R_{\overline\rho}^{\square}\right)^{\an}$ does not; in particular, it is not equal to the analytic locus of $\Spa R_{\overline\rho}^{\square}\htimes\Z_p[\![(\O_K^\times)^d]\!]\left\langle T,T^{-1}\right\rangle$ (and similarly for $\G_m^{\ad}\times \Spa\Z_p[\![(\O_K^\times)^d]\!]\times\left(\Spa R\htimes R_{\overline\rho}^{\square}\right)^{\an}$).

In order to construct integral models of annuli, we begin with an illustrative example.
\begin{example}

	Suppose $R=\Z_p[\![u]\!]$ and $h$ is an integer.  We may cover $Y:=\Spa\left(\Z_p[\![u]\!]\right)^{\an}$ with the open affinoid subspaces $U_1:=\Spa\left(\Q_p\left\langle \frac u p\right\rangle\right)$ and $U_2:=\Spa\left(\Z_p[\![u]\!]\left\langle \frac p u\right\rangle\left[\frac 1 u\right]\right)$; their intersection is the circle $U_1\cap U_2=\Spa\left(\Q_p\left\langle \frac u p, \frac p u\right\rangle\right)$.

The annulus $C_{U_1,h}$ is affinoid, with coordinate ring
		\[	\Q_p\left\langle \frac u p,p^hT,T_2\right\rangle/(TT_2-p^h) = \Q_p\left\langle \frac u p,T,T_1,T_2\right\rangle/(T_1-p^hT,TT_2-p^h)	\]
Restricting to $U_1\cap U_2$, we obtain an affinoid with coordinate ring
\[	\Z_p[\![u]\!]\left\langle \frac u p,\frac p u,T,T_1,T_2\right\rangle\left[\frac 1 u\right]/(T_1-u^h\left(\frac p u\right)^hT,TT_2-u^h\left(\frac p u\right)^h)	\]
Writing $T_1':=\left(\frac u p\right)^hT_1$ and $T_2':=\left(\frac u p\right)^hT_2$, we get
\[	\Z_p[\![u]\!]\left\langle \frac u p,\frac p u,T,T_1',T_2'\right\rangle\left[\frac 1 u\right]/(T_1'-u^hT,TT_2'-u^h)	\]
which is also the restriction of $C_{U_2,h}$ to $U_1\cap U_2$.

The affinoid spaces $C_{U_i,h}$ and $C_{U_1\cap U_2,h}$ have integral models, compatible with glueing.  Thus, we see that $C_{Y,h}$ in this case admits a formal model which lives over the blow-up $\mathrm{Bl}_{(p,u)}R$.

\end{example}

Returning to the general case, we may choose a $\Z_p$-basis for the torsion-free part of $\O_K^\times$ and corresponding coordinates on $\Spa \Z_p[\![\O_K^\times]\!]^{\an}$.  Then we may consider relative annuli over $\Spa \Z_p[\![\widehat{\O_K^\times}]\!]$; as above, these annuli glue to a space $\mathcal{T}_h$ 
which has an integral model $\mathfrak{T}_h$ over the blow-up $\mathrm{Bl}_I\Z_p[\![\O_K^\times]\!]$ (where $I=\cap \mathfrak{m}$ is the intersection of the maximal ideals of $\Z_p[\![\O_K^\times]\!]$).  Similarly, given some integer $d\geq 1$, we may define relative annuli $\mathcal{T}_h^d\subset \mathcal{T}^d$ over $\Spa \Z_p[\![(\O_K^\times)^d]\!]^{\an}$, which have integral models $\mathfrak{T}_h^d$.

Now we may set
\[	X_{{\tri},\overline\rho,\leq h}^\square:= X_{{\tri},\overline\rho}^\square\cap \left(\Spa R_{\overline\rho}^\square\times\mathcal{T}_h^d\right)^{\an}	\]
and
\[	X_{{\tri},\overline\rho,\leq h}^{\square,\psi,\underline\kappa}:= X_{{\tri},\overline\rho}^{\square,\psi,\underline\kappa}\cap \left(\Spa R\htimes R_{\overline\rho}^\square\times \mathcal{T}_h^d\right)^{\an}      \]

When $F$ is a totally real field and $\overline\rho$ is a representation of $\Gal_F$ unramified outside a finite set of places $S$, we may similarly define bounded global trianguline varieties $X_{{\tri},\overline\rho,S,\leq h}^\square$ and $X_{{\tri},\overline\rho,S,\leq h}^{\square,\psi,\underline\kappa}$ as subspaces of $\left(\Spa R_{\overline\rho}^\square\times\prod_{v\mid p}\mathcal{T}_h^d\right)^{\an}$ and $\left(\Spa R\htimes R_{\overline\rho}^\square\times\prod_{v\mid p}\mathcal{T}_h^d\right)^{\an}$, respectively.

\begin{remark}
	We emphasize that these bounded trianguline varieties are not canonical; they depend on a choice of coordinates on $\Spa \Z_p[\![\O_K^\times]\!]^{\an}$.  This may seem strange, but we will use our bounded trianguline varieties to study bounded-slope pieces of eigenvarieties, and the construction of these pieces of eigenvarieties \emph{also} depends on a choice of pseudouniformizer in the coefficients.
\end{remark}
\begin{remark}
	For our purposes, we in fact only need to consider relative annuli and bounded trianguline varieties over some affinoid subspace $U$ of weight space.  Nevertheless, for the sake of completeness we treat them here over all of weight space.
\end{remark}

Now we restrict to the case $K=\Q_p$, and we choose coordinates $z_1,\ldots,z_d$ on each component of $\Spa \Z_p[\![(\Z_p^\times)^{d}]\!]^{\an}$.  Write $h=a/b$, where $a,b$ are non-negative relatively prime integers.  We will construct an integral model of $X_{ {\tri},\overline\rho,\leq h}^{\square}$ using Corollary~\ref{cor: integral model in product}.  

For $z\in \{p,z_1,\ldots,z_d\}$, we get an affinoid $U_z:=\Spa R_z$, where $R_z:=\Z_p[\![(\Z_p^\times)^{d}]\!]\left\langle \frac{p}{z},\frac{z_1}{z},\ldots,\frac{z_d}{z}\right\rangle\left[\frac 1 z\right]$ with ring of integers $R_{z,0}$, inside $\Spa \Z_p[\![(\Z_p^\times)^{d}]\!]^{\an}$.  Then the restriction of $\mathcal{T}_h^d$ to $U_z$ has the presentation
\[	\Spa \Z_p[\![(\Z_p^\times)^{d}]\!]\left\langle \frac{p}{z},\frac{z_1}{z},\ldots,\frac{z_d}{z}, z^aT_1^{\pm b},\ldots, z^aT_d^{\pm b}\right\rangle\left[\frac 1 z\right]	\]
Over this space, there is a $d$-tuple $\delta_1,\ldots,\delta_d:\Q_p^\times\rightrightarrows R_z^\times$ where $\delta_i(p)=T_i$ and $(\delta_i|_{\Z_p^\times})$ is the restriction of the universal character on $(\Z_p^\times)^d$.

Given an affinoid $\Spa R\subset \left(\Spa R_{\overline\rho}^\square\right)^{\an}$ with pseudouniformizer $u\in R$, there is a $(\varphi,\Gamma)$-module $D_R$ of rank $d$ over $\Spa R$.  
To study the bounded trianguline variety, we first study morphisms $D_R\rightarrow \Lambda_{R,\rig,\Q_p}(\delta_d)$.  Equivalently, we consider the twist $D_R(\delta_d^{-1})$ over the (non-quasi-compact) space $\Spa R\times\Spa R_z\left\langle z^hT_d^{\pm 1}\right\rangle$ and consider morphisms $D_R(\delta_d^{-1})\rightarrow \Lambda_{R,\rig,\Q_p}$ to the trivial rank-$1$ $(\varphi,\Gamma)$-module.

We wish to first consider the closure $\overline Z_R$ of 
\[	Z_R:=\{(\rho_x,\delta_{d,x})\mid \text{there is a surjective map }D_R\rightarrow \Lambda_{\kappa(x),\rig,\Q_p}(\delta_{d,x})\}	\]
in $\Spa R_0\htimes R_{z,0}\left\langle z^aT_i^b,T_i'\right\rangle/(T_i^bT_i'-z^a)$.

There is a non-zero morphism at precisely the points $x\in \Spa R\times \Spa R_z$ where $H^0(D_R^\vee(\delta_d)_x)$ is non-vanishing; equivalently (by Tate duality), at precisely the points where $H^2(D_R(\delta_d^{-1}\chi_{\cyc})_x)$ is non-vanishing.  We write $Z_R'$ for the support of $H^2(D_R(\delta_d^{-1}\chi_{\cyc})_x)$.

The closure of $Z_R$ is a priori defined as a subspace of the locus
\[	\{\lvert z\rvert^h\leq \lvert T_i\rvert\leq \lvert z\rvert^h\}\subset \Spa R\times\Spa R_{z,0}\times\G_m 	\]
We first claim that there is a closed subspace of 
\[	\Spa R\times \Spa R_{z,0}\left\langle z^aT_i^b,T_i'\right\rangle/(T_i^bT_i'-z^a)	\]
whose restriction to this subspace is our original closure.

By ~\cite[Proposition 5.2]{bellovin2021} (more precisely, by the proof of the corresponding result ~\cite[Proposition 3.3]{kpx} in characteristic $0$), $H^2(D_R(\delta_d^{-1}\chi_{\cyc}))$ vanishes if $T_d$ is sufficiently $u$-adically small.  Here ``sufficiently small'' depends only on $D_R$, not on the twist by $\delta_d|_{\Z_p^\times}$.  Thus, $Z_R'$ is contained in the locus $\{T_d\geq u^N\}$ for some $N\gg0$. Since $\Spa R_0\htimes R_{z,0}\left\langle z^aT_i^b,T_i'\right\rangle/(T_i^bT_i'-z^a)$ is covered by the loci $\{\lvert T_d\rvert \geq \lvert u^{2N}\rvert\}$ and $\{\lvert T_d\rvert \leq \lvert u^{2N}\rvert\}$, the closure of $Z_R$ is well-behaved in $\Spa R_0\htimes R_{z,0}\left\langle z^aT_i^b,T_i'\right\rangle/(T_i^bT_i'-z^a)$. 

	We next consider the intersection of $Z_R$ with 
	\[      \{T=T_d\}\subset\Spa R\left\langle u^NT^{\pm 1}\right\rangle\times \Spa R_{z,0}\left\langle z^aT_d^{\pm b}\right\rangle      \]
	that is, to the locus where $T_d$ is not very $u$-adically large.
	We may apply ~\cite[Corollary 5.3]{bellovin2021} to the universal twist of $D_R$ over this space, and we conclude that $Z_R'$ is contained in the subspace $\{\lvert z^{N'}\rvert \leq \lvert u\rvert\}$ for some $N'\gg0$.  Since $Z_R\subset Z_R'$, the same is true of $Z_R$ and its Zariski closure $\overline Z_R$. 

On the other hand, if $T_d$ is $u$-adically large, say, if $T_d\geq u^{-1}$, then we have the pair of inequalities
\[	\lvert u^{-1}\rvert\leq T_d\leq \lvert z\rvert^{-h}	\]
This implies we are over the (affinoid) subspace 
\[	\{\lvert z\rvert^h \leq \lvert u\rvert\}\subset \Spa R\times \Spa R_{z,0}\left\langle z^aT_d^{\pm b}\right\rangle	\]


This lets us study the points consisting of a Galois representation together with a first step in a triangulation and their Zariski-closure; in order to proceed by induction and study the points consisting of Galois representations together with a full triangulation, we will need the following lemma:
\begin{lemma}
	Let $R$ be a pseudoaffinoid algebra with pseudouniformizer $u\in R$, and let $D$ be a family of $(\varphi,\Gamma)$-modules of rank $d$ over $R$ such that $H^0(k(x)\otimes_RD^\vee)$ is non-zero at a Zariski-dense set of maximal points $x\in\Spa R$.  Then there is a finite affinoid cover $\{U_i\}$ of $\Spa R$ and a collection of proper morphisms $\pi_i:\widetilde U_i\rightarrow U_i$ such that
	\begin{enumerate}
		\item 	There are morphisms $\lambda_i:\pi_i^\ast D\rightarrow \Lambda_{\widetilde U_i,\rig,\Q_p}\otimes \mathscr{L}$, for some line bundle $\mathscr{L}$ on $\widetilde U_i$
		\item	The kernel of $\lambda_i$ is a family of $(\varphi,\Gamma)$-modules of rank $d-1$
	\end{enumerate}
	\label{lemma: induction construct triangulation}
\end{lemma}
\begin{proof}
	After replacing $\Spa R$ with a connected component of its normalization, we may assume that $\Spa R$ is normal and irreducible.  Using ~\cite[Corollary 6.3.6(2)]{kpx}, there is a proper birational morphism $f:X_{R}\rightarrow \Spa R$ such that $H^i(f^\ast D^\vee)$ is flat for $i=0$ and has Tor-dimension at most $1$ for $i=1,2$.  For any $x\in X$, we have an exact sequence 
\[	0\rightarrow k_x\otimes H^0\left(f^\ast D^\vee\right)\rightarrow H^0\left(k_x\otimes f^\ast D_{R}^\vee\right)\rightarrow \Tor_1^{\O_X}\left(H^1(f^\ast D^\vee),k_x\right)\rightarrow 0	\]
(where we have used the low-degree exact sequences coming from the base-change spectral sequence cf. ~\cite[Corollary 3.12]{bellovin2021} and the assumption that $H^i(f^\ast D^\vee)$ has Tor-dimension at most $1$ for $i=1,2$).  Since we assumed that $H^0(k(x)\otimes_RD^\vee)$ is non-zero at a Zariski-dense set of maximal points $x\in\Spa R$, we see that $H^0(f^\ast D^\vee)$ is projective of non-zero rank.

Let $g:Y_{R}\rightarrow X_R$ be the projective space $\underline\Proj\left(\Sym H^0(f^\ast D^\vee\right)^\vee)$ over $X_R$. Since $g:Y_R\rightarrow X_R$ is flat, we have $g^\ast H^i(f^\ast D)\xrightarrow\sim H^i(g^\ast f^\ast D)$ for all $i$, and moreover, $g^\ast f^\ast D$ retains the property that $H^0(g^\ast f^\ast D)$ is flat (of non-zero rank) and $H^i(g^\ast f^\ast D)$ has Tor-dimension at most $1$ for $i=1,2$.

Over $Y_R$, there is a universal quotient $g^\ast H^0(f^\ast D^\vee)^\vee\twoheadrightarrow \O_{Y_R}(1)$, which induces an injection $\O_{Y_R}(-1)\rightarrow g^\ast H^0(f^\ast D^\vee)$ with projective cokernel. If we consider the composition
\[	\Lambda_{Y_R,\rig,\Q_p}\otimes\O_{Y_R}(-1)\rightarrow \Lambda_{Y_R,\rig,\Q_p}\otimes g^\ast H^0(f^\ast D^\vee) \rightarrow g^\ast f^\ast D^\vee	\]
we may again dualize to obtain a morphism $\lambda:g^\ast f^\ast D\rightarrow \O_{Y_R}(1)\otimes \Lambda_{Y_R,\rig,\Q_p}$.

There is a finite affinoid cover $\{\Spa R_j'\}$ of $Y_R$ trivializing $\O_{Y_R}(1)$; we let $\lambda_j$ denote the restriction of $\lambda$ to $\Spa R_j'$.  For any $x\in \Spa R_j'$, we again have an exact sequence
\[	0\rightarrow k_x\otimes H^0\left(g^\ast f^\ast D_R^\vee\right)\rightarrow H^0\left(k_x\otimes g^\ast f^\ast D_{R}^\vee\right)\rightarrow \Tor_1^{\O_X}\left(H^1(g^\ast f^\ast D_R^\vee),k_x\right)\rightarrow 0	\]
This implies in particular that the specialization of $\lambda_j$ is non-zero.  If $x$ has characteristic-$p$ residue field, this implies that the specialization of $\lambda_j$ is surjective.  As in the proof of ~\cite[Lemma 5.7]{bellovin2021}, this implies that there is an affinoid subdomain $V_j=\{\lvert p\rvert\leq \lvert u^{r_j}\rvert\}\subset \Spa R_j'$ containing the locus $\{p=0\}$ over which $\lambda_j$ is surjective.  

Let $N:=\max\{r_j\}$ and set $U_1:=\{\lvert p\rvert \leq \lvert u^N\rvert\}\subset \Spa R$.  Then the pre-image $(f\circ g)^{-1}\left(U_1\right)$ is contained in $\cup_j V_j$.  We will set $\widetilde U_1:=(f\circ g)^{-1}\left(U_1\right)$.  Then by construction, $\pi_1:\widetilde U_1\rightarrow U$ is surjective, and 
\[	\lambda|_{\widetilde U_1}:\pi_1^\ast D\rightarrow \pi_1^\ast\O_{Y_R}(1)|_{U_1}	\]
is surjective, so its kernel is a family of $(\varphi,\Gamma)$-modules of rank $d-1$.

On the other hand, set $U_2:=\{\lvert u^N\rvert \leq \lvert p\rvert\}\subset \Spa R$.  Then the pre-image $(f\circ g)^{-1}\left(U_2\right)$ is quasi-compact and contained in the characteristic-$0$ locus of $Y_R$, so we may apply the techniques of the proof of ~\cite[Theorem 6.3.9]{kpx}.  More precisely, we let $h:U_2'\rightarrow (f\circ g)^{-1}\left(U_2\right)$ be a proper birational morphism so that $H^i((f\circ g\circ h)^\ast D/t)$ is flat for $i=0$ and has Tor-dimension at most $1$ for $i=1, 2$ (again using ~\cite[Corollary 6.3.6(2)]{kpx}).  This lets us deduce that $h^\ast\lambda|_{U_2'}$ is surjective away from a proper Zariski-closed subspace, and locally on $U_2'$, its cokernel is killed by a power of $t$.  Then we make a further blow-up $\widetilde U_2\rightarrow U_2'$ such that over $U_2$, the kernel of $\lambda$ is a family of $(\varphi,\Gamma)$-modules of rank $d-1$, as desired.
\end{proof}

This permits us to use induction to deduce the following:
\begin{cor}
	Let $R$ be a pseudoaffinoid algebra with pseudouniformizer $u\in R$, and let $D$ be a family of rank-$d$ $(\varphi,\Gamma)$-modules over $R$.  Consider the Zariski closure $Z$ of the locus in $\Spa R_0\times \Spa R_{z,0}\times\G_{m,R}^d$ corresponding to points $x=(D_x,\underline\delta_x)$ where $\lvert \delta_{i,x}(p)^{\pm 1}\rvert \leq \lvert z\rvert^{-h}$ for all $i$, and $\underline\delta_x$ is a regular parameter of $D_x$.  Then there is some $N'\gg 0$ such that
	\[	Z\subset  \{\lvert z^{N'}\rvert \leq \lvert u\rvert\text{ for all }i\}\subset \Spa R\times \Spa R_z\left\langle z^hT_i^{\pm 1}\right\rangle	\]
\end{cor}

This is precisely the condition we need to apply Corollary~\ref{cor: integral model in product}, so the closure we are interested in is well-behaved in $\Spf R_0\htimes R_{z,0}\left\langle z^aT_i^b,T_i'\right\rangle/(T_i^bT_i'-z^a)$.

Letting $\Spa R$ range over a (finite) cover of $\left(\Spa R_{\overline\rho}^\square\right)^{\an}$, we obtain a closed subspace of $\Spf R_{\overline \rho}^\square\times \Spa R_{z,0}\left\langle z^aT_i^b,T_i'\right\rangle/(T_i^bT_i'-z^a)$. 
Letting $z$ range over $\{p,z_1,\ldots,z_d\}$, in turn, we may glue to get a closed subspace of $\Spf R_0\times\mathfrak{T}_h^d$, yielding the desired integral models of pieces of trianguline varieties:
\begin{cor}
	Suppose that $\overline\rho$ is a representation of $\Gal_K$, where $K$ is a finite extension of $\Q_p$, or of $\Gal_F$, where $F$ is a totally real number field (in which case we assume $\rho$ is unramified outside a finite set of places $S$).  Then there are formal schemes $\mathfrak{X}_{{\tri},\overline\rho,\leq h}^{\square}$
(resp. $\mathfrak{X}_{{\tri},\overline\rho,S,\leq h}^{\square}$)
and $\mathfrak{X}_{{\tri},\overline\rho,\leq h}^{\square,\psi,\underline\kappa}$ (resp. $\mathfrak{X}_{{\tri},\overline\rho,S,\leq h}^{\square,\psi,\underline\kappa}$), which are affine over $\mathfrak{T}_h$, such that $\left(\mathfrak{X}_{{\tri},\overline\rho,\leq h}^{\square}\right)^{\an}=X_{{\tri},\overline\rho,\leq h}^\square$ (resp. $\left(\mathfrak{X}_{{\tri},\overline\rho,S,\leq h}^{\square}\right)^{\an}=X_{{\tri},\overline\rho,S,\leq h}^\square$) and $\left(\mathfrak{X}_{{\tri},\overline\rho,\leq h}^{\square,\psi,\underline\kappa}\right)^{\an}=X_{{\tri},\overline\rho,\leq h}^{\square,\psi,\underline\kappa}$ (resp. $\left(\mathfrak{X}_{{\tri},\overline\rho,S,\leq h}^{\square,\psi,\underline\kappa}\right)^{\an}=X_{{\tri},\overline\rho,S,\leq h}^{\square,\psi,\underline\kappa}$).
\end{cor}

\section{Extended eigenvarieties}\label{section: extended eigenvarieties}

\subsection{Definitions}\label{subsection: eigenvariety definitions}

We briefly recall the construction of extended eigenvarieties in the two cases of interest to us.  Fix a number field $F$ and a reductive group $\H$ over $F$ which is split at all places above $p$; then we define $\G:=\Res_{F/\Q}\H$.  If we choose split models $\H_{\O_{F_v}}$ over $\O_{F_v}$ for each place $v\mid p$, along with split maximal tori and Borel subgroups $\T_v\subset \B_v\subset \H_{\O_{F_v}}$, we obtain an integral model $\G_{\Z_p}:=\prod_{v\mid p}\H_{\O_{F_v}}$ of $\G$, as well as closed subgroup schemes 
\[	\T:=\prod_{v\mid p}\Res_{\O_{F_v}/\Z_p}T_v\subset \B:=\prod_{v\mid p}\Res_{\O_{F_v}/\Z_p}\B_v	\]
Let $T_0:=\T(\Z_p)$, and let the Iwahori subgroup $I\subset\G_{\Z_p}(\Z_p)$ be the pre-image of $\B(\F_p)$ under the reduction map $\G_{\Z_p}(\Z_p)\rightarrow \G_{\Z_p}(\F_p)$.

We choose a tame level by choosing compact open subgroups $K_\ell\subset \G(\Q_\ell)$ for each prime $\ell\neq p$, such that $K_\ell=\mathcal{G}(\Z_\ell)$ for almost all primes $\ell$ (where $\mathcal{G}$ is some reductive model of $\G$ over $\Z[1/M]$ for some integer $M$).  Then we put $K^p:=\prod_{\ell\neq p}K_\ell$ and $K:=K^pI$; we assume throughout that $K$ contains an open normal subgroup $K'$ such that $[K:K']$ is prime to $p$ and
\begin{equation}\label{hyp: neatness}
	x^{-1}D^\times x\cap K' \subset \O_F^{\times, +} \text{\quad for all } x\in (\A_{F,f}\otimes_F D)^\times
\end{equation}
which is the neatness hypothesis of ~\cite{johansson-newton19}.\footnote{The authors assume throughout that the level is neat; to relax this assumption, one chooses an open normal subgroup $K'\subset K$ of index prime to $p$ such that $K'$ is neat, and considers the complexes $C_c^\bullet(K',-)^{K/K'}$ and $C_\bullet^{\mathrm{BM}}(K',-)_{K/K'}$.  Since $K/K'$ has order prime to $p$, the finite-slope subcomplexes $C_c^\bullet(K,\mathscr{D}_\kappa)_{\leq h}^{K/K'}$ and $C_\bullet^{\mathrm{BM}}(K',-)_{\leq h,K/K'}$ remain perfect.}
If $\Z$ denotes the center of $\G$, we let $Z(K):=\Z(\Q)\cap K$ and let $\overline{Z(K)}\subset T_0$ denote its $p$-adic closure.  We also let $K_\infty\subset \G(\R)$ be a maximal compact and connected subgroup at infinity, and let $Z_\infty^\circ\subset Z_\infty=:\Z(\R)$ denote the identity component.

Finally, let $\Sigma\subset T_0$ be the kernel of some splitting of the inclusion $T_0\subset \T(\Q_p)$; there are then certain submonoids $\Sigma^{\mathrm{cpt}}\subset \Sigma^+\subset \Sigma$, and we fix some $t\in \Sigma^{\mathrm{cpt}}$.

In the cases of interest to us, $F$ will be a totally real field, completely split at $p$, and $\H$ will be either $\GL_2$ or the reductive group $\underline D^\times$ corresponding to the units of a totally definite quaternion algebra over $F$ split at every place above $p$.  Fixing isomorphisms $D_v\xrightarrow{\sim}\Mat_2(F_v)$ for each place $v$ where $D$ is split, we may define integral models of $\H_v$ via $\H_{\O_{F_v}}(R_0):=\left(R_0\otimes \Mat_2(\O_{F_v})\right)^\times$ for all $\O_{F_v}$-algebras $R_0$ (whether $\H=\GL_2$ or $\underline D^\times$).  In either case, we let $\B_v\subset \H_{\O_{F_v}}$ be the standard upper-triangular Borel and we let $\T_v\subset\B_v$ be the standard diagonal maximal torus.

For either choice of $\H$, the adelic subgroup $K(N)\subset \left(\A_{F,f}\otimes \H(F)\right)^\times$ of full level $N$ is neat for $N\geq 3$ such that $N$ is prime to the finite places $v$ where $\H_v\neq \GL_2$.  Thus, if we assume $p\geq 5$, we may take $K^p$ arbitrary.

For either choice of $\H$, we define $\Sigma_v^+:=\left\{\left(\begin{smallmatrix}\varpi_v^{a_1} & 0 \\ 0 & \varpi_v^{a_2}\end{smallmatrix}\right) \mid a_2\geq a_1\right\}$ and $\Delta_v:=I_v\Sigma_v^+I_v$.  Similarly, we define $\Sigma^+:=\prod_{v\mid p}\Sigma_v^+$ and $\Delta_p:=I\Sigma^+I=\prod_{v\mid p}\Delta_v$. Then we fix $U_{\varpi_v}:=\left[I_v\left(\begin{smallmatrix}1 & \\ & \varpi_v\end{smallmatrix}\right)I_v\right]\in I_v\backslash \H(F_v)/I_v$ and $U_p:=\prod_{v\mid p}U_{\varpi_v}$.

For each prime $\ell\neq p$, we fix a monoid $\Delta_\ell\subset \G(\Q_\ell)$ containing $K_\ell$, which is equal to $\G(\Q_\ell)$ when $K_\ell=\mathcal{G}(\Z_\ell)$, such that $(\Delta_\ell, K_\ell)$ is a Hecke pair and the Hecke algebra $\mathbb{T}(\Delta_\ell, K_\ell)$ over $\Z_p$ is commutative.  Then we define $\Delta^p:=\prod_{\ell\neq p}^\prime \Delta_\ell$ and $\Delta:=\Delta^p\Delta_p$.  We write $\mathbb{T}(\Delta^p,K^p):=\otimes_{\ell\neq p}\mathbb{T}(\Delta_\ell, K_\ell)$ and $\mathbb{T}(\Delta,K):=\otimes_\ell \mathbb{T}(\Delta_\ell, K_\ell)$ for the corresponding global Hecke algebras.

A \emph{weight} is a continuous homomorphism $\kappa:T_0\rightarrow R^\times$ which is trivial on $Z(K)$, where $R$ is a pseudoaffinoid algebra over $\Z_p$.  We define \emph{weight space} $\mathscr{W}$ via
\[	\mathscr{W}(R):=\{\kappa\in\Hom_{\mathrm{cts}}(T_0, R^\times) \mid \kappa|_{Z(K)}=1\}	\]
It can be written explicitly as the analytic locus of $\Spa\left(\Z_p[\![T_0/\overline{Z(K)}]\!],\Z_p[\![T_0/\overline{Z(K)}]\!]\right)$.  Then $\mathscr{W}$ is equidimensional of dimension $1+[F:\Q]+\mathfrak{d}$, where $\mathfrak{d}$ is the defect in Leopoldt's conjecture for $F$ and $p$.

The next step is to construct a sheaf of Hecke modules over weight space, such that $U_p$ acts compactly and admits a Fredholm determinant.  We will actually use two such sheaves.  If $\kappa:T_0\rightarrow R^\times$ is a weight, then ~\cite{johansson-newton} construct certain modules of analytic functions $\mathcal{A}_\kappa^r$  and distributions $\mathcal{D}_\kappa^r$.  Here $r\in (r_\kappa,1)$, where $r_\kappa\in[1/p,1)$.  When $r_\kappa\in(1/p,1)$, they also construct $\mathcal{A}_\kappa^{<r}$ and $\mathcal{D}_\kappa^{<r}$, so that $\mathcal{D}_\kappa^r$ is the dual of $\mathcal{A}_\kappa^{<r}$ and $\mathcal{A}_\kappa^r$ is the dual of $\mathcal{D}_\kappa^{<r}$.  As in ~\cite{hansen} we fix augmented Borel--Serre complexes $C_\bullet^{\mathrm{BM}}(K,-)$ and $C_{c}^\bullet(K,-)$ for Borel--Moore homology and compactly supported cohomology, respectively, and we consider
\[	C_\bullet^{\mathrm{BM}}(K,\mathcal{A}_\kappa^r)	\]
as well as
\[	C_c^\bullet(K,\mathcal{D}_\kappa^r)\quad\text{ and }\quad C_c^\bullet(K,\mathcal{D}_\kappa^{<r})	\]
Now $\mathcal{A}_\kappa^r$ and $\mathcal{D}_\kappa^r$ are potentially orthonormalizable, so $C_\ast^{\mathrm{BM}}(K,\mathcal{A}_\kappa^r):=\oplus_i C_i^{\mathrm{BM}}(K,\mathcal{A}_\kappa^r)$ and $C_c^\ast(K,\mathcal{D}_\kappa^r):=\oplus_iC_c^i(K,\mathcal{D}_\kappa^r)$ are, as well.  Since $U_p$ acts compactly on $\mathcal{A}_\kappa^r$ and $\mathcal{D}_\kappa^r$, this implies that there are Fredholm determinants $F_\kappa^{r,\prime}$ and $F_\kappa^r$ for its action on $C_\ast^{\mathrm{BM}}(K,\mathcal{A}_\kappa^r)$ and $C_c^\ast(K,\mathcal{D}_\kappa^r)$, respectively.  

It turns out that  $F_\kappa^{r,\prime}$ and $F_\kappa^r$ are independent of $r$, by ~\cite[Proposition 4.1.2]{johansson-newton}; we set $\mathscr{D}_\kappa:=\varprojlim_r\mathcal{D}_\kappa^r$ and $\mathscr{A}_\kappa:=\varinjlim_r\mathcal{A}_\kappa^r$, and we write $F_\kappa$ and $F_\kappa^\prime$ for the Fredholm determinants of $U_p$ on $C_c^\ast(K,\mathscr{D}_\kappa)$ and $C_\ast^{\mathrm{BM}}(K,\mathscr{A}_\kappa)$, respectively.  Then $F_\kappa$ and $F_\kappa^\prime$ define spectral varieties $\mathscr{Z}\subset\A_{\mathscr{W}_F}^1$ and $\mathscr{Z}^\prime\subset\A_{\mathscr{W}_F}^1$.  We let $\pi:\mathscr{Z}\rightarrow\mathscr{W}_F$ and $\pi':\mathscr{Z}'\rightarrow\mathscr{W}_F$ be the projection on the first factor; they are flat morphisms of pseudorigid spaces.

By ~\cite[Theorem 2.3.2]{johansson-newton}, $\mathscr{Z}$ has a cover by open affinoid subspaces $V$ such that $U:=\pi(V)$ is an open affinoid subspace of $\mathscr{W}_F$ and $\pi|_V:V\rightarrow U$ is finite of constant degree.  This implies that over such a $V$, $F$ factors as $F_V=Q_VS_V$ where $Q_V$ is a multiplicative polynomial of degree $\deg \pi|_V$, $S_V$ is a Fredholm series, and $Q_V$ and $S_V$ are relatively prime.

If such a factorization exists, we may make $C_c^\bullet(K,\mathscr{D}_V)$ into a complex of $\O_{\mathscr{Z}}$-modules by letting $T$ act via $U_p^{-1}$.  Then the assignment $V\mapsto \ker Q_V^\ast(U_p)\subset C^\bullet(K,\mathscr{D}_{V})$ defines a bounded complex $\mathscr{K}^\bullet$ of coherent $\O_{\mathscr{Z}}$-modules, where $Q_V^\ast(T):=T^{\deg Q_V}Q_V(1/T)$.  If $V=\pi^{-1}(U)$, where $(U,h)$ is a slope datum, then $\mathscr{K}^\bullet$ is the slope-$\leq h$ subcomplex of $C_c^\bullet(K,\mathscr{D}_V)$.  We set
	\[	\mathscr{M}_c^\ast := \oplus_i H^i(\mathscr{K})	\]
	which is a coherent sheaf on $\mathscr{Z}$.  

	Such factorizations exist locally, by an extension of a result of ~\cite{ash-stevens}:
	\begin{prop}
		Let $R$ be a pseudoaffinoid algebra, and let $x_0\in\Spa R$ be a maximal point.  Let $F(T)\in R\{\!\{T\}\!\}$ be a Fredholm power series and fix $h\in \Q$.  Suppose $F_{x_0}\neq 0$, and let $F_{x_0}=Q_{0}S_{0}$ be the slope $\leq h$-factorization of the specialization of $F$ at $x_0$.  Then there is an open affinoid subspace $U\subset \Spa R$ containing $x_0$ such that $F_U$ has a slope $\leq h$-factorization $F_U=QS$ with $Q$ specializing to $Q_0$ and $S$ specializing to $S_0$ at $x_0$.
		\label{prop: local slope factorizations}
	\end{prop}
	\begin{proof}
		The existence of the factorization of $F_{x_0}$ follows from the version of the Weierstrass preparation theorem proved in ~\cite[Lemma 4.4.3]{ash-stevens}.  Then the proof of the proposition is nearly identical to that of ~\cite[Theorem 4.5.1]{ash-stevens}, up to replacing $p$ with $u$ and translating the numerical inequalities into rational localization conditions.
	\end{proof}

	Since spectral varieties are flat over weight space, we will be able to use the following result to show that slope factorizations exist:
	\begin{thm}[{\cite[Theorem A.1.2]{conrad2006}}]\label{thm: finiteness criterion}
		Let $f:X\rightarrow Y$ be a flat map of pseudorigid spaces.  Then $f$ is finite if and only it is quasi-compact and separated with finite fibers, and its fiber rank is locally constant on $Y$.
	\end{thm}
	\begin{remark}
		This result is stated in ~\cite{conrad2006} for classical rigid spaces, but the proof goes through unchanged for pseudorigid spaces.  The input from non-archimedean geometry is the theory of formal models (and flattening results) of ~\cite{bosch-lut}, ~\cite{bosch-lut2}; Although the authors had in mind applications to classical rigid analytic spaces, they worked in sufficient generality that their results hold in the more general pseudorigid context.  One uses this theory to reduce to the corresponding algebraic result of ~\cite[Lemma II.1.19]{deligne-rapoport}.
	\end{remark}

	We further observe that we have inclusions $\mathcal{D}_\kappa^r\subset\mathcal{D}_\kappa^{<r}\subset\mathcal{D}_\kappa^s$ for any $r_\kappa\leq s<r$.  Thus, the fact that $F_\kappa^r=F_\kappa^s$ implies that $\mathscr{M}_c^\ast = \oplus_i H_c^i(K,\mathscr{D}_\kappa^{<r})_{\leq h}$ for any $r>r_\kappa$.

	We may carry out the same procedure for the action of $U_p$ on $C_\ast^{\mathrm{BM}}(K,\mathscr{A}_\kappa)$, and obtain a coherent sheaf $\mathscr{M}_\ast^{\mathrm{BM}}=\oplus_iH_i^{\mathrm{BM}}(K,\mathscr{A}_\kappa)_{\leq h}$ on $\mathscr{Z}^{\prime}$.  Let $\mathbb{T}$ denote either $\mathbb{T}(\Delta^p,K^p)$ or $\mathbb{T}(\Delta,K)$.  Both $\mathscr{M}_c^\ast$ and $\mathscr{M}_\ast^{\mathrm{BM}}$ are Hecke modules, so we have constructed \emph{eigenvariety data} $(\mathscr{Z},\mathscr{M}_c^\ast,\mathbb{T},\psi)$ and $(\mathscr{Z}^\prime,\mathscr{M}_\ast^{\mathrm{BM}},\mathbb{T},\psi')$ (where $\psi:\mathbb{T}\rightarrow \End_{\O_\mathscr{Z}}(\mathscr{M}_c^\ast)$ and $\psi':\mathbb{T}\rightarrow \End_{\O_{\mathscr{Z}^\prime}}(\mathscr{M}_\ast^{\mathrm{BM}})$ give the Hecke-module structures).

	Finally, we may construct eigenvarieties from the eigenvariety data.  Let $\mathscr{T}$ and $\mathscr{T}^\prime$ denote the sheaves of $\O_{\mathscr{Z}}$-algebras generated by the images of $\psi$ and $\psi'$, respectively; in particular, if $\mathscr{Z}_{U,h}\subset\mathscr{Z}$ is an open affinoid corresponding to the slope datum $(U,h)$, then
	\[	\mathscr{T}(\mathscr{Z}_{U,h}) = \im\left(\O(\mathscr{Z}_{U,h})\otimes_{\Z_p}\mathbb{T}\rightarrow \End_{\O(\mathscr{Z}_{U,h})}\left(H_c^\ast(K,\mathscr{D}_{U}\right)_{\leq h}\right)=:\mathbb{T}_{U,h}	\]
		and
		\[	\mathscr{T}^\prime(\mathscr{Z}'_{U,h}) = \im\left(\O(\mathscr{Z}'_{U,h})\otimes_{\Z_p}\mathbb{T}\rightarrow \End_{\O(\mathscr{Z}'_{U,h})}\left(H_\ast^{\mathrm{BM}}(K,\mathscr{A}_{U}\right)_{\leq h}\right)=:\mathbb{T}_{U,h}'	\]
Then we set
\[	\mathscr{X}_{\G}^{\mathbb{T}}:=\underline\Spa \mathscr{T}	\]
and
\[	\mathscr{X}_{\G}^{\mathbb{T},\prime} :=\underline\Spa \mathscr{T}^\prime	\]
and we have finite morphisms $q:\mathscr{X}_{\G}\rightarrow\mathscr{Z}$ and $q':\mathscr{X}_{\G}^\prime\rightarrow \mathscr{Z}^\prime$, and $\Z_p$-algebra homomorphisms $\phi_{\mathscr{X}}:\mathbb{T}\rightarrow \O(\mathscr{X}_{\G}^{\mathbb{T}})$ and $\phi_{\mathscr{X}'}:\mathbb{T}\rightarrow \O(\mathscr{X}_{\G}^{\mathbb{T},\prime})$.  If the choice of Hecke operators is clear from context, we will drop $\mathbb{T}$ from the notation.

If $\mathbb{T}=\mathbb{T}(\Delta,K)$, then unlike ~\cite{johansson-newton}, we are adding the Hecke operators $U_{\varpi_v}$ at places $v\mid p$ to our Hecke algebras (and hence to the coordinate rings of our eigenvarieties), not just the controlling operator $U_p$.

\subsection{The middle-degree eigenvariety}

When $F=\Q$ and $\G=\H=\GL_2$, for any fixed slope $h$ such that $C_c^\bullet(K,\mathscr{D}_\kappa)$ has a slope-$\leq h$ decomposition, the complex $C_c^\bullet(K,\mathscr{D}_\kappa)_{\leq h}$ has cohomology only in degree $1$, and $H_c^1(K,\mathscr{D}_\kappa)_{\leq h}$ is projective.  As a result, the eigencurve is reduced and equidimensional, and classical points are very Zariski-dense.  For a general totally real field $F$, the situation is more complicated.  The complex $C_c^\bullet(K,\mathscr{D}_\kappa)_{\leq h}$ lives in degrees $[0,2d]$ and we are still primarily interested in the degree-$d$ cohomology; indeed, the discussion of ~\cite[\textsection 3.6]{harder1987} shows that cuspidal cohomological automorphic forms contribute only to middle degree cohomology in the classical finite-dimensional classical analogue.  However, there is no reason to expect the other cohomology groups to vanish.

Instead, following ~\cite{bergdall-hansen} we will sketch the construction of an open subspace $\mathscr{X}_{\GL_2/F,\mathrm{mid}}\subset \mathscr{X}_{\GL_2/F}$ where $H_c^i(K,\mathscr{D}_\kappa)$ vanishes for $i\neq d$; by ~\cite[Theorem B.0.1]{bergdall-hansen}, all classical points of $\mathscr{X}_{\GL_2/F}$ whose associated Galois representation have sufficiently large residual image lie in $\mathscr{X}_{\GL_2/F,\mathrm{mid}}$.  The cohomology and base change result ~\cite[Theorem 4.2.1]{johansson-newton} shows that the locus where $H_c^i(K,\mathscr{D}_\kappa)=0$ for $i\geq d+1$ is open, but we need to use the homology complexes $C_\bullet^{\mathrm{BM}}(K,\mathcal{A}_\kappa)$ to control $H_c^i(K,\mathscr{D}_\kappa)$ for $i\leq d-1$.

As in ~\cite{bergdall-hansen}, the key points are a base change result for Borel--Moore homology, and a universal coefficients theorem relating it to compactly supported cohomology:
\begin{prop}
	\begin{itemize}
		\item	There is a third-quadrant spectral sequence 
			\[	E_2^{i,j} = \Tor_{-i}^R(H_{-j}^{\mathrm{BM}}(K,\mathscr{A}_\kappa)_{\leq h},S)\Rightarrow H_{-i-j}^{\mathrm{BM}}(K,\mathscr{A}_{\kappa_S})_{\leq h}	\]
		\item	There is a second-quadrant spectral sequence 
			\[	E_2^{i,j}=\Ext_R^i(H_j^{\mathrm{BM}}(K,\mathscr{A}_\kappa)_{\leq h},R)\Rightarrow H_c^{i+j}(K,\mathscr{D}_\kappa)_{\leq h}	\]
	\end{itemize}
	These are spectral sequences of $\mathbb{T}(\Delta,K)$-modules.
\end{prop}
The proof uses both the fact that $\mathcal{D}_\kappa^{<r}$ is the continuous dual of $\mathcal{A}_\kappa^r$, and the fact that $H_c^i(K,\mathcal{D}_\kappa^{<r})_{\leq h}=H_c^i(K,\mathcal{D}_\kappa^r)_{\leq h}$ for all $r>r_\kappa$.

\begin{prop}
	If $(U,h)$ is a slope datum, then we have a natural commuting diagram
	\[
		\begin{tikzcd}
			\O(U)\otimes\mathbb{T}(\Delta,K) \ar[r]\ar[d] & \mathbb{T}_{U,h}^\prime \arrow[d,dashed]	\\
				\mathbb{T}_{U,h}' \ar[r] & \mathbb{T}_{U,h}^{\mathrm{red}}
		\end{tikzcd}
	\]
\end{prop}

Thus, we have a morphism $\tau:\mathscr{X}_{\GL_2/F}^{\red}\rightarrow \mathscr{X}_{\GL_2/F}^\prime$ and a closed immersion $i:\mathscr{X}_{\GL_2/F}^{\red}\hookrightarrow \mathscr{X}_{\GL_2/F}$.

\begin{definition}
	\[	\mathscr{X}_{\GL_2/F,\mathrm{mid}}:=\mathscr{X}_{\GL_2/F}\smallsetminus\left[\left(\cup_{j=d+1}^{2d}\supp(\mathscr{M}_c^j)\right)\cup\left(\cup_{j=0}^{d-1}\supp(i_\ast\tau^\ast\mathscr{M}_j^{\mathrm{BM}}\right)\right]	\]
	\label{def: middle deg eigen}
\end{definition}

By construction, a point $x\in\mathscr{X}_{\GL_2/F}$ of weight $\lambda_x$ lies in the Zariski-open subspace $\mathscr{X}_{\GL_2/F,\mathrm{mid}}\subset \mathscr{X}_{\GL_2/F}$ if and only if $H_c^j(K,k_x\otimes \mathscr{D}_{\lambda_x})_{\mathfrak{m}_x}=0$ for all $j\neq d$ (where $\mathfrak{m}_x$ is the maximal ideal of the Hecke algebra corresponding to $x$).
\begin{prop}
	\begin{enumerate}
		\item	The coherent sheaf $\mathscr{M}_c^d|_{\mathscr{X}_{\GL_2/F,\mathrm{mid}}}$ is flat over $\mathscr{W}$.
		\item	$\mathscr{X}_{\GL_2/F,\mathrm{mid}}$ is covered by open affinoids $W$ such that $W$ is a connected component of $(\pi\circ q)^{-1}(U)$, where $(U,h)$ is some slope datum, and $\mathscr{T}(W)$ acts faithfully on $\mathscr{M}_c^d(W)\cong e_WH_c^d(K,\mathscr{D}_\kappa)_{\leq h}$ (where $e_W$ is the idempotent projector restricting from $(\pi\circ q)^{-1}(U)$ to $W$).
	\end{enumerate}
\end{prop}
\begin{proof}
	This follows from the base change spectral sequence, and the criterion for flatness.
\end{proof}

\subsection{Jacquet--Langlands}

The classical Jacquet--Langlands correspondence lets us transfer automorphic forms between $\GL_2$ and quaternionic algebraic groups.  Over $\Q$, this correspondence was interpolated in ~\cite{chenevier2005} to give a closed immersion of eigencurves $\mathscr{X}_{D^\times/\Q}^\rig\hookrightarrow \mathscr{X}_{\GL_2/\Q}^\rig$; this interpolation was given for general totally real fields in ~\cite{birkbeck2019}.  We give the corresponding result for extended eigenvarieties.  However, as we have elected to work with the eigenvariety for $\GL_2/F$ constructed in ~\cite{johansson-newton} via overconvergent cohomology, instead of the eigenvariety constructed from Hilbert modular forms, we will never get an isomorphism of eigenvarieties, even when $[F:\Q]$ is even.

Let $D$ be a totally definite quaternion algebra over $F$, split at every place above $p$, and let $\mathfrak{d}_D$ be its discriminant.  For any ideal $\mathfrak{n}\subset\O_F$ with $(\mathfrak{d}_D,\mathfrak{n})=1$, we define the subgroup $K_1^{\underline D^\times}(\mathfrak{n})\subset (\O_D\otimes \widehat{\Z})^\times$ 
\[	K_1^{\underline D^\times}(\mathfrak{n}):=\left\{g\in (\O_D\otimes \widehat{\Z})^\times \mid g\equiv \left(\begin{smallmatrix}\ast & \ast \\ 0 & 1\end{smallmatrix}\right)\pmod{\mathfrak{n}}\right\}	\]
We may define a similar subgroup $K_1^{\GL_2/F}(\mathfrak{n})\subset \Res_{\O_F/\Z_p}\GL_2(\widehat\Z)$.

A \emph{classical algebraic weight} is a tuple $(k_\sigma)\in\Z_{\geq 2}^{\Sigma_\infty}$ together with a tuple $(v_\sigma)\in \Z^{\Sigma_\infty}$ such that $(k_\sigma)+(v_\sigma) = (r,\ldots,r)$ for some $r\in \Z$, where $\Sigma_\infty$ is the set of embeddings $F\hookrightarrow \R$.  Set $e_1:=(\frac{r+k_\sigma}{2})$ and $e_2:=(\frac{r-k_\sigma}{2})$, and define characters $\kappa_i:F^\times\rightarrow \R^\times$ for $i=1,2$ via
\[	\kappa_i(x) = \prod_{\sigma\in\Sigma_\infty}\sigma(x)^{e_{i,\sigma}}	\]
Then $(\kappa_1,\kappa_2)$ is a character on $\T(\Z)$ which is trivial on a finite-index subgroup of the center $Z_G(\Z)=\O_F^\times$.

Then we have the classical Jacquet--Langlands correspondence:
\begin{thm}
	Let $\kappa$ be a classical weight, and let $\mathfrak{n}\subset\O_F$ be an ideal such that $(\mathfrak{n},\mathfrak{d}_D)=1$.  There is a Hecke-equivariant isomorphism of spaces of cusp forms
	\[	S_\kappa^{\underline D^\times}(K_1^{\underline D^\times}(\mathfrak{n}))\xrightarrow\sim S_\kappa^{\mathfrak{d}_D-\mathrm{new}}(K_1^{\GL_2/F}(\mathfrak{n}\mathfrak{d}_D))	\]
\end{thm}

We will interpolate this correspondence to a closed immersion $\mathscr{X}_{\underline D^\times}\hookrightarrow \mathscr{X}_{\GL_2/F}$, where the source has tame level $K_1^{\underline D^\times}(\mathfrak{n})$ and the target has tame level $K_1^{\GL_2/F}(\mathfrak{n})$.  We use the interpolation theorem of ~\cite{johansson-newton17}:
\begin{thm}[{\cite[Theorem 3.2.1]{johansson-newton17}}]
	Let $\mathfrak{D}_i=(\mathscr{Z}_i,\mathscr{M}_i,\mathbb{T}_i,\psi_i)$ for $i=1,2$ be eigenvariety data, with corresponding eigenvarieties $\mathscr{X}_i$, and suppose we have the following:
	\begin{itemize}
		\item	A morphism $j:\mathscr{Z}_1\rightarrow\mathscr{Z}_2$
		\item	A $\Z_p$-algebra homomorphism $\mathbb{T}_2\rightarrow\mathbb{T}_1$
		\item	A subset $\mathscr{X}^{\mathrm{cl}}\subset \mathscr{X}_1$ of maximal points such that the $\mathbb{T}_2$-eigensystem of $x$ appears in $\mathscr{M}_2(j(\pi_1(x)))$ for all $x\in\mathscr{X}^{\mathrm{cl}}$.
	\end{itemize}
	Let $\overline{\mathscr{X}}\subset \mathscr{X}_1$ denote the Zariski closure of $\mathscr{X}^{\mathrm{cl}}$ (with its underlying reduced structure).  Then there is a canonical morphism $i:\overline{\mathscr{X}}\rightarrow \mathscr{X}_2$ lying over $j$, such that $\phi_{\overline{\mathscr{X}}}\circ \sigma = i^\ast\circ\phi_{\mathscr{X}_2}$.  If $j$ is a closed immersion and $\sigma$ is a surjection, then $i$ is a closed immersion.
\end{thm}

We remark that in the presence of integral structures, we can make a sharper statement:
\begin{cor}
	\label{cor: interpolation integral}
With notation as above, suppose that the $\mathscr{Z}_i=\Spa R_i$ are affinoid, with $R_{i,0}\subset R_i$ rings of definition such that $j$ is induced by a morphism $\Spf R_{1,0}\rightarrow \Spf R_{2,0}$, and suppose that $M_i:=\Gamma(\mathscr{Z}_i,\mathscr{M}_i)$ admit $R_{i,0}$-lattices $M_{i,0}$ stable under the actions of $\mathbb{T}_i$.  Let $R_{i,0}':=\im\left(R_{i,0}\otimes\mathbb{T}_i\rightarrow \End_{R_{i,0}}(M_{i,0})\right)$ and let $\overline{\mathscr{X}}_0$ denote the closure of $\mathscr{X}^{\mathrm{cl}}$ in $\Spf R_{1,0}'$.  Then there is a morphism $j_0:\overline{\mathscr{X}}_0\rightarrow \Spf R_{2,0}'$.
\end{cor}
\begin{proof}
	As in the proof of~\cite[Theorem 3.2.1]{johansson-newton17}, one reduces to the case where $R_0:=R_{1,0}=R_{2,0}$ and $\mathbb{T}:=\mathbb{T}_1=\mathbb{T}_2$, and one considers the actions of $\mathbb{T}_1\oplus \mathbb{T}_2$ on $M_{1,0}\oplus M_{2,0}$.  Then we have quotients 
	\[	R_{3,0}:=\im\left(R_0\otimes\mathbb{T}\rightarrow \End_{R_0}(M_{1,0}\oplus M_{2,0})\right)\twoheadrightarrow R_{i,0}'	\]
	Since $\overline{\mathscr{X}}\subset \mathscr{X}_2$ and $\Spf R_{3,0}$ is separated, we have $\overline{\mathscr{X}}_0\subset \Spf R_{2,0}'$, as desired.
\end{proof}

We take $\mathscr{Z}_1=\mathscr{Z}_2=\mathscr{W}_F\times\G_m$.  In order to define $\mathbb{T}=\mathbb{T}_1=\mathbb{T}_2$, we set
\[	\Delta_{v} = \begin{cases}
			\GL_2(F_v) &\mbox{ if } v\nmid p\mathfrak{d}_D\mathfrak{n}	\\
			{K_1^{\underline D^\times}(\mathfrak{n})}_v &\mbox{ if } v\mid \mathfrak{d}_D\mathfrak{n}
	\end{cases}	\]
For $v\mid p$, we take $\Delta_{v}$ as in \textsection\ref{subsection: eigenvariety definitions}.
In other words, $\mathbb{T}$ is the commutative $\Z_p$-algebra generated by $T_v:=[K_v\left( \begin{smallmatrix}1 & \\ & \varpi_v\end{smallmatrix} \right)K_v]$ and $S_v:=[K_v\left( \begin{smallmatrix}\varpi_v & \\ & \varpi_v\end{smallmatrix} \right)K_v]$ for $v\nmid p\mathfrak{d}_D\mathfrak{n}$ and $U_{\varpi_v}$ for $v\mid p$.

However, we cannot immediately combine this interpolation theorem with the Jacquet--Langlands correspondence, because our choice of weight space means that classical weights may not be Zariski dense unless Leopoldt's conjecture is true.  More precisely, given a classical algebraic weight, we constructed a character on $\T(\Z)$ trivial on a finite-index subgroup of $\O_F^\times$, and conversely, characters on $\T(\Z)$ trivial on a finite-index subgroup of $\O_F^\times$ yield classical algebraic weights.  This equivalence relies on Dirichlet's unit theorem.  

This means that there are two natural definitions of $p$-adic families of weights, $\mathscr{W}_F'=\Spa \Z_p[\![(\Res_{\O_F/\Z_p}\G_m)\times \Z_p^\times]\!]^{\an}$ interpolating classical algebraic weights, and $\mathscr{W}_F$ interpolating characters on $T_0$, and the equivalence of those two definitions depends on Leopoldt's conjecture.

Fortunately, the gap between these weight spaces can be controlled: there is a closed embedding $\mathscr{W}_F'\hookrightarrow \mathscr{W}_F$, and the twisting action by characters on $\O_{F,p}^\times/\overline{\O_F^{\times,+}}$ defines a surjective map
	\[	\widehat{\O_{F,p}^\times/\overline{\O_F^{\times,+}}}\times\mathscr{W}_F'\rightarrow \mathscr{W}_F^{\rig}	\]
We say that a weight $\lambda\in \mathscr{W}_F^{\rig}(\overline\Q_p)$ is \emph{twist classical} if it is in the $\widehat{\O_{F,p}^\times/\overline{\O_F^{\times,+}}}(\overline\Q_p)$-orbit of a classical weight.  Then twist classical weights are very Zariski dense in $\mathscr{W}_F$.

In addition, we may define a twisting action on Hecke modules, as in ~\cite{bergdall-hansen}.  Let $\Gal_{F,p}$ denote the Galois group of the maximal abelian extension of $F$ unramified away from $p$ and $\infty$, and let $\eta:\Gal_{F,p}\rightarrow \overline\Q_p^\times$ be a continuous character.  Global class field theory implies that $\Gal_{F,p}$ fits into an exact sequence
	\[	1\rightarrow \O_{F,p}^\times/\overline{\O_F^{\times,+}}\rightarrow \Gal_{F,p}\rightarrow \mathrm{Cl}_F^+\rightarrow 1	\]
where $\mathrm{Cl}_F^+$ is the narrow class group of $F$ (and hence finite).  Suppose $M$ is an $R$-module equipped with an $R$-linear left $\Delta_p$-action.  Then we may define a new left $\Delta_p$-module $M(\eta):=M\otimes\eta^{-1}|_{\O_{F,p}^\times}$, where the action of $g\in\Delta_p$ is given by 
	\[	g\cdot m = \left(\eta^{-1}|_{\O_{F,p}^\times}(\det g \cdot p^{-\sum_{v\mid p}v(\det g)})\right)\cdot (g\cdot m)	\]
In particular, $\mathscr{D}_\kappa(\eta)\cong \mathscr{D}_{\eta^{-1}\cdot\kappa}$ by ~\cite[Lemma 5.5.2]{bergdall-hansen}, and there is an isomorphism
	\[	\mathrm{tw}_{\eta}:H_c^\ast(K,\mathscr{D}_\kappa)\xrightarrow\sim H_c^\ast(K,\mathscr{D}_{\eta^{-1}\cdot\kappa}	\]

Suppose $x\in\mathscr{X}_{\underline D^\times}(\overline\Q_p)$ is a point with $\mathrm{wt}(x)=:\lambda$, corresponding to the system of Hecke eigenvalues $\psi_x:\mathbb{T}\rightarrow \overline\Q_p$.  Then we define a new system of Hecke eigenvalues, via
\[	\mathrm{tw}_\eta(\psi_x)(T) = \begin{cases}
		\eta(\varpi_v)\psi_x(T) & \mbox{ if } v\nmid p\mathfrak{d}_D\mathfrak{n} \mbox{ and } T=T_v	\\
		\eta(\varpi_v)^2\psi_x(T) & \mbox{ if } v\nmid p\mathfrak{d}_D\mathfrak{n} \mbox{ and } T=S_v	\\
		\eta(\varpi_v)\psi_x(T) & \mbox{ if } v\mid p
\end{cases}	\]
Then it follows from ~\cite[Proposition 5.5.5]{bergdall-hansen} that $\mathrm{tw}_\eta(\psi_x)$ corresponds to a point $\mathrm{tw}_\eta(x)\in\mathscr{X}_{\underline D^\times}$ of weight $\eta^{-1}|_{\O_{F,p}^\times}\cdot\kappa$.

We say that a point $x\in \mathscr{X}_{\underline D^\times}(\overline\Q_p)$ is \emph{twist classical} if it is in the $\widehat{\Gal_{F,p}}(\overline\Q_p)$-orbit of a point corresponding to a classical system of Hecke eigenvalues.

\begin{prop}\label{prop: twist classical dense}
	Twist classical points are very Zariski dense in $\mathscr{X}_{\underline D^\times}$.
\end{prop}
\begin{proof}
Recall that $\mathscr{X}_{\underline D^\times}$ admits a cover by affinoid pseudorigid spaces of the form $\Spa \mathscr{T}(\mathscr{Z}_{U,h})$, where $\pi:\mathscr{Z}_{U,h}\rightarrow U$ is finite of constant degree, and 
	\[	\mathscr{T}(\mathscr{Z}_{U,h}) = \im\left(\O(\mathscr{Z}_{U,h})\otimes_{\Z_p}\mathbb{T}^p\rightarrow \End_{\O(\mathscr{Z}_{U,h})}(H_c^\ast(K,\mathscr{D}_U)_{\leq h}\right)	\]
We write $U=\Spa R$ for some pseudoaffinoid algebra $R$ over $\Z_p$.  We will show that $\Spec\mathscr{T}(\mathscr{Z}_{U,h})\rightarrow \Spec R$ carries irreducible components surjectively onto irreducible components, and we will construct a Zariski dense set of maximal points $W_{U,h}^{\mathrm{tw-cl}}\subset U$ such that the points of $\mathrm{wt}^{-1}(W_{U,h}^{\mathrm{tw-cl}})$ are twist classical.  By ~\cite[Lemme 6.2.8]{chenevier2004}, this implies the desired result.
	
To see that irreducible components of $\Spec\mathscr{T}(\mathscr{Z}_{U,h})$ map surjectively onto irreducible components of $\Spec R$, we observe that $D$ is totally definite, so the associated Shimura manifold is a finite set of points and $H_c^\ast(K,\mathscr{D}_U)$ vanishes outside degree $0$.  The base change spectral sequence of ~\cite[Theorem 4.2.1]{johansson-newton} implies that the formation of $H^0(K,\mathscr{D}_U)_{\leq h}$ commutes with arbitrary base change on $U$, which implies that $H^0(K,\mathscr{D}_U)_{\leq h}$ is flat.  Then ~\cite[Lemme 6.2.10]{chenevier2004} implies that $\Spec\mathscr{T}(\mathscr{Z}_{U,h})\rightarrow \Spec R$ carries irreducible components surjectively onto irreducible components, as desired.

Thus, it remains to construct $W_{U,h}^{\mathrm{tw-cl}}$.  Birkbeck proved a ``small slope implies classical'' result ~\cite[Theorem 4.3.7]{birkbeck2019}, and constructed a set $W_{U,h}^{\mathrm{cl}}$ Zariski dense in $U\cap \mathscr{W}_F'$ such that the points of $\mathrm{wt}^{-1}(W_{U,h}^{\mathrm{cl}})$ are classical (see the proof of ~\cite[Theorem 6.1.9]{birkbeck2019}).  Setting $W_{U,h}^{\mathrm{tw-cl}}$ to be the $\widehat{\O_{F,p}^\times/\overline{\O_F^{\times,+}}}(\overline\Q_p)$-orbit of  $W_{U,h}^{\mathrm{cl}}$, ~\cite[Lemma 6.3.1]{bergdall-hansen} implies that points of $\mathrm{wt}^{-1}(W_{U,h}^{\mathrm{tw-cl}})$ are twist classical, and we are done.
\end{proof}

As a corollary, we deduce that $\mathscr{X}_{\underline D^\times}$ has no components supported entirely in characteristic $p$:
\begin{cor}\label{cor: quat cpts in char p}
	$\mathscr{X}_{\underline D^\times}^{\rig}$ is Zariski dense in $\mathscr{X}_{\underline D^\times}$.
\end{cor}

We may use similar arguments to show that $\mathscr{X}_{\underline D^\times}$ is reduced:
\begin{prop}\label{prop: XD reduced}
	The eigenvariety $\mathscr{X}_{\underline D^\times}$ is reduced.
	\label{prop: quat eigenvariety reduced}
\end{prop}
\begin{proof}
	We first show that $\mathscr{X}_{\underline D^\times}^{\rig}$ is reduced.  By ~\cite[Proposition 6.1.2]{johansson-newton} (which adapts ~\cite[Proposition 3.9]{chenevier2005}), it is enough to find a Zariski dense set of twist classical weights $W_{U,h}^{\ssimple}\subset U\subset \mathscr{W}_F^{\rig}$ for each slope datum $(U,h)$ such that $\mathscr{M}(\mathscr{Z}_{U,h})_\kappa$ is a semi-simple Hecke module for all $\kappa\in W_{U,h}^{\ssimple}$.  Birkbeck ~\cite[Lemma 6.1.12]{birkbeck2019} constructed sets $W_{U,h}^{\prime,\ssimple}$ Zariski dense in $U\cap \mathscr{W}_F^{\prime,\rig}$ with this property, and we will again use twisting by $p$-adic characters to construct $W_{U,h}^{\ssimple}$.

	If $\eta:\O_{F,p}^\times/\overline{\O_F^{\times,+}}\rightarrow \overline\Q_p^\times$ is a character, we have an isomorphism 
	\[	\mathrm{tw}_\eta:H_c^\ast(K,\mathscr{D}_\kappa)\xrightarrow\sim H_c^\ast(K,\mathscr{D}_{\eta^{-1}\cdot\kappa}	\]
	By ~\cite[Proposition 5.5.5]{bergdall-hansen}, $\mathrm{tw}_\eta$ is Hecke-equivariant up to scalars, so $\mathscr{M}(\mathscr{Z}_{U,h})_{\kappa}$ is a semi-simple Hecke module if and only if $\mathscr{M}(\mathscr{Z}_{\eta^{-1}\cdot U,h})_{\eta^{-1}\cdot\kappa}$ is.  Thus, we may take $W_{U,h}^{\ssimple}$ to be the $\widehat{\O_{F,p}^\times/\overline{\O_F^{\times,+}}}(\overline\Q_p)$-orbit of $\cup_{U'}W_{U',h}^{\prime,\ssimple}$, as $(U',h)$ varies through slope data, and we see that $\mathscr{X}_{\underline D^\times}^{\rig}$ is reduced.

	Now let $X\subset \mathscr{X}_{\underline D^\times}$ be an open affinoid subspace, and let $\{X_i\}$ be an open affinoid cover of the rigid analytic locus $X^{\rig}\subset X$.  Since $X\smallsetminus X^{\rig}$ contains no open subset of $X$, the natural map
	\[	\O(X)\rightarrow \prod_i\O(X_i)	\]
	is injective.  Each $\O(X_i)$ is reduced, so $\O(X)$ is, as well.
\end{proof}

Now the Jacquet--Langlands correspondence for eigenvarieties follows immediately:
\begin{cor}
	There is a closed immersion $\mathscr{X}_{\underline D^\times}\hookrightarrow \mathscr{X}_{\GL_2/F}$ interpolating the classical Jacquet--Langlands correspondence on (twist) classical points, where the source has tame level $K_1^{\underline D^\times}(\mathfrak{n})$ and the target has tame level $K_1^{\GL_2/F}(\mathfrak{n})$.
\end{cor}

In particular, if $[F:\Q]$ is even, we can find $D$ split at all finite places and ramified at all infinite places.  Then we may take in particular $\mathfrak{n}=\O_F$ to obtain a morphism of eigenvarieties of tame level $1$.

\subsection{Cyclic base change}\label{section: cyclic base change}

Fix an integer $N\in\N$, and let $S$ be a finite set of primes containing every prime dividing $pN$.  For any number field $F$, we again let $K_F^p\subset\GL_2(\A_F)$ be the compact open subgroup given by
\[	K_F^p:=\left\{g\in\GL_2(\A_F)\mid g\equiv\left(\begin{smallmatrix}\ast & \ast \\ 0 & 1\end{smallmatrix}\right)\pmod N\right\}	\]
and we let $K_F:=K_F^pI$.  We also define the Hecke algebra
\[	\mathbb{T}_F^S:=\mathbb{T}_{\GL_2/F}^S:=\otimes_{v\notin S}\mathbb{T}(\GL_2(F_v),\GL_2(\O_{F_v}))	\]

There is a homomorphism $\sigma_F^S:\mathbb{T}_F^S\rightarrow \mathbb{T}_{\Q}^S$ induced by unramified local Langlands and restriction of Weil representations from $W_F$ to $W_\Q$.

Similarly, there is a morphism of weight spaces $\mathscr{W}_{\Q,0}\hookrightarrow\mathscr{W}_\Q\rightarrow\mathscr{W}_F$ induced by the norm map $T_{F,0}\rightarrow T_{\Q,0}$.

In the special case where $F/\Q$ is cyclic, the classical base change map produces cuspidal automorphic representations of $\GL_2(\A_F)$ from certain cuspidal automorphic representations of $\GL_2(\A_{\Q})$, and ~\cite{johansson-newton17} interpolated it to a morphism of eigenvarieties:
\begin{thm}[{\cite[Theorem 4.3.1]{johansson-newton17}}]\footnote{The authors only construct the morphism when $N\geq 5$, to maintain their running assumption that the level is actually neat (as opposed to containing an open neat subgroup with index prime to $p$).  However, the argument is identical for small $N$.}
	There is a finite morphism 
	\[	\mathscr{X}_{\GL_2/\Q,\mathrm{cusp},F\mathrm{ - ncm}}^{S}\rightarrow \mathscr{X}_{\GL_2/F}^S	\]
	lying over $\mathscr{W}_{\Q}\rightarrow \mathscr{W}_F$ and compatible with the homomorphism $\sigma_F^S$.
\end{thm}
Here the source includes only cuspidal components with a Zariski-dense set of forms without CM by an imaginary quadratic subfield of $F$.

Let $\Gal(F/\Q)=\left\langle \tau\right\rangle$. Then a cuspidal automorphic representation $\pi$ of $\GL_2(\A_F)$ is in the image of the base change map if and only if $\pi\circ\tau\cong\pi$.  In particular, the systems of Hecke eigenvalues of base-changed representations must be fixed by $\Gal(F/\Q)$.

This characterization of the image of the classical base change map permits us to prove automorphy lifting theorems by passing to a more convenient solvable extension.  We therefore wish to characterize the image of the interpolated base change map when $F$ is totally real and completely split at $p$ (so that the ``$F$ - ncm'' condition is vacuous).  We further assume that $[F:\Q]$ is prime to $p$.

We will study the ``$\Gal(F/\Q)$-fixed locus'' in the $\GL_{2/F}$-eigenvariety $\mathscr{X}_{\GL_2/F}^{S,\Gal(F/\Q)}$ and show that it is the image of the cyclic base change map; Xiang ~\cite{xiang2018} used a similar idea to construct $p$-adic families of essentially self-dual automorphic representations.

\begin{remark}
	We expect that it is possible to construct a base change morphism and characterize its image for more general cyclic extensions of number fields $F'/F$; however, for simplicity (and compatibility with ~\cite{johansson-newton17}) we have chosen to restrict to this setting.
\end{remark}

We first observe that $\Gal(F/\Q)$ acts on $\GL_{2/F}$, stabilizing $\T\subset\B$ and $I$, and also stabilizing the tame level $K_F^p$.  Next, observe that $\Gal(F/\Q)$ acts on $\mathbb{T}_F^S$ via $(\tau\cdot T)(g)=T(\tau^{-1}(g))$ for all $T\in \mathbb{T}_F^S$ and $g\in \GL_2(\A_{F,f})$.  Then for any $\delta\in\Delta$, $(\tau\cdot [K_F\delta K_F])(g) = [K_F\tau^{-1}(\delta)K_F](g)$; in particular, $\tau\cdot U_{\varpi_v}=U_{\tau(v)}$, and hence $\Gal(F/\Q)$ fixes $U_p$.  Similarly, we have an action of $\Gal(F/\Q)$ on $\mathscr{W}_\Q$ given via $(\tau\cdot\lambda)(g) = \lambda(\tau^{-1}(g))$;
the image of $\mathscr{W}_{\Q}$ in $\mathscr{W}_F$ is the diagonal locus, i.e., exactly the $\Gal(F/\Q)$-fixed locus.

Since $U_p$ is fixed by $\Gal(F/\Q)$, we see that if $\kappa$ is a weight fixed by $\Gal(F/\Q)$, then the Fredholm determinant $F_\kappa(T)$ of the action of $U_p$ on $C^\bullet(K_F,\mathscr{D}_\kappa)$ is fixed by $\Gal(F/\Q)$.  Thus, we have a spectral variety $\mathscr{Z}^{\Gal(F/\Q)}\subset \mathscr{W}_F^{\Gal(F/\Q)}\times \A^{1,\an}$ over $\mathscr{W}_F^{\Gal(F/\Q)}$.

\begin{lemma}
	Let $\kappa:T_0\rightarrow R^\times$ be a weight fixed by $\Gal(F/\Q)$.  There is an action of $\Gal(F/\Q)$ on $C^\bullet(K_F,\mathscr{D}_\kappa)$ and if $\mathscr{D}_\kappa$ admits a slope-$\leq h$ decomposition, the action of $\Gal(F/\Q)$ stabilizes $C^\bullet(K_F,\mathscr{D}_\kappa)_{\leq h}$.
	\label{lemma: gal action on coh}
\end{lemma}
\begin{proof}
	Referring to the definition of $\mathscr{D}_\kappa$ for an arbitrary weight $\kappa$, we have $\mathscr{D}_\kappa=\varprojlim\mathcal{D}_\kappa^r$, where $\mathcal{D}_\kappa^r$ is the completion of a module $\mathcal{D}_\kappa$ with respect to a norm $\|\cdot\|_r$.  The module $\mathcal{D}_\kappa$ itself is the continuous dual of the space $\mathcal{A}_\kappa\subset \mathcal{C}(I,R)$ of continuous functions $f:I\rightarrow R$ such that $f(gb)=\kappa(b)f(g)$ for all $g\in I$ and $b\in B_0$.  It follows that we have a map $\tau:\mathcal{A}_\kappa\rightarrow \mathcal{A}_{\tau(\kappa)}$ (since the action of $\Gal(F/\Q)$ preserves both $I$ and $B_0$).  If $\kappa$ is fixed by $\tau$, we obtain a dual action of $\Gal(F/\Q)$ on $\mathcal{D}_\kappa$, and hence $\mathcal{D}_\kappa^r$ and $\mathscr{D}_\kappa$.

	Since $K_F^p$ is also stable under the action of $\Gal(F/\Q)$ and the actions of $K_F^p$ and $\Gal(F/\Q)$ on $\mathscr{D}_\kappa$ commute, by functoriality we obtain an action of $\Gal(F/\Q)$ on $C^\bullet(K_F,\mathscr{D}_\kappa)$.  Moreover, the action of $\Gal(F/\Q)$ fixes the Hecke operator $U_p$, so ~\cite[Proposition 2.2.11]{johansson-newton} implies that the action of $\Gal(F/\Q)$ also preserves $C^\bullet(K_F,\mathscr{D}_\kappa)_{\leq h}$.
\end{proof}

\begin{lemma}
	Let $\kappa:T_0\rightarrow R^\times$ be a weight fixed by $\Gal(F/\Q)$.  For any $T\in\mathbb{T}_F^S$, we have $\tau\cdot T = \tau\circ T\circ \tau^{-1}$ as operators on $C^\bullet(K_F,\mathscr{D}_\kappa)$.
	\label{lemma: twisted action on cohomology}
\end{lemma}
\begin{proof}
	We may assume $T=[K_F\delta K_F]$ for some $\delta\in\Delta$.  Then $\tau\cdot [K_F\delta K_F] = [K_F\tau(\delta)K_F]$, and the corresponding morphism 
	\[	C^\bullet(K_F,\mathscr{D}_\kappa) \rightarrow C^\bullet(\tau(\delta) K_F\tau(\delta)^{-1}, \mathscr{D}_\kappa)	\]
	is induced by the conjugation map $\tau(\delta)K_F\tau(\delta)^{-1}\rightarrow K_F$ and the map $\mathscr{D}_\kappa\rightarrow \mathscr{D}_\kappa$ given by $d\mapsto \tau(\delta)\cdot d$.  But $\tau(\delta) K_F\tau(\delta)^{-1} = \tau\left(\delta \tau^{-1}(K_F) \delta^{-1}\right)$, so we may factor the conjugation map as
	\[	\tau(\delta) K_F\tau(\delta)^{-1}\xrightarrow{\tau^{-1}} \delta \tau^{-1}(K_F) \delta^{-1} \rightarrow \tau^{-1}(K_F) \xrightarrow{\tau} K_F 	\]
	Similarly, $d\mapsto \tau(\delta)\cdot d$ factors as $\tau\circ T\circ \tau^{-1}$, so our morphism of complexes also factors as desired.
\end{proof}

We may restrict $\mathscr{M}_c^\ast$ to $\mathscr{Z}^{\Gal(F/Q)}$; we denote this restriction by $\mathscr{H}^\ast$ and by abuse of notation, we again use $\mathscr{T}$ to denote the sheaf generated by the image of $\mathbb{T}_F^S$ in $\mathscr{E}nd_{\mathscr{Z}^{\Gal(F/\Q)}}\left(\mathscr{H}^\ast\right)$.  Then the slice of the eigenvariety $\mathscr{X}_{\GL_2/F}^S$ over $\mathscr{W}_F^{\Gal(F/\Q)}$ is, by definition, $\underline\Spa \mathscr{T}$.

Both $\mathbb{T}(\Delta^p,K_F^p)$ and $\End_{\O(V)}\left(\mathscr{H}^\ast\right)$ have actions of $\Gal(F/\Q)$, and Lemma~\ref{lemma: twisted action on cohomology} implies that they are compatible.  Thus, $\mathscr{T}(V)$ and $\mathscr{X}_{\GL_2/F}^S|_{\mathscr{W}_F^{\Gal(F/\Q)}}$ have actions of $\Gal(F/\Q)$.

The subspace of $\mathscr{X}_{\GL_2/F}$ fixed by $\Gal(F/\Q)$ corresponds to the sheaf $V\mapsto \mathscr{T}(V)_{\Gal(F/\Q)}$ of co-invariants of $\mathscr{T}$.  Since $\Gal(F/\Q)$ is a finite group with order prime to $p$, $\mathscr{T}(V)_{\Gal(F/\Q)}$ is a $\O_{\mathscr{W}_F}(V)$-linear direct summand of $\mathscr{T}(V)$.

The above discussion gives us a closed subspace $\mathscr{X}_{\GL_2/F}^{S,\Gal(F/\Q)}\hookrightarrow \mathscr{X}_{\GL_2/F}^S$.



	We let 
	\[	\mathscr{X}_{\GL_2/F}^{S, \Gal(F/\Q),\circ}:= \mathscr{X}_{\GL_2/F}^{S,\Gal(F/\Q)}\cap \mathscr{X}_{\GL_2/F,\mathrm{mid}}^S	\]
	and we let $\overline{\mathscr{X}_{\GL_2/F}^{S, \Gal(F/\Q),\circ}}$ denote its Zariski closure in $\mathscr{X}_{\GL_2/F}^S$.
	\begin{lemma}
		Classical points are very Zariski dense in $\mathscr{X}_{\GL_2/F}^{S, \Gal(F/\Q),\circ}$.
		\label{lemma: gal fixed zariski dense}
	\end{lemma}
	\begin{proof}
		If $(U,h)$ is a slope datum and $W\subset \mathscr{X}_{\GL_2/F}^S$ is a connected affinoid subspace of the pre-image of $U$, then $\mathscr{T}(W)=e_W\mathscr{T}(U)$ and $\mathscr{M}_c^\ast(W)\cong e_WH_c^\ast(K,\mathscr{D}_U)_{\leq h}$, where $e_W$ is the idempotent projector to $W$.  If $W\subset \mathscr{X}_{\GL_2/F,\mathrm{mid}}^S$, then $\mathscr{M}_c^\ast\cong e_WH_c^d(K,\mathscr{D}_U)_{\leq h}$ and $H_c^d(K,\mathscr{D}_U)_{\leq h}$ is a projective $\O_{\mathscr{W}_F}(U)$-module.  It follows that the restriction of $\mathscr{M}_c^\ast$ to $\mathscr{X}_{\GL_2/F}^{S, \Gal(F/\Q),\circ}$ is a vector bundle over $U$.

		By ~\cite[Lemme 6.2.10]{chenevier2004}, $\mathscr{T}(U)$ is $\O_{\mathscr{W}_F}(U)$-torsion-free, and it remains torsion-free after any flat base change on $U$. Since $\mathscr{T}(U)_{\Gal(F/\Q)}$ is a direct summand of $\mathscr{T}(U)$, the same property holds for the co-invariants.  

		Now we may apply ~\cite[Lemme 6.2.10]{chenevier2004} again to conclude that $W\cap \mathscr{X}_{\GL_2/F}^{S,\Gal(F/\Q),\circ}$ is equidimensional of dimension $\dim \O_{\mathscr{W}_F^{\Gal(F/\Q)}}(U)$, and every irreducible component of $W\cap \mathscr{X}_{\GL_2/F}^{S,\Gal(F/\Q),\circ}$ surjects onto an irreducible component of $\Spec \O_{\mathscr{W}_F^{\Gal(F/\Q)}}(U)$.  

		If $x\in W\cap \mathscr{X}_{\GL_2/F}^{S,\Gal(F/\Q),\circ}$ has a classical weight that is sufficiently large (where ``sufficiently large'' depends on $h$), then $x$ corresponds to a classical Hilbert modular form.  But sufficiently large classical weights are very Zariski dense in $\mathscr{W}_F^{\Gal(F/\Q)}$, so ~\cite[Lemme 6.2.8]{chenevier2004} implies that classical points are very Zariski dense in $\mathscr{X}_{\GL_2/F}^{S,\Gal(F/\Q),\circ}$.
	\end{proof}

	\begin{remark}
		The proof Lemma~\ref{lemma: gal fixed zariski dense} is the only time we use our assumption that $\lvert \Gal(F/\Q)\rvert$ is prime to $p$.  If we restricted to the rigid analytic locus (where $p$ is invertible, so that $\mathscr{T}_{\Gal(F/\Q}$ is unconditionally a direct summand of $\mathscr{T}$, this assumption would be unnecessary.
	\end{remark}

	\begin{cor}\label{cor: image cyclic base change}
		The image of the cyclic base change morphism in $\mathscr{X}_{\GL_2/F,\mathrm{mid}}^S$ is exactly $\overline{\mathscr{X}_{\GL_2/F}^{S, \Gal(F/\Q),\circ}}$.
	\end{cor}
	\begin{proof}
		Since the cyclic base change morphism $\mathscr{X}_{\GL_2/\Q,\mathrm{cusp}}\rightarrow \mathscr{X}_{\GL_2/F}$ is finite, it has closed image.  Moreover, cyclic base change carries any classical point of $\mathscr{X}_{\GL_2/\Q,\mathrm{cusp}}$ to a point of $\mathscr{X}_{\GL_2/F}^{S, \Gal(F/\Q),\circ}$; since $\mathscr{X}_{\GL_2/F}^{S, \Gal(F/\Q),\circ}$ is closed in $\mathscr{X}_{\GL_2/F}$, it contains the entire image of the cyclic base change morphism.  On the other hand, every classical point of $\mathscr{X}_{\GL_2/F}^{S, \Gal(F/\Q),\circ}$ is in the image of cyclic base change, by the classical theorem, so Lemma~\ref{lemma: gal fixed zariski dense} implies the desired result.
	\end{proof}

	\subsection{Galois representations}\label{subsect: galois reps}

In ~\cite[\textsection 5.4]{johansson-newton}, the authors construct families of Galois determinants (in the sense of ~\cite{chenevier2014}) over the eigenvarieties $\mathscr{X}_{\mathbf{G}}$ when $\mathbf{G}=\Res_{F/\Q}\GL_n$ and $F$ is totally real or CM, and prove that they satisfy local-global compatibility at places away from $p$ and the level.  Then the Jacquet--Langlands correspondence lets us deduce the following:

\begin{thm}\label{thm: galois det on eigenvariety}
	Let $D$ be a quaternion algebra over a totally real field $F$, such that $F$ is totally split at $p$ and $D$ is split at all places above $p$. Let $K=K^pI\subset (\A_{F,f}\otimes D)^\times$ be the level, and let $S$ be the set of finite places $v$ of $F$ for which $D$ is ramified or $K_v\neq \GL_2(\O_{F_v})$.  Then there is a continuous $2$-dimensional Galois determinant $D:\Gal_{F,S}\rightarrow \O(\mathscr{X}_{\underline D^\times})^+$ such that
	\[	D(1-X\cdot{\Frob}_v) = P_v(X)	\]
	for all $v\notin S$, where $P_v(X)$ is the standard Hecke polynomial.  

	Moreover, if $v\mid p$ then for every maximal point $x\in \mathscr{X}_{\underline D^\times}$ of weight $\kappa_x=(\kappa_{x,1},\kappa_{x,2})$, we let $\psi:\O(\mathscr{X}_{\underline D^\times})^+\rightarrow k(x)^+$ denote the corresponding specialization map.  Then there is a proper Zariski-closed subspace $Z\subset \mathscr{X}_{\underline D^\times}$ such that for $x\notin Z$, the Galois representation corresponding to $D_x|_{\Gal_{F,v}}$ is trianguline with parameters $\delta_1,\delta_2:F_v^\times\rightrightarrows k(v)^\times$, where 
	\[	\delta_1|_{\O_{F_v}^\times} = \kappa_{x,2}^{-1}|_{\O_v^\times}\text{ and }\delta_1(\varpi_v) = \psi(U_{\varpi_v})	\]
and
\[	\delta_2|_{\O_{F_v}^\times} = (\kappa_{x,1}|_{\O_v^\times}\chi_{\cyc})^{-1}\text{ and }\delta_2(\varpi_v) = \psi(I_v\left(\begin{smallmatrix}\varpi_v & \\ & 1\end{smallmatrix}\right)I_v)	\]
\end{thm}
\begin{proof}
	It only remains to check local-global compatibility at places above $p$.  But this is true for non-critical classical points by work of Saito, Blasius--Rogawski, and Skinner, and it is true for twists of those classical points by the definition of twisting.  Then the result follows from ~\cite[Corollary 6.3.10]{kpx} and ~\cite[Theorem 6.8]{bellovin2021}.
\end{proof}

\begin{remark}
	For each point $x\in \mathscr{X}_{\underline D^\times}$, there is a residual Galois determinant $\overline D_x$ valued in a finite field.  These residual Galois determinants are constant on each connected component of $\mathscr{X}_{\underline D^\times}$, as a consequence of ~\cite[Lemma 3.10]{chenevier2014}.
	\label{rmk: residual local constancy}
\end{remark}

As in Corollary~\ref{cor: interpolation integral}, we can make a sharper local statement in the presence of integral structures.  Suppose $\kappa:T_0/\overline{Z(K)}\rightarrow R^\times$ is a weight, where $R$ is a pseudoaffinoid algebra equipped with a norm adapted to $\kappa$, and $R_0\subset R$ is the corresponding unit ball (so in particular, $\kappa$ takes values in $R_0$).  If $(\Spa R,h)$ is a slope datum, for any $r>r_\kappa$ we define
\[	H^0(K,\mathcal{D}_\kappa^{<r})_{\leq h}:=\im\left( H^0(K,\mathcal{D}_\kappa^{<r,\circ})\rightarrow H^0(K,\mathcal{D}_\kappa^{<r})_{\leq h} \right)	\]
	and
	\[	\mathbb{T}_{\kappa,\leq h}^{<r,\circ}:=\im\left( R_0\otimes\mathbb{T}\rightarrow \End_{R_0}(H^0(K,\mathcal{D}_\kappa^{<r,\circ})_{\leq h}) \right)	\]

	\begin{cor}
		With hypotheses and notation as above, there is a $2$-dimensional Galois determinant $D_0:\Gal_{F,S}\rightarrow\mathbb{T}_{\kappa,\leq h}^{<r,\circ,\red}$ such that 
		\[	R^\circ\otimes_{R_0}D_0=R^\circ\otimes_{\O(\mathscr{X}_{\underline D^\times})^+}D	\]
	\end{cor}
	\begin{proof}
		This is a corollary of the construction of ~\cite[\textsection 5.4]{johansson-newton}, rather than of Theorem~\ref{thm: galois det on eigenvariety}.  For each maximal point $x\in \Spa R$ with residue field $L$ and ring of integers $\O_L$, let $\kappa_x$ be the composition of $\kappa$ with $R_0\rightarrow \O_L$. By ~\cite[Corollary 5.3.2(2)]{johansson-newton} combined with Corollary~\ref{cor: interpolation integral}, there is a $2$-dimensional Galois determinant $D_x:\Gal_{F,S}\rightarrow \mathbb{T}_{\kappa_x,\leq h}^{<r,\circ,\mathrm{red}}$ valued in the reduced quotient of $\mathbb{T}_{\kappa_x,\leq h}^{<r,\circ}$.  We have an injection
		\[	\mathbb{T}_{\kappa,\leq h}^{<r,\circ,\red}\hookrightarrow \prod_x \mathbb{T}_{\kappa_x,\leq h}^{<r,\circ,\mathrm{red}}	\]
		where the $x$ range over maximal points of $\Spa R$.  The ring $\mathbb{T}_{\kappa,\leq h}^{<r,\circ,\red}$ is compact since it is a finite $R_0$-module, so by ~\cite[Example 2.3.2]{chenevier2014} the $\mathbb{T}_{\kappa_x,\leq h}^{<r,\circ,\mathrm{red}}$-valued determinants glue to $D_0$.
	\end{proof}

	\subsection{Quaternionic sub-eigenvarieties}\label{subsect: sub-eigenvarieties}

	In order to study suitable spaces of overconvergent quaternionic modular forms, we will need to define and study eigenvarieties parametrizing quaternionic modular forms with certain auxiliary data fixed.
	We let $F$ be a totally real number field totally split at $p$, and we let $D$ be a totally definite quaternion algebra over $F$, split at all places above $p$.  We fix a level $K\subset \left(\A_{F,f}\otimes_F D\right)^\times$ and monoid $K\subset \Delta\subset \left(\A_{F,f}\otimes_F D\right)^\times$, and we set $\mathbb{T}$ to be either $\mathbb{T}(\Delta^p,K^p)$ or $\mathbb{T}(\Delta,K)$.

In order to construct an eigenvariety for $\underline D$, we fixed a Borel--Serre complex $C_c^\bullet(K,-)$ and we considered the cohomology $C_c^\bullet(K,\mathscr{D}_\kappa)$.  However, because we assumed $D$ is totally definite, the associated Shimura manifold is a finite set of points, and so the cohomology of $C_c^\bullet(K,-)=C^\bullet(K,-)$ vanishes outside of degree $0$.  

Thus, we can give an extremely concrete description of the automorphic forms of interest to us and of the Hecke operators acting on them.  Suppose that $M$ is a left $R[\Delta]$-module, for some pseudoaffinoid algebra $R$.  Then if $f:D^\times\backslash(\A_{F,f}\otimes_F D)^\times\rightarrow M$ is a function and $\gamma\in \Delta$, we define ${}_\gamma|f$ via ${}_\gamma|f(g) = \gamma \cdot f(g\gamma)$.  Then
\[	H^0(K,M) = \left\{f:D^\times\backslash(\A_{F,f}\otimes_F D)^\times\rightarrow M \mid {}_\gamma|f=f\text{ for all } \gamma\in K\right\}	\]
We can describe the Hecke operator $[KgK]:H^0(K,M)\rightarrow H^0(K,M)$ explicitly for any $g\in \Delta$; we decompose the double coset $KgK=\coprod_i g_iK$ as a finite disjoint union of cosets, and we have
\[	[KgK]f:=\sum_i{{}_{g_i}|}f	\]

The first piece of auxiliary data we want to fix is the central character.  If $\xi:\A_{F,f}^\times/F^\times\rightarrow R_0^\times$ is a continuous character such that $\xi|_{K_v\cap \O_{F_v}^\times}$ agrees with the action of $K_v\cap \O_{F_v}^\times$ on $M$ for all finite places $v$ of $F$, we may extend the action of $K$ on $M$ to an action of $K\cdot\A_{F,f}^\times$, by letting $\A_{F,f}^\times$ act by $\xi$.  Then we define
\[	H^0(K,M)[\xi]:= \{f\in H^0(K,M) \mid {}_z|f = f \text{ for all }z\in\A_{F,f}^\times\}	\]
If we write $D^\times\backslash (\A_{F,f}\otimes_F D)^\times/K=\coprod_{i\in I}D^\times g_i K\A_{F,f}^\times$ for some finite set of elements $g_i\in (\A_{F,f}\otimes_F D)^\times$, the natural map
\begin{align*}
	H^0(K,M)[\xi] &\rightarrow \oplus_{i\in I}M^{(K\A_{F,f}^\times\cap g_i^{-1}D^\times g_i)/F^\times}	\\
	f &\mapsto (f(g_i))
\end{align*}
is an isomorphism.

The calculations of ~\cite[Lemma 1.1]{taylor2006} show that $(K\A_{F,f}^\times\cap g_i^{-1}D^\times g_i)/F^\times$ is a finite group with order prime to $p$ for all $i$ (since we assumed $p\neq 2$).  Thus, if $M$ is a potentially orthonormalizable Banach $R$-module, then so is $H^0(K,M)[\xi]$, and we will be able to apply the formalism of slope decompositions to quaternionic modular forms with fixed central character.  More precisely, we may consider the action of a compact operator $U$ on $H^0(K,M)[\xi]$. If $H^0(K,M)[\xi]$ admits a slope-$\leq h$-decomposition, then $H^0(K,M)[\xi]_{\leq h}$ is a finite $R$-module which is a direct summand of $H^0(K,M)[\xi]$.  Since $H^0(K,M)[\xi]$ is potentially orthonormalizable, $H^0(K,M)[\xi]_{\leq h}$ satisfies the property (Pr) of ~\cite{buzzard} and by ~\cite[Lemma 2.11]{buzzard} it is actually projective as an $R$-module.

The coefficient modules of interest to us are the modules of distributions $\mathscr{D}_\kappa$ constructed in ~\cite{johansson-newton}, and we fix a character $\xi:\A_{F,f}^\times/F^\times\rightarrow \Z_p[\![T_0/\overline{Z(K)}]\!]^\times$ as above.  The operator $U_p$ commutes with the action of $\A_{F,f}^\times/F^\times$ on $\mathscr{D}_\kappa$ given by $\xi$, so $U_p$ acts compactly on $C^\ast(K,\mathscr{D}_\kappa)[\xi]$.  We may construct a corresponding spectral variety $\mathscr{Z}_\xi$ and eigenvariety datum $(\mathscr{Z}_\xi, \mathscr{M}_{\xi},\mathbb{T},\psi)$, where $\mathscr{M}_{\xi}$ is the coherent sheaf on $\mathscr{Z}_\xi$ coming from factorizations of the characteristic power series of $U_p$; we write $\mathscr{X}_{\underline D^\times,\xi}$ for the corresponding eigenvariety.

By construction, $H^0(K,\mathscr{D}_\kappa)[\xi]_{\leq h}$ is a projective $R$-module whenever $(U,h)$ is a slope datum. Then ~\cite[Lemme 6.2.10]{chenevier2004} implies that if $\mathscr{M}_{\xi}$ is non-zero, $\mathscr{X}_{\underline D^\times,\xi}$ is equidimensional of the same dimension as $\mathscr{W}_F$.

Moreover, for each maximal point $x\in \mathscr{X}_{\underline D^\times,\xi}$, the corresponding Hecke eigensystem appears in $\mathscr{X}_{\underline D^\times}$ (with unrestricted central character), by construction.  Then the interpolation theorem ~\cite[Theorem 3.2.1]{johansson-newton17} implies that there is a closed immersion $\mathscr{X}_{\underline D^\times,\xi}^{\red}\hookrightarrow \mathscr{X}_{\underline D^\times}$, and dimension considerations imply that its image is a union of irreducible components of $\mathscr{X}_{\underline D^\times}$.  

This implies in particular that as $(U,h)$ runs over slope data for $C^\ast(K,\mathscr{D}_\kappa)[\xi]$, the sets $W_{U,h}^{\prime,\mathrm{ss}}\subset U$ of semi-simple weights constructed in Proposition~\ref{prop: XD reduced} are Zariski dense.  Then we may repeat the argument of that proposition to conclude that $\mathscr{X}_{\underline D^\times,\xi}$ is itself reduced.

We have shown the following:
\begin{prop}
	Given a character $\xi: \A_{F,f}^\times/F^\times\rightarrow \O(\mathscr{W}_F)^{\times}$ as above, there is an eigenvariety $\mathscr{X}_{\underline D^\times,\xi}$ of quaternionic modular forms with central character $\xi$.  It is reduced and equidimensional, and it is naturally identified as a (possibly empty) union of irreducible components of $\mathscr{X}_{\underline D^\times}$.
\end{prop}

We also wish to introduce eigenvarieties localized at maximal ideals of Hecke algebras.  Let $\mathfrak{m}\subset \mathbb{T}$ be a maximal ideal.  By Theorem~\ref{thm: galois det on eigenvariety} and Remark~\ref{rmk: residual local constancy}, the residual Hecke eigenvalues are locally constant on $\mathscr{X}_{\underline D^\times}$. It follows that the restrictions $\mathscr{M}_{\mathfrak{m}}$ and $\mathscr{M}_{\xi,\mathfrak{m}}$ are supported on unions of connected components of $\mathscr{Z}$, which we write $\mathscr{Z}_\mathfrak{m}$ and $\mathscr{Z}_{\xi,\mathfrak{m}}$, respectively.
In particular, if $(U,h)$ is a slope datum, then $H^0(K,\mathscr{D}_U)_{\leq h,\mathfrak{m}}$ and $H^0(K,\mathscr{D}_U)_{\leq h,\xi,\mathfrak{m}}$ are again finite projective $\O(U)$-modules.  Then an identical argument shows the following:
\begin{prop}
	Given a character $\xi: \A_{F,f}^\times/F^\times\rightarrow \O(\mathscr{W}_F)^{\times}$ as above and a maximal ideal $\mathfrak{m}\subset \mathbb{T}$ as above, for any choice of Hecke algebra $\mathbb{T}'$ (possibly different from $\mathbb{T}$) there are eigenvarieties $\mathscr{X}_{\underline D^\times,\mathfrak{m}}^{\mathbb{T}'}$ and $\mathscr{X}_{\underline D^\times,\xi,\mathfrak{m}}^{\mathbb{T}'}$ of quaternionic modular forms localized at $\mathfrak{m}$.  They are reduced and equidimensional, and they are naturally identified as (possibly empty) unions of connected components of $\mathscr{X}_{\underline D^\times}^{\mathbb{T}'}$.
\end{prop}

\begin{remark}
	We write $h=m/n$ and consider the closed ball $\mathbb{B}_{U,h}:=\{\lvert T^n\rvert\leq \lvert u^{-m}\rvert\}\subset\mathbb{A}_U^1$ for some open affinoid $U\subset \mathscr{W}_F$.  Setting $Z_{U,h}:=\mathscr{Z}_{\mathfrak{m}}\cap \mathbb{B}_{U,h}$ (resp. $Z_{U,h}:=\mathscr{Z}_{\xi,\mathfrak{m}}\cap \mathbb{B}_{U,h}$), we abuse terminology slightly and say that $(U,h)$ is a slope datum for $\mathscr{X}_{\underline D^\times,\mathfrak{m}}$ (resp. $\mathscr{X}_{\underline D^\times,\xi,\mathfrak{m}}$) if $Z_{U,h}\rightarrow U$ is finite of constant degree.
	\label{}
\end{remark}

	\section{Overconvergent quaternionic modular forms}\label{section: modular forms}

\subsection{Definitions}\label{subsect: oc quat definitions}

We will use overconvergent cohomology to define and study spaces of overconvergent quaternionic modular forms.  Maintaining our notation from \textsection~\ref{subsection: eigenvariety definitions}, and in particular \textsection~\ref{subsect: sub-eigenvarieties}, we fix a level $K\subset \left(\A_{F,f}\otimes_F D\right)^\times$ and monoid $K\subset \Delta\subset \left(\A_{F,f}\otimes_F D\right)^\times$, and we set $\mathbb{T}$ to be either $\mathbb{T}(\Delta^p,K^p)$ or $\mathbb{T}(\Delta,K)$.

The coefficients for our families of overconvergent modular forms will be a pseudoaffinoid algebra $R$ over $\Z_p$; we set $U:=\Spa R$.  We also fix a pseudouniformizer $u\in R$.  If $\kappa:T_0/\overline{Z(K)}\rightarrow R^\times$ is a weight, we choose a norm $\lvert\cdot\rvert$ on $R$ so that $\lvert\cdot\rvert$ is adapted to $\kappa$ and multiplicative with respect to $u$, and $\log_p\lvert\cdot\rvert$ is discrete (which we may do, by Lemma~\ref{lemma: adapted discrete norm} below).  Then the unit ball $R_0\subset R$ is a ring of definition containing $u$.  

Fix some $r\geq r_\kappa$.  We let $\mathcal{D}_\kappa^{r, \circ}\subset \mathcal{D}_\kappa^r$ denote the unit ball, and we also consider larger modules of distributions $\mathcal{D}_\kappa^{<r}\supset \mathcal{D}_\kappa^r$, with unit ball $\mathcal{D}_\kappa^{<r,\circ}\subset \mathcal{D}_\kappa^{<r}$.  Following~\textsection\ref{subsect: sub-eigenvarieties}, we also fix a character $\xi:\A_{F,f}^\times/F^\times\rightarrow R^\times$ such that $\xi|_{K_v\cap \O_{F_v}^\times}$ agrees with the action of $K_v\cap \O_{F_v}^\times$ on $\mathcal{D}_\kappa^r$, that is, such that $\xi|_{K_v\cap \O_{F_v}^\times}$ is trivial for $v\nmid p$ and $\xi|_{I_v\cap\O_{F_v}^\times}$ is equal to the action of $I_v\cap\O_{F_v}^\times$ on $\mathcal{D}_\kappa^r$ for $v\mid p$.

The construction of the required norm on $R$ is a variant of ~\cite[Lemma 3.3.1]{johansson-newton}, and we refer to that paper for the terminology:
\begin{lemma}
	If $R$ is a pseudoaffinoid algebra over $\Z_p$ and $\kappa:T_0/\overline{Z(K)}$ is a weight, there is a norm $\lvert\cdot\rvert$ on $R$ such that $\lvert\cdot\rvert$ is adapted to $\kappa$ and multiplicative with respect to $u$, the unit ball $R_0$ is noetherian, and $\log_p\lvert\cdot\rvert$ is discrete.
	\label{lemma: adapted discrete norm}
\end{lemma}
\begin{proof}
	Choose a noetherian ring of definition $R_0\subset R$ formally of finite type over $\Z_p$.  As in the proof of ~\cite[Lemma 3.3.1]{johansson-newton}, $\kappa(T_0)\subset R^\circ$ and $\kappa(T_\epsilon)\subset 1+R^{\circ\circ}$; since both groups are topologically finitely generated, we may replace $R_0$ with a finite integral extension and assume that $\kappa(T_0)\subset R_0$, and we may find some integer $m\geq 1$ so that $\kappa(T_\epsilon)^m\subset 1+uR_0$.

	Let $R':=R[u^{1/m}]$, let $R_0':=R_0[u^{1/m}]$, and let $u':=u^{1/m}$.  Then $R'$ is a finite $R$-module, so it has a canonical topology, and the subspace topology it induces on $R$ agrees with the original topology on $R$.  Now for any $a\in \R_{>1}$ we may define a norm $\lvert\cdot\rvert'$ on $R'$ via
	\[	\lvert r' \rvert' = \inf\{a^s\mid {u'}^sr'\in R_0'\}	\]
	The restriction of $\lvert\cdot\rvert'$ to $R$ has the desired properties.
\end{proof}

When $U$ is a subspace of $\mathscr{W}_F$, we can make a more precise statement.  In this case, $R$ is reduced, so the ring of power-bounded elements $R_0:=R^\circ$ is a ring of definition.  Then we may define a norm $\lvert\cdot\rvert$ on $R$ via
\[	\lvert r\rvert:= \inf \{p^{-n}\mid r\in u^nR_0, n\in\Z\}	\]
\begin{lemma}
	If $U$ is a rational subspace of $\mathscr{W}_F$ and $\kappa$ is the restriction of the universal character on $\mathscr{W}_F$, then $\lvert\cdot\rvert$ is adapted to $\kappa$.
	\label{lemma: norm adapted univ char}
\end{lemma}
The proof is essentially identical to that of ~\cite[Lemma 6.3.1]{johansson-newton}.

	Recall that we have Fredholm power series 
\[	F_\kappa:=\det\left(1-TU_p\mid H^0(K,\mathcal{D}_\kappa^r)\right)	\]
and
\[	F_{\kappa,\xi}:=\det\left(1-TU_p\mid H^0(K,\mathcal{D}_\kappa^r)[\xi]\right)	\]
and they are independent of $r\geq r_\kappa$, by ~\cite[Proposition 4.1.2]{johansson-newton}.  

If $H^0(K,\mathcal{D}_\kappa^r)$ (resp. $H^0(K,\mathcal{D}_\kappa^r)[\xi]$) admits a slope $\leq h$-factorization, then the formalism of slope decompositions implies that we have a decomposition 
\[	H^0(K,\mathcal{D}_\kappa^r)=H^0(K,\mathcal{D}_\kappa^r)_{\leq h}\oplus H^0(K,\mathcal{D}_\kappa^r)_{>h}	\]
resp.
\[	H^0(K,\mathcal{D}_\kappa^r)[\xi] = H^0(K,\mathcal{D}_\kappa^r)[\xi]_{\leq h}\oplus H^0(K,\mathcal{D}_\kappa^r)[\xi]_{>h}	\]
for all $r\geq r_\kappa$, and the decomposition is independent of $r$.

Moreover, if $r'\in [r_\kappa,r)$, the inclusions
\[	\mathcal{D}_\kappa^r\subset \mathcal{D}_\kappa^{<r}\subset \mathcal{D}_\kappa^{r'}	\]
induce an isomorphism $H^0(K,\mathcal{D}_{\kappa}^r)_{\leq h}\xrightarrow{\sim}H^0(K,\mathcal{D}_{\kappa}^{r'})_{\leq h}$.  We may therefore define
\[	H^0(K,\mathcal{D}_\kappa^{<r})_{\leq h}:=\im\left(H^0(K,\mathcal{D}_\kappa^r)_{\leq h}\rightarrow H^0(K,\mathcal{D}_\kappa^{<r})\right)	\]
and
\[	H^0(K,\mathcal{D}_\kappa^{<r})[\xi]_{\leq h}:=\im\left(H^0(K,\mathcal{D}_\kappa^r)[\xi]_{\leq h}\rightarrow H^0(K,\mathcal{D}_\kappa^{<r})[\xi]\right)	\]

We make the additional definitions
\[	H^0(K,\mathcal{D}_\kappa^{<r, \circ})_{\leq h}:=\im\left(H^0(K,\mathcal{D}_\kappa^{<r,\circ})\rightarrow H^0(K,\mathcal{D}_\kappa^{<r})\rightarrow H^0(K,\mathcal{D}_{\kappa}^{<r})_{\leq h}\right)	\]
and
\[	H^0(K,\mathcal{D}_\kappa^{<r, \circ})[\xi]_{\leq h}:=\im\left(H^0(K,\mathcal{D}_\kappa^{<r,\circ})[\xi]\rightarrow H^0(K,\mathcal{D}_\kappa^{<r})[\xi]\rightarrow H^0(K,\mathcal{D}_{\kappa}^{<r})[\xi]_{\leq h}\right)	\]

We are now in a position to define spaces of overconvergent quaternionic modular forms, together with an integral structure and Hecke algebras:
	\begin{definition}
		Suppose that $H^0(K,\mathcal{D}_\kappa^{<r})$ admits a slope-$\leq h$-decomposition, where $h=a/b$ for $a,b$ positive and relatively prime integers.  We define the modular forms of weight $\kappa$ and slope-$\leq h$ to be the module 
		\[	S_\kappa(K)_{\leq h}:=H^0(K,\mathscr{D}_\kappa)_{\leq h}; 	\]
		it is a module over the Hecke algebra 
		\[	\mathbb{T}_{\kappa,\leq h}:=\im\left(\mathbb{T}\otimes_{\Z_p}R\rightarrow \End_R(S_\kappa(K)_{\leq h})\right)	\]

		We define two modules of integral overconvergent modular forms (and corresponding Hecke algebras). As in \textsection~\ref{subsect: galois reps}, we set
		\[	S_{\kappa}^{<r,\circ}(K)_{\leq h}:= H^0(K,\mathcal{D}_\kappa^{<r,\circ})_{\leq h}	\]
		and
		\[	\mathbb{T}_{\kappa,\leq h}^{<r,\circ}:=\im\left(\mathbb{T}\otimes_{\Z_p}R_0\rightarrow\End_{R_0}\left(H^0(K,\mathcal{D}_\kappa^{<r,\circ})_{\leq h}\right)\right)       \]
		We also define a second lattice
		\[      S_{\kappa}^{\circ}(K)_{\leq h}:= \im\left(\mathbb{T}[\{u^aU_{\varpi_v}^{-b}\}_{v\mid p}]\otimes_{\mathbb{T}} S_{\kappa}^{<r,\circ}(K)_{\leq h} \rightarrow S_{\kappa}(K)_{\leq h}\right)     \]
		which is stable under the operators $u^aU_{\varpi_v}^{-b}$, as well; we set
		\[	\mathbb{T}_{\kappa,\leq h}^{\circ}:=\im\left(\mathbb{T}[u^aU_{\varpi_v}^{-b}]\otimes_{\Z_p}R_0\rightarrow \End_{R_0}\left(S_\kappa^\circ(K)_{\leq h}\right)\right)       \]

		If $\xi:\A_{F,f}^\times/F^\times\rightarrow R_0^\times$ is a continuous character as above and $H^0(K,\mathcal{D}_\kappa^{<r})[\xi]$ admits a slope-$\leq h$ decomposition, we define the modular forms with central character $\xi$ to be $S_{\kappa,\xi}(K)_{\leq h}:= H^0(K,\mathscr{D}_\kappa)[\xi]_{\leq h}$ and similarly for integral modular forms with central character $\xi$.
			\label{def: overconv quat mod forms}
	\end{definition}
	\begin{remark}
		We expect that $S_{\kappa}^{\circ}(K)_{\leq h}$ and the corresponding Hecke algebra $\mathbb{T}_{\kappa,\leq h}^{\circ}$ depend on $r$, but we have suppressed that from the notation for the sake of compactness.
	\end{remark}

	\begin{remark}
		We will write $\mathbb{T}_{K,\kappa,\leq h}^{<r,\circ}$ and $\mathbb{T}_{K,\kappa,\leq h}^{\circ}$ for these Hecke algebras if the level is not clear from context.
	\end{remark}

	We again write $h=a/b$ with $a,b$ positive and relatively prime integers.  If $S_\kappa(K)_{\leq h}$ (resp. $S_{\kappa,\xi}(K)_{\leq h}$ has rank $d$, then the characteristic polynomial of $u^aU_{\varpi_v}^{-b}$ is a monic degree-$d$ polynomial over $R$.  By the definition of a slope decomposition, its roots are integral at every rank-$1$ point of $\Spa R$.  Hence the coefficients actually live in $R^\circ$ and $u^aU_{\varpi_v}^{-b}$ is power-bounded on $S_\kappa(K)_{\leq h}$ (resp. $S_{\kappa,\xi}(K)_{\leq h}$.  In particular, if $R$ is reduced and $R_0=R^\circ$, we see that $S_{\kappa}^{\circ}(K)_{\leq h}$ (resp. $S_{\kappa,\xi}^{\circ}(K)_{\leq h}$) is given concretely by 
		\[	\sum_{(i_v)\in \{0,\ldots,d-1\}^{\Sigma_p}}\prod_{v\mid p}(u^aU_{\varpi_v}^{-b})^{i_v}\left( S_{\kappa}^{<r,\circ}(K)_{\leq h} \right)	\]
	In particular, $U_p^{b(d-1)}\left( S_{\kappa}^{\circ}(K)_{\leq h} \right)\subset S_{\kappa}^{<r,\circ}(K)_{\leq h}$ (and similarly for $S_{\kappa,\xi}^{\circ}(K)_{\leq h}$).


	We now fix a choice of Hecke algebra.
Let $S$ denote the set of places of $F$ such that $v\mid p$, $D$ is ramified at $v$, or $K_v\neq \mathcal{O}_{D,v}^\times$.  For $v\notin S$, we define
\[	S_v:=\left[K\left(\begin{smallmatrix}\varpi_v & \\ & \varpi_v\end{smallmatrix}\right)K\right], \quad T_v:=\left[K\left(\begin{smallmatrix}1 & \\ & \varpi_v\end{smallmatrix}\right)K\right]\in K\backslash (\A_{F,f}\otimes D)^\times/K	\]
for some fixed uniformizer $\varpi_v$ of $\O_{F_v}$.

We define the Hecke algebra $\mathbb{T}$ to be the free commutative $\Z_p$-algebra generated by $\{U_{\varpi_v}\}_{v\mid p}$ and $\{S_w, T_w\}_{w\notin S}$.  Since $\Delta_p$ acts on the modules of distributions $\mathcal{D}_\kappa^{<r,\circ}$ and Hecke operators away from $p$ preserve the slope decomposition, we may view $S_\kappa^{<r,\circ}(K)_{\leq h}$ as a $\mathbb{T}$-module.

We also describe the so-called diamond operators, after modifying the tame level $K^p$.  Suppose we have a finite set $Q$ of places of $F$ such that for each $v\in Q$, $v\nmid p$, $\Nm v\equiv 1\pmod p$, $D$ is split at $v$, and $K_v=\GL_2(\O_{F_v})$.  For each $v\in Q$, we again let $K_0(v)\subset \H(F_v)$ denote the subgroup $\left\{\left(\begin{smallmatrix}\ast & \ast \\ 0 & \ast\end{smallmatrix}\right)\mod v\right\}$, and we consider the homomorphism 
\[	K_0(v)\rightarrow k(v)^\times\rightarrow \Delta_v	\]
given by composing
\[	\left(\begin{smallmatrix}a & b\\ c & d\end{smallmatrix}\right)\mapsto ad^{-1}	\]
with the projection to the $p$-power quotient $k(v)^\times\rightarrow\Delta_v$.  Let $K^-(v)$ denote the group 
\[	K^-(v):=\left\{\left(\begin{smallmatrix}a & \ast \\ c & d\end{smallmatrix}\right)\in K_0(v)\mid ad^{-1}\mapsto 1\text{ in } \Delta_v\right\}	\]
for each $v\in Q$, and let
\[	K_0(Q):=\prod_{v\in Q}K_0(v)\cdot\prod_{v\notin Q}K_v   \]
and
\[	K^-(Q):=\prod_{v\in Q}K^-(v)\cdot\prod_{v\notin Q}K_v	\]
Then $K_0(v)/K^-(v)\cong \Delta_v$, and every $h\in\Delta_Q:=\prod_{v\in Q}\Delta_v$ gives rise to a Hecke operator
\[	\left\langle h\right\rangle:=\left[K^-(Q)\widetilde hK^-(Q)\right]	\]
on $S_\kappa^{<r,\circ}(K^-(Q))$, where $\widetilde h$ is a lift of $h$ to $K_0(Q)$; $\left\langle h\right\rangle$ is independent of the choice of $\widetilde h$.

We let $\mathbb{T}_Q^-$ be the free commutative $\Z_p$-algebra generated by $\{U_{\varpi_v}\}_{v\mid p}$, $\{S_v, T_v\}_{v\notin S}$, and $\{U_{\varpi_v}\}_{v\in Q}$, where $U_{\varpi_v}:=[K^-(v)\left(\begin{smallmatrix}1 & 0 \\ 0 & \varpi_v\end{smallmatrix}\right)K^-(v)]$; it acts naturally on $S_\kappa^{<r,\circ}(K^-(Q))_{\leq h}$, and we let $\mathbb{T}_{K^-(Q),\leq h}^{<r,\circ}$ denote the $R_0$-algebra its image generates in $\End_{R_0}(S_\kappa^{<r,\circ}(K^-(Q))_{\leq h})$.  Similarly, we let $\mathbb{T}_{0,Q}$ be the free commutative $\Z_p$-algebra generated by $\{U_{\varpi_v}\}_{v\mid p}$, $\{S_v, T_v\}_{v\notin S}$, and $\{U_{\varpi_v}\}_{v\in Q}$, where $U_{\varpi_v}:=[K_0(v)\left(\begin{smallmatrix}1 & 0 \\ 0 & \varpi_v\end{smallmatrix}\right)K_0(v)]$.

\subsection{Integral overconvergent quaternionic modular forms}

We need to make a closer study of the structure of the integral modules of distributions and their finite-slope subspaces.  
	\begin{lemma}
		If $\kappa:T_0/\overline{Z(K)}\rightarrow R^\times$ is a weight and $H^0(K,\mathscr{D}_\kappa)$ (resp. $H^0(K,\mathscr{D}_\kappa)[\xi]$) admits a slope-$\leq h$-decomposition, then $S_\kappa(K)_{\leq h}$ (resp. $S_{\kappa,\xi}(K)_{\leq h}$) is a finite projective $R$-module.  If the Fredholm power series $F_\kappa$ has a slope $\leq h$-factorization, then $S_\kappa(K)_{\leq h}$ (resp. $S_{\kappa,\xi}(K)_{\leq h}$) is compatible with arbitrary base change on $R$.
		\label{lemma: SK finite proj}
	\end{lemma}
	\begin{proof}
		We prove the result for $H^0(K,\mathscr{D}_\kappa)$; $H^0(K,\mathscr{D}_\kappa)[\xi]$ is handled similarly.  It is enough to handle the case where the tame level is neat.  Then $S_\kappa(K)_{\leq h}$ is a direct summand of the potentially orthonormalizable Banach $R$-module $H^0(K,\mathscr{D}_\kappa)$ which is finitely generate over $R$, so by ~\cite[Lemma 2.11]{buzzard} is is projective.

		If the Fredholm power series $F_\kappa$ has a slope-$\leq h$-factorization, then the slope decomposition is functorial in $R$, by ~\cite[Theorem 2.2.13]{johansson-newton}.
	\end{proof}

	\begin{cor}
		If $\kappa:T_0/\overline{Z(K)}\rightarrow R^\times$ is a weight and $H^0(K,\mathscr{D}_\kappa)$ (resp. $H^0(K,\mathscr{D}_\kappa)[\xi]$) admits a slope-$\leq h$-decomposition, then $S_\kappa^{<r, \circ}(K)_{\leq h}$ (resp. $S_{\kappa,\xi}^{<r,\circ}(K)_{\leq h}$) is a finite $R_0$-module.
	\end{cor}
	\begin{proof}
		This follows from the equality $H^0(K,\mathcal{D}_\kappa^r)_{\leq h}=H^0(K,\mathcal{D}_{\kappa}^{<r})_{\leq h}$, and the fact that $\mathcal{D}_\kappa^{<r,\circ}$ is bounded in $\mathcal{D}_\kappa^{<r}$.
	\end{proof}

	Now we consider the behavior of $H^0(K,\mathcal{D}_{\kappa}^{<r,\circ})[\xi]_{\leq h}$ under change of coefficients.  Let $\kappa_R:T_0/\overline{Z(K)}\rightarrow R^\times$ be a weight.  If $f:R\rightarrow R'$ is a homomorphism of pseudoaffinoid algebras, we let $\kappa_{R'}$ denote the composition $T_0/\overline{Z(K)}\xrightarrow{\kappa_R}R^\times\xrightarrow{f}{R'}^\times$.  By ~\cite[Corollary A.14]{johansson-newton}, $f$ is topologically of finite type, so we have a surjection $R\left\langle X_1,\ldots,X_n\right\rangle\twoheadrightarrow R'$.  If $R$ is equipped with a norm adapted to $\kappa_R$ and $R_0\subset R$ is the corresponding ring of definition, with $u\in R_0$ a pseudouniformizer, we define $R_0':=R_0\left\langle X_1,\ldots,X_n\right\rangle$ and $u':=f(u)$.

	Let $a:=\lvert u\rvert_R$.  We define a norm $\lvert\cdot\rvert_{R'}$ on $R'$ via
	\[	\lvert r'\rvert_{R'}:=\inf\{a^{-n}\mid r'\in {u'}^nR_0'\}	\]
	Then $R_0'$ is the unit ball of $R'$ with respect to $\lvert\cdot\rvert_{R'}$, and $\lvert u'\rvert_{R'} = \vert u\rvert_R$.  Moreover, if $\lvert\cdot\rvert_R$ is adapted to $\kappa_R$, then $\lvert\cdot\rvert_{R'}$ is adapted to $\kappa_{R'}$.
	\begin{lemma}
		With notation as above, suppose that $f:R_0\rightarrow R_0'$ is a finite map.  Then the natural map $R_0'\htimes_{R_0}\mathcal{D}_{\kappa_R}^{<r,\circ}\rightarrow \mathcal{D}_{\kappa_{R'}}^{<r,\circ}$ is a topological isomorphism (with respect to the $u'$-adic topology), where the completed tensor product is taken with respect to the $u$-adic topology on $\mathcal{D}_{\kappa_R}^{<r,\circ}$ and the $u'$-adic topology on $R_0'$.
		\label{lemma: base change Dcirc}
	\end{lemma}
	\begin{proof}
		We first check that the morphism $R_0'\htimes_{R_0}\mathcal{D}_{\kappa_R}^{<r,\circ}\rightarrow \mathcal{D}_{\kappa_{R'}}^{<r,\circ}$ is an isomorphism of $R_0'$-modules.  The discussion after ~\cite[Proposition 3.2.7]{johansson-newton} shows that 
		\[	\mathcal{D}_{\kappa_R}^{<r,\circ} \cong \prod_\alpha R_0\cdot u^{-n_R(r,u,\alpha)}\mathbf{n}^\alpha	\]
		where $n_R(r,u,\alpha):= \left\lfloor\frac{\lvert\alpha\rvert\log_pr}{\log_p\lvert u\rvert_R}\right\rfloor$, $\mathbf{n}$ is a certain (non-canonical but explicit) finite set (depending only on the group-theoretic data we fixed at the beginning of ~\textsection\ref{section: extended eigenvarieties}), and $\alpha$ is a multi-index (and similarly for $\mathcal{D}_{\kappa_{R'}}^{<r,\circ}$).  Now $R_0'$ is a finitely presented $R_0$-module, and for any finitely presented $R_0$-module $M$, the natural morphism $M\otimes_{R_0}\prod_\alpha R_0\cdot u^{-n_R(r,u,\alpha)}\mathbf{b}^\alpha\rightarrow \prod_\alpha M\cdot u^{-n_R(r,u,\alpha)}\mathbf{b}^\alpha$ is an isomorphism.  By construction, $n_R(r,u,\alpha) = n_{R'}(r,u',\alpha)$ for all $\alpha$, so the claim follows.

		Finally, the morphism $R_0'\htimes_{R_0}\mathcal{D}_{\kappa_R}^{<r,\circ}\rightarrow \mathcal{D}_{\kappa_{R'}}^{<r,\circ}$ is clearly continuous, so the open mapping theorem implies that it is a topological isomorphism.
	\end{proof}

	\begin{cor}\label{cor: integral modular forms finite base change surj}
		With notation as above, suppose that $f:R_0\rightarrow R_0'$ is a finite map.  If $F_\kappa$ has a slope $\leq h$-factorization, then the natural map
		\[	R_0'\otimes_{R_0}S_{\kappa_R,\xi}^{<r,\circ}(K)_{\leq h}\rightarrow S_{\kappa_{R'},\xi}^{<r,\circ}(K)_{\leq h}	\]
		is surjective.
	\end{cor}
	\begin{proof}
		Writing $D^\times\backslash (\A_{F,f}\otimes_F D)^\times/K=\coprod_{i\in I}D^\times g_i K$ for some finite set of elements $g_i\in (\A_{F,f}\otimes_F D)^\times$, we have an isomorphism 
		\[	H^0(K,\mathcal{D}_{\kappa_R}^{<r,\circ})[\xi]\cong \oplus_{i\in I}\left(\mathcal{D}_{\kappa_R}^{<r,\circ}\right)^{(K\A_{F,f}^\times\cap g_i^{-1}D^\times g_i)/F^\times}	\]
		For every map $R\rightarrow R'$ as above, Lemma~\ref{lemma: base change Dcirc} implies that the base change map
		\[	R_0'\htimes_{R_0}\oplus_i\mathcal{D}_{\kappa_R}^{<r,\circ}\rightarrow \oplus_i \mathcal{D}_{\kappa_{R'}}^{<r,\circ}	\]
		is an isomorphism.  Moreover, the calculations of ~\cite[Lemma 1.1]{taylor2006} show that the order of $(K\A_{F,f}^\times\cap g_i^{-1}D^\times g_i)/F^\times$ is prime to $p$ for all $i$, so the base change map $R_0'\htimes_{R_0}H^0(K,\mathcal{D}_{\kappa_R}^{<r,\circ})[\xi]\rightarrow H^0(K,\mathcal{D}_{\kappa_{R'}}^{<r,\circ})[\xi]$ is an isomorphism.

		Now we have a commutative diagram
		\[
			\begin{tikzcd}
				R_0'\htimes_{R_0} H^0\left(K,\mathcal{D}_{\kappa_R}^{<r,\circ}\right)[\xi] \ar[r]\ar[d, "\sim", sloped] & R'\htimes_RH^0\left(K,\mathcal{D}_{\kappa_R}^{<r}\right)[\xi] \ar[r, two heads]\ar[d] & R'\otimes_RH^0\left(K,\mathcal{D}_{\kappa_R}^{<r}\right)[\xi]_{\leq h} \ar[d, "\sim", sloped]	\\
				H^0\left(K,\mathcal{D}_{\kappa_{R'}}^{<r,\circ}\right)[\xi]\ar[r] & H^0\left(K,\mathcal{D}_{\kappa_{R'}}^{<r}\right)[\xi] \ar[r, two heads] & H^0\left(K,\mathcal{D}_{\kappa_{R'}}^{<r}\right)[\xi]_{\leq h}
			\end{tikzcd}
		\]
		(where the fact that the right vertical arrow is an isomorphism follows from Lemma~\ref{lemma: SK finite proj}).  This implies first of all that the map $R'\otimes_{R}H^0\left(K,\mathcal{D}_{\kappa_R}^{<r}\right)[\xi]\rightarrow H^0\left(K,\mathcal{D}_{\kappa_{R'}}^{<r}\right)[\xi]$ carries $R_0'\otimes_{R_0}S_{\kappa_R,\xi}^{<r,\circ}(K)_{\leq h}$ to $S_{\kappa_{R'},\xi}^{<r,\circ}(K)_{\leq h}$.

		To prove surjectivity, we may lift $f\in S_{\kappa_{R'},\xi}^{<r,\circ}(K)_{\leq h}$ to an element of $R_0'\htimes_{R_0} H^0\left(K,\mathcal{D}_{\kappa_R}^{<r,\circ}\right)[\xi]$, since the left vertical arrow is an isomorphism. Its image in $R'\otimes_RH^0\left(K,\mathcal{D}_{\kappa_R}^{<r}\right)[\xi]_{\leq h}$ is therefore an element of $R_0'\otimes_{R_0}S_{\kappa_R,\xi}^{<r,\circ}(K)_{\leq h}$ in the pre-image of $f$.
	\end{proof}

	We may also extend ~\cite[Lemma 2.1.4]{kisin2009} and ~\cite[Lemma 2.1.7]{kisin2009} to statements about families of integral overconvergent modular forms.
	\begin{prop}
		Let $\kappa:T_0/\overline{Z(K)}\rightarrow R^\times$ be a weight, and let $\chi:\Delta_Q\rightarrow R^\times$ be a character.  For any finite set of primes $Q$ as in \textsection ~\ref{subsect: oc quat definitions}, suppose that $H^0(K_0(Q),\mathscr{D}_\kappa)$ and $H^0(K^-(Q),\mathscr{D}_\kappa)$ admit slope-$\leq h$-decompositions. Then the natural map
		\[	\sum_{h\in\Delta_Q}\chi(h)^{-1}\left\langle h\right\rangle:\left(S_{\kappa,\xi}^{<r,\circ}(K^-(Q))_{\leq h}\right)_{\Delta_Q=\chi}\rightarrow S_{\kappa,\xi}^{<r,\circ}(K^-(Q))_{\leq h}^{\Delta_Q=\chi}	\]
		is an isomorphism.  
	\end{prop}
	Here $\left(S_{\kappa,\xi}^{<r,\circ}(K^-(Q))_{\leq h}\right)_{\Delta_Q=\chi}$ is the maximal quotient of $S_{\kappa,\xi}^{<r,\circ}(K^-(Q))_{\leq h}$ on which $\Delta_Q$ acts by $\chi$; if $\chi$ is the trivial character, this is simply the co-invariants.
	\begin{proof}
		We first assume that $K$ is neat.  Writing $D^\times\backslash(\A_{F,f}\otimes_FD)^\times/K_0(Q) = \sqcup_{i\in I}D^\times g_iK_0(Q)$, we have a finite disjoint union
		\[	H^0(K^-(Q),\mathcal{D}_\kappa^{<r,\circ})[\xi] = \oplus_{i\in I}\oplus_{h\in\Delta_Q}\mathcal{D}_\kappa^{<r,\circ} 	\]
		We claim that $\Delta_Q$ acts freely on $D^\times\backslash (\A_{F,f}\otimes_F D)^\times/K^-(Q)$.  But if $D^\times g_ih_j K^-(Q) = D^\times g_{i'}h_{j'}K^-(Q)$, then the neatness hypothesis ~\ref{hyp: neatness} implies that $i=i'$ and $j=j'$.  Hence we have 
		\[	H^0(K^-(Q),\mathcal{D}_\kappa^{<r,\circ})[\xi] = \oplus_{i\in I}R_0[\Delta_Q]\otimes_{R_0}\mathcal{D}_\kappa^{<r,\circ}     \]
		and we can write
		\[	\sum_{h\in\Delta_Q}\chi(h)^{-1}\left\langle h\right\rangle:\left(H^0(K^-(Q),\mathcal{D}_\kappa^{<r,\circ})[\xi]\right)_{\Delta_Q=\chi}\xrightarrow\sim H^0(K^-(Q),\mathcal{D}_\kappa^{<r,\circ})[\xi]^{\Delta_Q=\chi}	\]
		and
		\[	\sum_{h\in\Delta_Q}\chi(h)^{-1}\left\langle h\right\rangle:\left(H^0(K^-(Q),\mathcal{D}_\kappa^{<r})[\xi]\right)_{\Delta_Q=\chi}\xrightarrow\sim H^0(K^-(Q),\mathcal{D}_\kappa^{<r})[\xi]^{\Delta_Q=\chi}	\]

		If $K'\vartriangleleft K$ with $K'$ neat and $[K:K']$ prime to $p$, then $\sum_{h\in\Delta_Q}\chi(d)\left\langle h\right\rangle$ induces diagrams
		\[	
			\begin{tikzcd}
				\left(H^0(K^-(Q),\mathcal{D}_\kappa^{<r,\circ})[\xi]\right)_{\Delta_Q=\chi} \ar[d, "\sim", sloped]\ar[r, "\sim",dashed] & H^0(K^-(Q),\mathcal{D}_\kappa^{<r,\circ})[\xi]^{\Delta_Q=\chi} \ar[d, "\sim", sloped]	\\
				\left(H^0({K'}^-(Q),\mathcal{D}_\kappa^{<r,\circ})[\xi]\right)^{K'/K}_{\Delta_Q=\chi}\ar[r, "\sim", sloped] & H^0({K'}^-(Q),\mathcal{D}_\kappa^{<r,\circ})[\xi]^{\Delta_Q=\chi,K'/K}    
			\end{tikzcd}
		\]
		and
		 \[
                        \begin{tikzcd}
				\left(H^0(K^-(Q),\mathcal{D}_\kappa^{<r})[\xi]\right)_{\Delta_Q=\chi} \ar[d, "\sim", sloped]\ar[r, "\sim", dashed] & H^0(K^-(Q),\mathcal{D}_\kappa^{<r})[\xi]^{\Delta_Q=\chi} \ar[d, "\sim", sloped]   \\
				\left(H^0({K'}^-(Q),\mathcal{D}_\kappa^{<r})[\xi]\right)^{K'/K}_{\Delta_Q=\chi}\ar[r, "\sim", sloped] & H^0({K'}^-(Q),\mathcal{D}_\kappa^{<r})[\xi]^{\Delta_Q=\chi,K'/K}
                        \end{tikzcd}
                \]
		Using ~\cite[Proposition 2.2.11]{johansson-newton}, for any level $K$ we obtain an isomorphism
		\[	\sum_{h\in\Delta_Q}\chi(h)^{-1}\left\langle h\right\rangle: S_{\kappa,\xi}(K^-(Q)_{\leq h})_{\Delta_Q}\xrightarrow\sim S_{\kappa,\xi}(K_0(Q))_{\leq h}	\]
		Then we have a diagram
		\[
			\begin{tikzcd}
				0 \ar[r] & I_{\Delta_Q,\chi}H^0(K^-(Q),\mathcal{D}^{<r,\circ})[\xi] \ar[r] \ar[d] & H^0(K^-(Q),\mathcal{D}^{<r,\circ})[\xi] \ar[r]\ar[d] & H^0(K^-(Q),\mathcal{D}_\kappa^{<r})[\xi]^{\Delta_Q=\chi} \ar[d]\ar[r] & 0	\\
				0 \ar[r] & I_{\Delta_Q,\chi}H^0(K^-(Q),\mathcal{D}^{<r})[\xi]_{\leq h} \ar[r] & H^0(K^-(Q),\mathcal{D}^{<r})[\xi]_{\leq h} \ar[r] & H^0(K_0(Q),\mathcal{D}_\kappa^{<r})[\xi]_{\leq h}^{\Delta_Q=\chi} \ar[r] & 0
			\end{tikzcd}
		\]
		where $I_{\Delta_Q,\chi}\subset R_0[\Delta_Q]$ denotes the ideal generated by the elements $\left\langle h\right\rangle -\chi(h)$ for $h\in\Delta_Q$.  A diagram chase shows that we have the desired isomorphism
		\[	\sum_{h\in\Delta_Q}\chi(h)^{-1}\left\langle h\right\rangle: \left(S_{\kappa,\xi}^{<r,\circ}(K^-(Q))_{\leq h}\right)_{\Delta_Q=\chi}\xrightarrow\sim S_{\kappa,\xi}^{<r,\circ}(K^-(Q))_{\leq h}^{\Delta_Q=\chi}	\]
	\end{proof}

	Now assume that $R$ is a local field with uniformizer $u$, that is, a finite extension of $\Q_p$ or $\F_p(\!(u)\!)$.  Then by ~\cite[Theorem 4.4.2]{ash-stevens}, it is automatic that $H^0(K,\mathscr{D}_\kappa)$ and $H^0(K,\mathscr{D}_\kappa)[\xi]$ admit slope-$\leq h$-decompositions (and that $F_\kappa$ and $F_{\kappa,\xi}$ admit slope-$\leq h$-factorizations).
	\begin{prop}
		If $R$ is a local field, with ring of integers $R_0$ and uniformizer $u\in R_0$, the module $S_{\kappa,\xi}^{<r,\circ}(K^-(Q))_{\leq h}$ is finite projective over $R_0[\Delta_Q]$.
	\label{prop: fin proj diamond operators}
	\end{prop}
	\begin{proof}
		To check that $S_{\kappa,\xi}^{<r,\circ}(K^-(Q)_{\leq h})$ is projective over $R_0[\Delta_Q]$, we may replace $R_0$ with a finite extension, so we may assume that $R_0$ contains the values of all characters $\chi:\Delta_Q\rightarrow \overline{R}^\times$.  If $\Delta_v$ has order $p^{n_v}$, we can write $R_0[\Delta_Q]$ explicitly (but non-canonically) as $R_0[\{x_v\}_{v\mid p}]/(\{x_v^{p^{n_v}}-1)$; this assumption implies that the polynomials $x_v^{p^{n_v}}-1$ split completely, and the ideals $I_{\Delta_Q,\chi}$ introduced above are the non-maximal prime ideals of $R_0[\Delta_Q]$.
	
	On the other hand, we have a family of surjections
	\[	H^0(K^-(Q),\mathcal{D}_\kappa^{<r,\circ})[\xi]_{\leq h}\twoheadrightarrow H^0(K^-(Q),\mathcal{D}_\kappa^{<r,\circ})[\xi]_{\leq h}^{\Delta_Q=\chi}	\]
	The target is a lattice in $H^0(K^-(Q),\mathcal{D}_\kappa^{<r})[\xi]_{\leq h}^{\Delta_Q=\chi}$; since $R_0$ is a discrete valuation ring, $H^0(K^-(Q),\mathcal{D}_\kappa^{<r,\circ})[\xi]_{\leq h}^{\Delta_Q=\chi}$ is free of some rank $d_\chi$.  Since $R_0[\Delta_Q]$ is a local ring, $H^0(K^-(Q),\mathcal{D}_\kappa^{<r,\circ})[\xi]_{\leq h}$ can be generated by $d_{\chi}$ elements as a $R_0[\Delta_Q]$-module.

	Furthermore, $R_0[\Delta_Q]_{\Delta_Q=\chi}\cong R_0$. Since $H^0(K^-(Q),\mathcal{D}_\kappa^{<r,\circ})[\xi]_{\leq h}^{\Delta_Q=\chi}$ cannot be generated as an $R_0$-module by fewer than $d_\chi$ elements, this implies that the ranks $d_\chi$ agree for all characters $\chi$; call this number $d$.

	We therefore have a presentation of $S_{\kappa,\xi}^{<r,\circ}(K^-(Q)_{\leq h})$:
	\[	R_0[\Delta_Q]^{\oplus d'}\rightarrow R_0[\Delta_Q]^{\oplus d}\rightarrow S_{\kappa,\xi}^{<r,\circ}(K^-(Q)_{\leq h})\rightarrow 0	\]
	Since the surjection $R_0[\Delta_Q]^{\oplus d}\rightarrow S_{\kappa,\xi}^{<r,\circ}(K^-(Q)_{\leq h})$ is an isomorphism modulo each $I_{\Delta_Q,\chi}$, the image of $R_0[\Delta_Q]^{\oplus d'}$ in $R_0[\Delta_Q]^{\oplus d}$ is contained in $I_{\Delta_Q,\chi}R_0[\Delta_Q]^{\oplus d}$.  In particular, if $d'\neq 0$, then the Fitting ideals of $S_{\kappa,\xi}^{<r,\circ}(K^-(Q)_{\leq h})$ are contained in $\cap_\chi I_{\Delta_Q,\chi}$.

	On the other hand, the module $H^0(K^-(Q),\mathcal{D}_\kappa^{<r})[\xi]$ is potentially orthonormalizable as an $R[\Delta_Q]$-module, so by ~\cite[Theorem 2.2.2]{johansson-newton} $H^0(K^-(Q),\mathcal{D}_\kappa^{<r})[\xi]_{\leq h}$ is a finite projective $R[\Delta_Q]$-module.  In particular, for each prime $\mathfrak{p}\subset R[\Delta_Q]$, there is some integer $d_{\mathfrak{p}}\leq d$ such that $\Fitt_{k}(S_{\kappa,\xi}(K^-(Q))_{\leq h,\mathfrak{p}})=0$ for $k<d_{\mathfrak{p}}$ and $\Fitt_{k}(S_{\kappa,\xi}(K^-(Q))_{\leq h,\mathfrak{p}})=R[\Delta_Q]_{\mathfrak{p}}$ for $k\geq d_{\mathfrak{p}}$.  But the formation of Fitting ideals is functorial in the coefficients, and $\cap_\chi I_{\Delta_Q,\chi}$ does not generate the unit ideal in $R[\Delta_Q]$, so $d'=0$ and $S_{\kappa,\xi}^{<r,\circ}(K^-(Q)_{\leq h})$ is free of rank $d$ over $R_0[\Delta_Q]$.
	\end{proof}

We may consider characteristic polynomials of operators on $S_{\kappa,\xi}(K^-(Q))_{\leq h}$, viewed as either a rank-$d$ projective $R[\Delta_Q]$-module, or as a rank-$d\lvert\Delta_Q]$ projective $R$-module.  In particular, we have seen that if $h=a/b$, the $R$-linear characteristic polynomial of $u^aU_{\varpi_v}^{-b}$ has coefficients in $R^\circ$.  Using properties of circulant matrices, we see that the $R[\Delta_Q]$-linear characteristic polynomial of $u^aU_{\varpi_v}^{-b}$ has coefficients in $R^\circ[\Delta_Q]$.

	\begin{cor}
		Let notation be as above, and let $d$ denote the rank of $S_{\kappa,\xi}(K_0(Q))_{\leq h}$.  Suppose that $R$ is reduced and $R_0=R^\circ$.  Then the natural map
		\[	 \sum_{h\in\Delta_Q}\left\langle h\right\rangle:\left(S_{\kappa,\xi}^{\circ}(K^-(Q))_{\leq h}\right)_{\Delta_Q}\rightarrow S_{\kappa,\xi}^{\circ}(K_0(Q))_{\leq h} \]
		is surjective, and its kernel is annhilated by $u^{(d-1)a}$.
		\label{cor: control new lattice}
	\end{cor}
	\begin{proof}
		Since $u^aU_{\varpi_v}^{-b}$ is power-bounded for all $v\mid p$ on both $S_{\kappa,\xi}(K_0(Q))_{\leq h}$ and $S_{\kappa,\xi}(K^-(Q))_{\leq h}$, by assumption, and $U_{\varpi_v}$ commutes with the diamond operators, surjectivity follows.

		To study the kernel of $\sum_{h\in\Delta_Q}\left\langle h\right\rangle$, we first observe that for $f\in S_{\kappa,\xi}^{\circ}(K^-(Q))_{\leq h}$, $U_p^{b(d-1)}(f)\in S_{\kappa,\xi}^{<r,\circ}(K^-(Q))_{\leq h}$.  Suppose $f\in S_{\kappa,\xi}^\circ(K^-(Q))_{\leq h}$ is in the kernel of $\sum_{h\in\Delta_Q}\left\langle h\right\rangle$.  Since $U_p$ commutes with the diamond operators, $(u^aU_{p}^{-b})^{d-1}(f)$ is also in the kernel of $\sum_{h\in\Delta_Q}\left\langle h\right\rangle$, and by Proposition~\ref{prop: fin proj diamond operators} it actually lives in $I_{\Delta_Q}S_{\kappa,\xi}^{<r}(K^-(Q))_{\leq h}$.  But then
		\[	u^{(d-1)a}f = (u^aU_{p}^{-b})^{d-1}U_p^{b(d-1)}(f) \in I_{\Delta_Q}S_{\kappa,\xi}^{\circ}(K^-(Q))_{\leq h}	\]
		as desired.
	\end{proof}

\begin{prop}
	Suppose that $\kappa$ is an open weight such that $\Spa R$ contains a Zariski-dense set of classical weights, and suppose that $F_\kappa$ admits a slope $\leq h$-factorization.  Let $A_v, B_v\in \mathbb{T}_K^{<r,\circ}$ be lifts of $\alpha_v, \beta_v$, respectively.  Then the map
	\[	\prod_{v\in {Q}}(U_{\varpi_v}-B_v):S_{\kappa,\xi}^{<r,\circ}(K)_{\leq h,\mathfrak{m}}\rightarrow S_{\kappa,\xi}^{<r,\circ}(K_0(Q))_{\leq h,\mathfrak{m}_{Q,0}}	\]
	is an isomorphism (where we view $S_{\kappa,\xi}^{<r,\circ}(K)_{\leq h,\mathfrak{m}}$ as a submodule of $S_{\kappa,\xi}^{<r,\circ}(K_0(Q))_{\leq h,\mathfrak{m}}$).
	\label{prop: add K0 level structure}
\end{prop}
\begin{proof}
	We may assume $Q=\{v\}$, by induction on the size of $Q$.  Then the source and the target are finite $R_0$-modules.  After inverting $u$, ~\cite[Lemma 2.1.7]{kisin2009} implies that the map is an isomorphism when specialized to any sufficiently large classical weight.  It follows that $S_{\kappa,\xi}(K)_{\leq h,\mathfrak{m}}$ and $S_{\kappa,\xi}(K_0(Q))_{\leq h,\mathfrak{m}_{Q,0}}$ have the same rank over $R$.  We claim that it suffices to check that $U_{\varpi_v}-B_v$ is surjective after specializing at every maximal ideal of $R_0$.  Indeed, this implies that 
	\[	U_{\varpi_v}-B_v:S_{\kappa,\xi}(K)_{\leq h,\mathfrak{m}}\rightarrow S_{\kappa,\xi}(K_0(Q))_{\leq h,\mathfrak{m}_{Q,0}}  \]
	is a surjection of projective $R$-modules of the same rank, so it is injective.  Then the kernel of $U_{\varpi_v}-B_v$ on $S_{\kappa,\xi}^{<r,\circ}(K)_{\leq h,\mathfrak{m}}$ is $u$-torsion.  But $S_{\kappa,\xi}^{<r,\circ}(K)_{\leq h,\mathfrak{m}}$ has no $u$-torsion, by definition, so the kernel is trivial.

	Thus, we need to check that 
	\[	U_{\varpi_v}-B_v: \F'\otimes_{R_0}S_{\kappa,\xi}^{<r,\circ}(K)_{\leq h,\mathfrak{m}}\rightarrow \F'\otimes_{R_0}S_{\kappa,\xi}^{<r,\circ}(K_0(Q))_{\leq h,\mathfrak{m}_{Q,0}}	\]
	is surjective for any specialization $R_0\rightarrow\F'$ at a maximal ideal.  There is some maximal point $x\in\Spa R$ with residue field $R_x$ and ring of integers $R_{x,0}$ such that $R_0\rightarrow \F'$ factors through $R_0\rightarrow R_{x,0}$, and by Corollary~\ref{cor: integral modular forms finite base change surj} the maps $R_{x,0}\otimes_{R_0}S_{\kappa,\xi}^{<r,\circ}(K)_{\leq h,\mathfrak{m}}\rightarrow S_{\kappa_x,\xi}^{<r,\circ}(K)_{\leq h,\mathfrak{m}}$ and $R_{x,0}\otimes_{R_0}S_{\kappa,\xi}^{<r,\circ}(K_0(Q))_{\leq h,\mathfrak{m}_{0,Q}}\rightarrow S_{\kappa_x,\xi}^{<r,\circ}(K_0(Q))_{\leq h,\mathfrak{m}_{Q,0}}$ are surjective.  It therefore suffices to prove that
	\[	U_{\varpi_v}-B_v: \F'\otimes_{R_{x,0}}S_{\kappa_x,\xi}^{<r,\circ}(K)_{\leq h,\mathfrak{m}}\rightarrow \F'\otimes_{R_{x,0}}S_{\kappa_x,\xi}^{<r,\circ}(K_0(Q))_{\leq h,\mathfrak{m}_{Q,0}}   \]
	is surjective.  But this is a map of vector spaces of the same dimension, so it is enough to prove injectivity.  

	The module $\F'\otimes_{R_0}S_{\kappa_x,\xi}^{<r,\circ}(K)_{\leq h,\mathfrak{m}}$ is a finite module over the artin local ring $\mathbb{T}_\mathfrak{m}/\pi$, so if the kernel of $U_{\varpi_v}-B_v$ is non-trivial, it contains $f\neq 0$ which is $\mathfrak{m}$-torsion.  In particular, $T_v(f)=(\alpha_v+\beta_v)x$ and $U_{\varpi_v}(f)=\beta_v$.

	Since 
	\[	[K_0(v)\left(\begin{smallmatrix}1 & \\ & \varpi_v\end{smallmatrix}\right)K_0(v)] = \coprod_{\alpha\in k_v}\left( \begin{smallmatrix}1 & \\ \widetilde\alpha\varpi_v & \varpi_v\end{smallmatrix} \right)K_0(v)	\]
	where $\widetilde\alpha$ denotes a lift of $\alpha$, we have
	\[	U_{\varpi_v}f = \sum_{a\in k(v)}{ }_{\left(\begin{smallmatrix}1 & \\ \widetilde\alpha\varpi_v & \varpi_v\end{smallmatrix}\right)}|f	\]
	But $\left(\begin{smallmatrix}1 & \\ \widetilde\alpha\varpi_v & \varpi_v\end{smallmatrix}\right) = \left(\begin{smallmatrix}1 & \\ & \varpi_v\end{smallmatrix}\right)\left(\begin{smallmatrix}1 & \\ \widetilde\alpha & 1 \end{smallmatrix}\right)$ and ${}_{\left(\begin{smallmatrix}1 & \\ \widetilde\alpha & 1 \end{smallmatrix}\right)}|f = f$, since $f$ is fixed by $\left(\begin{smallmatrix}1 & \\ \widetilde\alpha & 1 \end{smallmatrix}\right)\in\GL_2(\O_{F_v})$ by assumption, so 
	\[	U_{\varpi_v}f = \lvert k(v)\rvert {}_{\left(\begin{smallmatrix}1 & \\ & \varpi_v\end{smallmatrix}\right)}|f = {}_{\left(\begin{smallmatrix}1 & \\ & \varpi_v\end{smallmatrix}\right)}|f	\]

	Similarly, we have
	\[	[\GL_2(\O_{F_v})\left(\begin{smallmatrix}1 & \\ & \varpi_v\end{smallmatrix}\right)\GL_2(\O_{F_v})] = \left( \begin{smallmatrix}\varpi_v & \\ & 1\end{smallmatrix} \right)\GL_2(\O_{F_v})\bigsqcup \coprod_{\alpha\in k_v}\left( \begin{smallmatrix}1 & \\ \widetilde\alpha & \varpi_v \end{smallmatrix}\right)\GL_2(\O_{F_v})	\]
	so 
	\[	T_vf = {}_{\left(\begin{smallmatrix}\varpi_v & \\ & 1 \end{smallmatrix}\right)}|f + \sum_{\alpha\in k(v)} {}_{\left( \begin{smallmatrix}1 & \\ \widetilde\alpha & \varpi_v \end{smallmatrix}\right)}|f	\]
	Now for any $\alpha\in k(v)$,
	\[	{}_{\left(\begin{smallmatrix}1 & \\ \widetilde\alpha & \varpi_v \end{smallmatrix}\right)}|f = {}_{\left(\begin{smallmatrix}1 & \\ \widetilde\alpha & 1 \end{smallmatrix}\right)\left(\begin{smallmatrix}1 & \\ & \varpi_v \end{smallmatrix}\right)}|f = \beta_vf	\]
	so
	\[	{}_{\left(\begin{smallmatrix}\varpi_v & \\ & 1 \end{smallmatrix}\right)}|f = (T_v - U_{\varpi_v})(f) = \alpha_vf	\]
	But
	\[	{}_{\left(\begin{smallmatrix}\varpi_v & \\ & 1 \end{smallmatrix}\right)}|f = {}_{\left(\begin{smallmatrix} & 1\\1 &\end{smallmatrix}\right)\left(\begin{smallmatrix}1 & \\ & \varpi_v \end{smallmatrix}\right)\left(\begin{smallmatrix} & 1\\1 &\end{smallmatrix}\right)}|f = {}_{\left(\begin{smallmatrix}1 & \\ & \varpi_v \end{smallmatrix}\right)}|f = \beta_vf	\]
	since $f$ is fixed by $\left(\begin{smallmatrix}1 & \\ & \varpi_v \end{smallmatrix}\right)\in\GL_2(\O_{F_v})$, so $\alpha_v=\beta_v$, which contradicts our assumption. 
\end{proof}

\begin{cor}
	With notation as above, the map
	\[      \prod_{v\in {Q}}(U_{\varpi_v}-B_v):S_{\kappa,\xi}^{\circ}(K)_{\leq h,\mathfrak{m}}\rightarrow S_{\kappa,\xi}^{\circ}(K_0(Q))_{\leq h,\mathfrak{m}_{Q,0}}  \]
	is an isomorphism.
\end{cor}

\subsection{Varying the level}

We record some results on the existence of slope decompositions as we vary the tame level. Fix a set of places $Q$ as above, and fix a maximal ideal $\mathfrak{m}\subset \mathbb{T}$ which corresponds to the residual Hecke eigenvalues at some maximal point of $\mathscr{X}_{\underline D^\times}$.  There is a corresponding Galois representation $\overline\rho_{\mathfrak{m}}:\Gal_F\rightarrow \GL_2(\F)$ for some finite field $\F$; it is unramified at all places of $Q$ and the characteristic polynomial of $\overline\rho_{\mathfrak{m}}({\Frob}_v)$ is $X^2-T_vX+\Nm(v)S_v$ for all $v\in Q$.  After replacing $\F$ with a quadratic extension if necessary, we may assume that each such characteristic polynomial has roots $\{\alpha_v,\beta_v\}$ in $\F$; we assume that $\alpha_v\beta_v^{-1}\notin\{1,\Nm(v)^{\pm}\}$.

Let $E/\Q_p$ be a finite extension with ring of integers $\O_E$, uniformizer $\pi$, and residue field containing $\F$, and replace the Hecke algebras $\mathbb{T}$ and $\mathbb{T}_{Q,0}$ with $\O_E\otimes_{\Z_p}\mathbb{T}$ and $\O_E\otimes_{\O_E}\mathbb{T}_{Q,0}$, respectively.  Similarly, replace the coefficient module $\mathscr{D}_\kappa$ with its base-change to $\O_E$, so that the Hecke algebras continue to act (the upshot is that we also base-change the resulting eigenvarieties from $\Z_p$ to $\O_E$, but we suppress this from the notation).  Fix a root $\alpha_v\in \mathbb{F}$ of each characteristic polynomial, and fix a lift $A_v\in \O_E$ of each $\alpha_v$.  Then we define $\mathfrak{m}_{Q,0}\subset \mathbb{T}_{Q,0}$ to be the maximal ideal generated by $\mathfrak{m}\cap \mathbb{T}_{Q,0}$ and $U_{\varpi_v}-A_v$ for all $v\in Q$.
\begin{lemma}
	Fix a central character $\xi:\A_{F,f}^\times/F^\times\rightarrow \O(\mathscr{W}_F)^\times$.  Then there is an isomorphism $\mathscr{X}_{\underline D^\times,\xi,\mathfrak{m}}^{K,\mathbb{T}_{Q,0}}\xrightarrow{\sim}\mathscr{X}_{\underline D^\times,\xi,\mathfrak{m}_{Q,0}}^{K_Q(0),\mathbb{T}_{Q,0}}$, compatible with the respective morphisms to $\mathscr{W}_F$.
	\label{lemma: eigenvariety level K0Q isom}
\end{lemma}
\begin{proof}
	Let $S_\kappa(K)_{\leq h,\mathfrak{m}}:=\mathbb{T}_{\mathfrak{m}}\otimes_{\mathbb{T}}S_\kappa(K)_{\leq h}$, and similarly for $S_\kappa(K_Q(0))_{\leq h,\mathfrak{m}_{Q,0}}$.  By ~\cite[Lemma 2.1.7]{kisin2009}, for any slope $h$ and any sufficiently large classical weight $\kappa$, we have an isomorphism of $\mathbb{T}_{Q,0}$-modules
	\[	S_\kappa(K)_{\leq h,\mathfrak{m}}\xrightarrow{\sim}S_\kappa(K_Q(0))_{\leq h,\mathfrak{m}_{Q,0}}	\]
	By construction, classical points are dense in $\mathscr{X}_{\underline D^\times,\xi,\mathfrak{m}}^{K,\mathbb{T}_{Q,0}}$ and $\mathscr{X}_{\underline D^\times,\xi,\mathfrak{m}_{Q,0}}^{K_Q(0),\mathbb{T}_{Q,0}}$, so we may use ~\cite[Theorem 3.2.1]{johansson-newton17} to construct morphisms of eigenvarieties 
\[	\mathscr{X}_{\underline D^\times,\xi,\mathfrak{m}_{Q,0}}^{K_Q(0),\mathbb{T}_{Q,0}}\rightarrow \mathscr{X}_{\underline D^\times,\xi,\mathfrak{m}}^{K,\mathbb{T}_{Q,0}}     \]
	and
	\[	\mathscr{X}_{\underline D^\times,\xi,\mathfrak{m}}^{K,\mathbb{T}_{Q,0}}\rightarrow \mathscr{X}_{\underline D^\times,\xi,\mathfrak{m}_{Q,0}}^{K_Q(0),\mathbb{T}_{Q,0}}	\]
	These morphisms are mutually inverse, so they are isomorphisms.
\end{proof}

\begin{cor}
	Fix a central character $\xi:\A_{F,f}^\times/F^\times\rightarrow \O(\mathscr{W}_F)^\times$.  Let $U=\Spa R\subset \mathscr{W}_F$ be a connected affinoid open, corresponding to a weight $\kappa$, and fix $h\in\mathbb{Q}_{>0}$.  Then $(U,h)$ is a slope datum for $\mathscr{X}_{\underline D^\times,\xi,\mathfrak{m}}^{K}$ if and only if it is a slope datum for $\mathscr{X}_{\underline D^\times,\xi,\mathfrak{m}_{Q,0}}^{K_0(Q)}$.
	\label{cor: slope data change of level K0Q}
\end{cor}
\begin{proof}
	We write $h=m/n$ and consider the closed ball $\mathbb{B}_{U,h}:=\{\lvert T^n\rvert\leq \lvert u^{-m}\rvert\}\subset\mathbb{A}_U^1$.  If $\mathscr{Z}$ and $\mathscr{Z}'$ denote the spectral varieties for $\mathscr{X}_{\underline D^\times,\xi,\mathfrak{m}}^{K}$ and $\mathscr{X}_{\underline D^\times,\xi,\mathfrak{m}_{Q,0}}^{K_0(Q)}$, respectively, we set $Z_{U,h}:=\mathscr{Z}\cap \mathbb{B}_{U,h}$ and $Z_{U,h}':=\mathscr{Z}'\cap \mathbb{B}_{U,h}$.  We need to show that $Z_{U,h}\rightarrow U$ is finite with constant degree if and only if $Z_{U,h}'\rightarrow U$ is.

	Since the morphisms $\mathscr{Z}\rightarrow\mathscr{W}_F$ and $\mathscr{Z}'\rightarrow\mathscr{W}_F$ are flat and we have assumed $U$ is connected, it is enough to prove that $Z_{U,h}\rightarrow U$ is finite if and only if $Z_{U,h}'\rightarrow U$ is.  To see this, it is enough to show the same statement about the morphisms $Z_{U,h}^{\red}, Z_{U,h}^{',\red}\rightarrow U$ on the underlying reduced subspaces.
	
	Setting $\mathbb{T}:=\O_E$, we have 
	\[	\mathscr{Z}^{\red} = \mathscr{X}_{\underline D^\times,\xi,\mathfrak{m}}^{K,\mathbb{T}}	\]
	and
	\[	\mathscr{Z}^{',\red} = \mathscr{X}_{\underline D^\times,\xi,\mathfrak{m}}^{K_0(Q),\mathbb{T}}	\]
	Then as in Lemma~\ref{lemma: eigenvariety level K0Q isom}, we have an isomorphism $\mathscr{X}_{\underline D^\times,\xi,\mathfrak{m}}^{K,\mathbb{T}}\rightarrow \mathscr{X}_{\underline D^\times,\xi,\mathfrak{m}}^{K_0(Q),\mathbb{T}}$ compatible with the respective morphisms to $\mathscr{W}_F$, and the result follows.


\end{proof}

We may also compare slope data for $\mathscr{X}_{\underline D^\times}^{K_0(Q)}$ and $\mathscr{X}_{\underline D^\times}^{K^-(Q)}$:  


\begin{lemma}
	Fix a central character $\xi:\A_{F,f}^\times/F^\times\rightarrow \O(\mathscr{W})^\times$.  Let $U=\Spa R\subset \mathscr{W}_F$ be an affinoid open, corresponding to a weight $\kappa$, and fix $h\in\mathbb{Q}_{>0}$.  Then $(U,h)$ is a slope datum for $\mathscr{X}_{\underline D,\xi}^{K_0(Q)}$ if and only if it is a slope datum for $\mathscr{X}_{\underline D,\xi}^{K^-(Q)}$.
\label{lemma: slope data change of level K-Q}
\end{lemma}
\begin{proof}
	Let $\mathscr{Z}$ and $\mathscr{Z}'$ be the spectral varieties for $\mathscr{X}_{\underline D,\xi}^{K_0(Q)}$ and $\mathscr{X}_{\underline D,\xi}^{K^-(Q)}$, respectively, and let $x\in U$ be a maximal point.  Then by Proposition~\ref{prop: fin proj diamond operators} the module $H^0(K^-(Q),\mathscr{D}_{\kappa_x})[\xi]_{\leq h}$ is finite projective over $k_x[\Delta_Q]$ of constant rank, and the natural map $H^0(K^-(Q),\mathscr{D}_{\kappa_x})[\xi]_{\leq h,\Delta_Q}\rightarrow H^0(K_0(Q),\mathscr{D}_{\kappa_x})_{\leq h}$ is an isomorphism.  It follows that the fiber of $Z_{U,h}$ over $x$ is finite of order $d$ if and only if the fiber of $\mathscr{Z}'$ over $x$ is finite of order $d\lvert\Delta_Q\rvert$. Since spectral varieties are flat over weight space, the result follows from Theorem~\ref{thm: finiteness criterion}.
\end{proof}

\section{Patching and modularity}\label{section: patching}

\subsection{Set-up}

Let us recall our goal.  Assume $p\geq 5$.  Fix a non-archimedean characteristic $p$ local field $L$ with ring of integers $\O_L$, residue field $\F_q$, and uniformizer $u$.  Fix a continuous odd representation $\overline\rho:\Gal_{\Q}\rightarrow \GL_2(\F_q)$, such that:
\begin{itemize}
	\item	$\overline\rho$ is modular
	\item	$\overline\rho|_{\Gal_{\Q(\zeta_p)}}$ is absolutely irreducible
	\item	The image of $\overline\rho$ contains $\SL_2(\F_p)$
	\item	$\overline\rho$ is unramified at all places away from $p$
	\item	$\overline\rho \not\sim \chi\otimes\left(\begin{smallmatrix} \overline\chi_{\cyc} & \ast \\ & 1\end{smallmatrix}\right)$ for any character $\chi:\Gal_{\Q}\rightarrow\F_q^\times$.
\end{itemize}
The assumption that $\overline\rho$ has large image is stronger than the typical hypothesis. This is because we need to use ~\cite[Theorem B.0.1]{bergdall-hansen} to ensure that we can work with middle-degree eigenvarieties for Hilbert modular forms.

We wish to prove the following modularity theorem:
\begin{thm}\label{thm: main thm precise}
	Suppose $\rho:\Gal_\Q\rightarrow \GL_2(\O_L)$ is a continuous odd representation unramified away from $p$ and trianguline at $p$ with regular parameters, whose reduction modulo $u$ is as above.  Then $\rho$ is the twist of a Galois representation arising from an overconvergent modular form.
\end{thm}

The predicted weight $\kappa$ can be read off from the parameters of the triangulation,
as can the predicted slope $h$.

More precisely, we will show that $\rho$ corresponds to a class in $S_{\kappa}(K)_{\leq h}$, where $K=I\cdot K_1(N)^p = I\cdot \prod_{\ell\neq p,\ell\nmid N}\GL_2(\Q_\ell)\cdot \prod_{\ell\mid N}K_1(\ell)$ for some $N\geq 5$ prime to $p$.  To do this, we will consider an open weight $\kappa:T_0\rightarrow \O(U)^\times$, where $U\subset\mathscr{W}$ contains a point corresponding to $\kappa$ and $(U,h)$ is a slope datum, and we will study the spaces $S_\kappa(K^-(Q))_{\leq h}$ for varying sets of primes $Q$.

\subsection{Patched eigenvarieties}

In this section, we construct local pieces of patched quaternionic eigenvarieties, using the language of ultrafilters of ~\cite[\textsection 9]{scholze-lubin-tate}.   We fix a totally real field $F$ split at all places above $p$ and a totally definite quaternion algebra $D$ over $F$, which is ramified at all infinite places and split at all finite places.  We also fix the tame level $K^p:=\GL_2(\A_{F,f}^p)$.  We further assume that $F/\Q$ is abelian, so that Leopoldt's conjecture is known to hold.  Unlike ~\cite{scholze-lubin-tate}, we do not assume that $F$ has a unique prime above $p$; we let $\Sigma_p:=\{v\mid p\}$.  We expect these hypotheses can be relaxed considerably, but this is not necessary for our applications.  Fix some finite extension $E/\Q_p$ with residue field containing $\mathbf{F}_q$.

Recall that there are Galois deformation rings $R_{\overline\rho,\Sigma_p}^{\square}$ and $R_{\overline\rho,\Sigma_p}$, parametrizing deformations of $\overline\rho$ unramified outside of $\Sigma_p$, where $R_{\overline\rho}^{\square}$ additionally parametrizes framings of the deformations at places of $\Sigma_p$. There is also a local framed deformation ring $R_{\overline\rho,\loc}^{\square}:=\htimes_{v\in\Sigma_p}R_{\overline\rho_v}^{\square}$, where $R_{\overline\rho_v}^{\square}$ parametrizes framed deformations of $\overline\rho|_{\Gal_{F_v}}$, and there is a natural map $R_{\overline\rho,\loc}^{\square}\rightarrow R_{\overline\rho}^{\square}$.

We define a distinguished family of characters $\eta_{\mathrm{univ}}:\Gal_F\rightarrow \Z_p[\![T_0/\overline{Z(K)}]\!]^\times$ over integral weight space.  We have a universal weight $\underline\lambda=(\lambda_1,\lambda_2)$, where each $\lambda_i$ is a character $\prod_{v\in\Sigma_p}\O_{F_v}^\times\rightarrow \Z_p[\![T_0/\overline{Z(K)}]\!]^\times$, and we define $\eta_v:\O_{F_v}^\times\cong\Z_p^\times\rightarrow \Z_p[\![T_0/\overline{Z(K)}]\!]^\times$ via $\eta(x):=\left(\lambda_1|_{\O_{F_v}^\times}(x)\lambda_2|_{\O_{F_v}^\times}(x)\right)^{-1}$.  Then because we have assumed that Leopoldt's conjecture holds for $F$, we see that $\eta_v$ is independent of $v\in\Sigma$; global class field theory gives us a corresponding character $\Gal_{\Q}\rightarrow \Z_p[\![T_0/\overline{Z(K)}]\!]^\times$, which we restrict to $\Gal_F$ to obtain $\eta_{\mathrm{univ}}$.

We fix an unramified continuous character $\psi_0:\Gal_F\rightarrow \O_E[\![T_0/\overline{Z(K)}]\!]^\times$ such that the reduction $\overline\psi_0$ modulo the maximal ideal satisfies $\det\overline\rho=\overline\psi_0\overline\eta_{\mathrm{univ}}\overline\chi_{\cyc}^{-1}$, and we set $\psi:=\psi_0\eta_{\mathrm{univ}}$ and $\psi':=\psi_0\eta_{\mathrm{univ}}\chi_{\cyc}^{-1}$.  Then we constructed quotients
	\begin{align*}
		\O_E[\![T_0/\overline{Z(K)}]\!]\htimes R_{\overline\rho,\Sigma_p}^{\square}&\twoheadrightarrow R_{\overline\rho,\Sigma_p}^{\square,\psi'}	\\
		\O_E[\![T_0/\overline{Z(K)}]\!]\htimes R_{\overline\rho,\loc}^{\square}&\twoheadrightarrow R_{\overline\rho,\loc}^{\square,\psi'}	\\
	\O_E[\![T_0/\overline{Z(K)}]\!]\htimes R_{\overline\rho,\Sigma_p}&\twoheadrightarrow R_{\overline\rho,\Sigma_p}^{\psi'}
	\end{align*}
parametrizing families of deformations with fixed determinants.

We also define families of weights $\underline\kappa_v$ over $\mathscr{W}_F$ via
\[	\underline\kappa_v=(\kappa_{v_1},\kappa_{v,2}) = \left(\lambda_2|_{\O_{F_v}^\times}^{-1}, \lambda_1|_{\O_{F_v}^\times}^{-1}\chi_{\cyc}^{-1}\right)	\]

In order to find sets of Taylor--Wiles primes, we impose the following standard hypotheses:
\begin{enumerate}
	\item	$p\geq 5$
	\item	$\overline\rho|_{F(\zeta_p)}$ is absolutely irreducible
	\item	If $p=5$ and $\overline\rho$ has projective image $\PGL_2(\F_5)$, the kernel of $\overline\rho$ does not fix $F(\zeta_5)$
\end{enumerate}
Then we have the following relative version of ~\cite[Proposition 2.2.4]{kisin2009} (since we assumed $p$ splits completely in $F$, $[F:\Q]=\lvert\Sigma_p\rvert$):
\begin{prop}
	Let $g:=\dim_{\F_q}H^1(\Gal_{F,\Sigma_p},\ad^0\overline\rho(1))-1$.  Then for each positive integer $n$, there exists a finite set $Q_n$ of places of $F$, disjoint from $\Sigma_p$, of cardinality $g+1$, such that 
	\begin{enumerate}
		\item	for all $v\in Q_n$, $\Nm(v)\equiv 1\pmod{p^n}$, and $\overline\rho(\Frob_v)$ has distinct eigenvalues
		\item	the global relative Galois deformation ring $R_{\overline\rho,\Sigma_p\cup Q_n}^{\square,\psi'}$ parametrizing families of deformations with determinant $\psi$ unramified outside $\Sigma_p\cup Q_n$ can be topologically generated as an $R_{\overline\rho,\loc}^{\square,\psi'}$-algebra by $g$ elements.
	\end{enumerate}
\end{prop}
\begin{proof}
	This follows from Lemma~\ref{lemma: local global def ring top gens}, as in~\cite[Proposition 3.2.5]{kisin2009moduli}.
\end{proof}

We fix such a set $Q_n$ for each $n\geq 1$, as well as a non-principal ultrafilter $\mathfrak{F}$ on $\{n\geq 1\}$ (more precisely, on its power set, ordered by inclusion).  For notational convenience, we set $Q_0:=\emptyset$, and we let $Q_n':=Q_n\cup\Sigma_p$.
For each $n$, we again let $K^-(Q_n)\subset K_0(Q_n)\subset G(\A_{F,f}^p)\cong \GL_2(\A_{F,f}^p)$ be the compact open subgroups 
\[	K^-(Q_n):=\prod_{v\notin Q_n}\GL_2(\O_{F_v})\times\prod_{v\in Q_n}K^-(v)\subset \prod_{v\notin Q_n}\GL_2(\O_{F_v})\times\prod_{v\in Q_n}K_0(v)	\]
Let $\xi:\A_{F,f}^\times/F^\times\rightarrow \O(\mathscr{W}_F)^\times$ be the central character corresponding to $\psi$ via class field theory. 

Now we analyze the eigenvarieties $\mathscr{X}_{\underline D^\times}^{K^-(Q_n)}$.  
Let $\kappa$ be a weight valued in a reduced pseudoaffinoid $\Z_p$-algebra $R$, and write $U:=\Spa R$.  Assume that the Fredholm determinant corresponding to $H^0(K,\mathscr{D}_\kappa)[\xi]$ admits a slope-$\leq h$-factorization for some slope $h=a/b$ (where $a,b$ are relatively prime non-negative integers); by Corollary ~\ref{cor: slope data change of level K0Q} and Lemma~\ref{lemma: slope data change of level K-Q}, the Fredholm determinants corresponding to $H^0(K_0(Q_n),\mathscr{D}_\kappa)[\xi]$ and $H^0(K^-(Q_n),\mathscr{D}_\kappa)[\xi]$ also admit slope-$\leq h$-factorizations. 
We also assume that $R$ can be equipped with a norm adapted to $\kappa$ such that the corresponding unit ball is the ring of definition $R_0=R^\circ$; this is possible, for example, if $U$ is an affinoid open or a maximal point in $\mathscr{W}_F$, by Lemma~\ref{lemma: norm adapted univ char}.  Then we fix some $r>r_\kappa$.  

The modularity of the residual representation $\overline\rho$ means that $\overline\rho$ corresponds to a maximal ideal $\mathfrak{m}\subset\mathbb{T}$.  
For each $v\in Q_n$, we fix a root $\alpha_v$ of the characteristic polynomial $X^2-T_vX+\Nm(v)S_v$ of $\overline\rho(\Frob_v)$ (increasing $\F_q$, and hence $E$, if necessary), and we consider the corresponding maximal ideal $\mathfrak{m}_{Q_n}\subset\mathbb{T}_{Q_n}^-$ (as in \textsection~\ref{subsect: sub-eigenvarieties}).

Then we have a collection of diagrams
\[
	\begin{tikzcd}
		\O_E\times\mathscr{X}_{\underline D^\times,\xi,\mathfrak{m}_{Q_n}}^{K^-(Q_n)}\ar[r] \arrow[d, "\mathrm{wt}"] & \coprod_{\overline\rho}\Spa R_{\overline\rho,Q_n\cup\Sigma_p}	\\
		\O_E\times\mathscr{W}_F
	\end{tikzcd}
\]
The pre-image $\mathrm{wt}^{-1}(U)$ in the slope-$\leq h$ part of the eigenvariety has the form $\Spa\left(\mathbb{T}_{K^-(Q_n),\kappa,\xi,\leq h},\mathbb{T}_{K^-(Q_n),\kappa,\xi,\leq h}^\circ\right)$, and since $\mathscr{X}_{\underline D^\times,\xi,\mathfrak{m}_{Q_n}}^{K^-(Q_n)}$ is reduced, $\mathbb{T}_{K^-(Q_n),\kappa,\xi,\leq h}^\circ\subset \mathbb{T}_{K^-(Q_n),\kappa,\xi,\leq h}$ is a ring of definition.  

For each $n$, the module of overconvergent modular forms $S_{\kappa,\xi}(K^-(Q_n))_{\leq h}$ is a $\mathbb{T}_{K^-(Q_n),\kappa,\xi,\leq h}$-module, and it is projective as an $R$-module; let $d$ be its rank over $R$.  The $\mathbb{T}_{K^-(Q_n),\kappa,\xi,\leq h}^{<r,\circ}$-submodule $S_{\kappa,\xi}^{<r,\circ}(K^-(Q_n))_{\leq h}$ is a lattice in $S_{\kappa,\xi}(K^-(Q_n))_{\leq h}$.  Recall from Definition~\ref{def: overconv quat mod forms} that there is a second lattice
\[	S_{\kappa,\xi}^{\circ}(K^-(Q_n))_{\leq h}:=\sum_{v\mid p}\sum_{i\geq 0} (u^aU_{\varpi_v}^{-b})^i\left( S_{\kappa,\xi}^{<r,\circ}(K^-(Q_n))_{\leq h} \right)	\]
which is stable under the operators $u^aU_{\varpi_v}^{-b}$, as well.

Let $R_{\overline\rho,Q_n'}^{\psi'}|_U$ denote the localization $R_0\htimes_{\O_E[\![T_0/\overline{Z(K)}]\!]} R_{\overline\rho,Q_n'}^{\psi'}$; the formal scheme $\Spf R_{\overline\rho,Q_n'}^{\psi'}|_U$ is an integral model for the pseudorigid space $U\times \Spa R_{\overline\rho,Q_n'}^{\psi'}$.  Similarly, we will write $R_{ {\tri},\overline\rho,Q_n',\leq h}^{\square,\psi',\underline\kappa}|_U$ for  the localization $R_0\htimes_{\O_E[\![T_0/\overline{Z(K)}]\!]}R_{ {\tri},\overline\rho,Q_n',\leq h}^{\square,\psi',\underline\kappa}$ and $R_{ {\tri},\overline\rho,\loc,\leq h}^{\square,\psi',\underline\kappa}|_{U}$ for the localization $R_0\htimes_{\O_E[\![T_0/\overline{Z(K)}]\!]}R_{ {\tri},\overline\rho,\loc,\leq h}^{\square,\psi',\underline\kappa}$.
Using the existence of Galois representations, we see that $\mathbb{T}_{K^-(Q_n),\kappa,\xi,\leq h}^{\circ}$ is a $R_{\overline\rho,Q_n'}^{\psi'}|_U$-algebra.

By Lemma~\ref{prop: fin proj diamond operators} $S_{\kappa,\xi}^{<r,\circ}(K^-(Q_n))_{\leq h,\mathfrak{m}_{Q_n}^-}$ is a finite $R_0[\Delta_{Q_n}]$-module, with 
\[	R_0\otimes_{R_0[\Delta_{Q_n}]}S_{\kappa,\xi}^{<r,\circ}(K^-(Q_n))_{\leq h,\mathfrak{m}_{Q_n}^-}\cong S_{\kappa,\xi}^{<r,\circ}(K_0(Q_n))_{\leq h,\mathfrak{m}_{0,Q_n}}	\]
Since the augmentation ideal $I_{\Delta_{Q_n}}$ is contained in the Jacobson radical of $R_0[\Delta_{Q_n}]$, this implies that $S_{\kappa,\xi}^{<r,\circ}(K_0(Q_n))_{\leq h,\mathfrak{m}_{0,Q_n}}$ and $S_{\kappa,\xi}^{<r,\circ}(K^-(Q_n))_{\leq h,\mathfrak{m}_{Q_n}^-}$ can be generated by the same number of elements (over $R_0$ and $R_0[\Delta_{Q_n}]$, respectively).

	Similarly, $S_{\kappa,\xi}^{\circ}(K^-(Q_n))_{\leq h,\mathfrak{m}_{Q_n}^-}$ is a finite $R_0[\Delta_{Q_n}]$-module. Since $S_{\kappa,\xi}^{\circ}(K^-(Q_n))_{\leq h,\mathfrak{m}_{Q_n}^-}$ is generated by $d^{\lvert\Sigma_p\rvert}$ translates of $S_{\kappa,\xi}^{<r,\circ}(K^-(Q_n))_{\leq h,\mathfrak{m}_{Q_n}^-}$, we see that the number of generators of $S_{\kappa,\xi}^{\circ}(K^-(Q_n))_{\leq h,\mathfrak{m}_{Q_n}^-}$ over $R_0[\Delta_{Q_n}]$ is bounded independently of $n$.

Set $j=4\lvert\Sigma_p\rvert-1$ and $k=\lvert Q_n\rvert = g+1$.  Using local-global compatibility at places in $Q_n$, there is a homomorphism $R_0\htimes \Z_p[\![y_1,\ldots,y_{k}]\!]\rightarrow R_{\overline\rho,Q_n'}^{\psi'}|_U$ such that the action of $R_0\htimes \Z_p[\![y_1,\ldots,y_{k}]\!]$ on $S_{\kappa,\xi}^{\circ}(K^-(Q_n))^\circ_{\leq h,\mathfrak{m}_{Q_n}^-}$ is compatible with the action of $R_0[\Delta_{Q_n}]$ via a fixed surjection $R_0\htimes\Z_p[\![y_1,\ldots,y_{k}]\!]\rightarrow R_0[\Delta_{Q_n}]$.  

We observe that we may view $S_{\kappa,\xi}(K^-(Q_n))_{\leq h,\mathfrak{m}_{Q_n}^-}$ as a module over $\Spa R[\Delta_Q]\times_U C_{U,h}$, where $C_{U,h}$ is the annulus of radius $h$, by letting the coordinate on $C_{U,h}$ act as $U_p^{-1}$.  

Now we consider local-global compatibility at places in $\Sigma_p$.   Recall that the actions of $u^{a}U_{\varpi_v}^b$ and $u^{a}U_{\varpi_v}^{-b}$ on $S_{\kappa,\xi}(K^-(Q_n))_{\leq h}$ are power-bounded for all $v\mid p$.  Thus, we can make $S_{\kappa,\xi}^{\circ}(K^-(Q_n))_{\leq h,\mathfrak{m}_{Q_n}^-}$ into a module over $R_{\overline\rho,Q_n'}^{\psi'}|_U\left\langle p^{h}T_1^{\pm 1},\ldots,p^{h}T_{\lvert\Sigma_p\rvert}^{\pm 1}\right\rangle$ by letting $T_i$ act as $U_{\varpi_{v_i}}^{-1}$.  But local-global compatibility tells that over the analytic locus, $S_{\kappa,\xi}(K^-(Q_n))$ is supported on the trianguline locus, so $S_{\kappa,\xi}^{\circ}(K^-(Q_n))_{\leq h,\mathfrak{m}_{Q_n}^-}$ is actually a 
$R_{ {\tri},\overline\rho,Q_n',\leq h}^{\psi',\underline\kappa}|_U$-module, where the coordinates of $\G_m^{\Sigma_p}$ act as $U_{\varpi_v}^{-1}$.

Since $R_{\overline\rho,Q_n'}^{\psi'}\rightarrow R_{\overline\rho,Q_n'}^{\square,\psi'}$ is formally smooth of dimension $j$, we may construct a homomorphism 
\[	R_0\htimes\Z_p[\![y_1,\ldots,y_{k},y_{k+1},\ldots,y_{k+j}]\!]\rightarrow R_{{\tri},{\overline\rho},{Q_n'},{\leq h}}^{\square,\psi',\underline\kappa}|_U	\]
compatible with 
\[	R_0\htimes\Z_p[\![y_1,\ldots,y_{k}]\!]\rightarrow R_{\overline\rho,Q_n'}^{\psi'}|_U	\]
such that $y_{k+1},\ldots,y_{k+j}$ are the framing variables.  Finally, we fix a surjection $R_{\overline\rho,\loc}^{\square,\psi'}[\![x_1,\ldots,x_g]\!]\twoheadrightarrow R_{\overline\rho,Q_n'}^{\square,\psi'}$ and a map $R_0\htimes\Z_p[\![y_1,\ldots,y_{k+j}]\!]\rightarrow R_{\overline\rho,\loc}^{\square,\psi'}[\![x_1,\ldots,x_g]\!]$ such that the corresponding diagram
\[
	\begin{tikzcd}
		R_0\htimes\Z_p[\![y_1,\ldots,y_{k+j}]\!] \ar[r]\ar[rd] & R_{{\tri},\overline\rho,{\loc},\leq h}^{\square,\psi',\underline\kappa}[\![x_1,\ldots,x_g]\!]|_U \ar[d]	\\
		& R_{{\tri},\overline\rho,Q_n',\leq h}^{\square,\psi',\underline\kappa}|_U
	\end{tikzcd}
\]
commutes.

Now we can patch.  We add framing variables by setting
\[	M_n^{<r}:=\Z_p[\![y_{k+1},\ldots,y_{k+j}]\!]\htimes_{\Z_p}S_{\kappa,\xi}^{<r,\circ}(K_{Q_n}^-)_{\leq h,\mathfrak{m}_{Q_n}^-}	\]
and
\[	M_n:= \Z_p[\![y_{k+1},\ldots,y_{k+j}]\!]\htimes_{\Z_p}S_{\kappa,\xi}^{\circ}(K_{Q_n}^-)_{\leq h,\mathfrak{m}_{Q_n}^-}	\]
so that 
\[	R_0\otimes_{R_0\htimes\Z_p[\![y_1,\ldots,y_{k+j}]\!]}M_n^{<r}\cong S_{\kappa,\xi}^{<r,\circ}(K)_{\leq h,\mathfrak{m}}	\]
for all $n\geq 1$, and 
\[	R_0\otimes_{R_0\htimes\Z_p[\![y_1,\ldots,y_{k+j}]\!]}M_n\twoheadrightarrow S_{\kappa,\xi}^\circ(K)_{\leq h,\mathfrak{m}}	\]
for all $n\geq 1$.

For any open ideal $I\subset R_0\htimes\Z_p[\![y_i]\!]$, we define
\[	M_{I}^{<r}:=R_0\htimes\Z_p[\![y_i]\!]/I\otimes_{\prod_{n\geq 1}R_0\htimes\Z_p[\![y_i]\!]/I}\prod_nM_n^{<r}/I	\]
and
\[	M_{I}:=R_0\htimes\Z_p[\![y_i]\!]/I\otimes_{\prod_{n\geq 1}R_0\htimes\Z_p[\![y_i]\!]/I}\prod_nM_n/I	\]
Here the homomorphism $\prod_{n\geq 1}R_0\htimes\Z_p[\![y_i]\!]/I\rightarrow R_0\htimes\Z_p[\![y_i]\!]/I$ is the localization map coming from our choice of non-principal ultrafilter.

Passing to the inverse limit, we obtain the patched modules 
\[	M_\infty^{<r}:=\varprojlim_I M_{I}^{<r}	\]
and
\[	M_\infty:=\varprojlim_I M_{I}	\]

Similarly, we may define patched global deformation rings $R_{{\tri},\overline\rho,\infty,\leq h,I}^{\square,\psi',\underline\kappa}|_U$ and Hecke algebras $\mathbb{T}_{\infty,\kappa,\leq h,I}^{<r,\circ}$ and $\mathbb{T}_{\infty,\kappa,\leq h,I}^\circ$ via
\[	R_{ {\tri},\overline\rho,\infty,\leq h,I}^{\square,\psi',\underline\kappa}|_ U := R_0\htimes\Z_p[\![y_i]\!]/I\otimes_{\prod_{n\geq 1}R_0\htimes\Z_p[\![y_i]\!]/I}\prod_n R_{ {\tri},\overline\rho,Q_n',\leq h}^{\square,\psi',\underline\kappa}|_U/I	\]
\[	\mathbb{T}_{\infty,\kappa,\leq h,I}^{<r,\circ}:= R_0\htimes\Z_p[\![y_i]\!]/I\otimes_{\prod_{n\geq 1}R_0\htimes\Z_p[\![y_i]\!]/I}\prod_n \mathbb{T}_{K^-(Q_n),\kappa,\leq h}^{<r,\circ}/I	\]
\[	\mathbb{T}_{\infty,\kappa,\leq h,I}^{\circ}:= R_0\htimes\Z_p[\![y_i]\!]/I\otimes_{\prod_{n\geq 1}R_0\htimes\Z_p[\![y_i]\!]/I}\prod_n \mathbb{T}_{K^-(Q_n),\kappa,\leq h}^{\circ}/I	\]
Setting $R_{ {\tri},\overline\rho,\infty,\leq h}^{\square,\psi',\underline\kappa}|_U:=\varprojlim_I R_{ {\tri},\overline\rho,\infty,\leq h,I}^{\square,\psi',\underline\kappa}|_ U$, $\mathbb{T}_{\infty,\kappa,\leq h}^{<r,\circ}:=\varprojlim_I \mathbb{T}_{\infty,\kappa,\leq h,I}^{<r,\circ}$, and $\mathbb{T}_{\infty,\kappa,\leq h}^\circ:=\varprojlim_I \mathbb{T}_{\infty,\kappa,\leq h,I}^\circ$, we have a sequence of homomorphisms
\[	R_0\htimes\Z_p[\![y_i]\!]\rightarrow R_{{\tri},\overline\rho,{\loc},\leq h}^{\square,\psi',\underline\kappa}[\![x_i]\!]|_U\rightarrow R_{ {\tri},\overline\rho,\infty,\leq h}^{\square,\psi',\underline\kappa}|_U\rightarrow \mathbb{T}_{\infty,\kappa}^{\circ}	\]
compatible with their actions on $M_\infty$.

Note that for each open ideal $I$, we have a surjection
\[	R_{ {\tri},\overline\rho, {\loc},\leq h}^{\square,\psi',\underline\kappa}[\![x_i]\!]|_U/I\twoheadrightarrow R_{ {\tri},\overline\rho,\infty,\leq h,I}^{\square,\psi',\underline\kappa}|_ U	\]
Hence we have a surjection 
\[	R_{ {\tri},\overline\rho, {\loc},\leq h}^{\square,\psi',\underline\kappa}[\![x_i]\!]|_U\twoheadrightarrow R_{ {\tri},\overline\rho,\infty,\leq h}^{\square,\psi',\underline\kappa}|_ U \]
and a closed immersion
\[	X_{ {\tri},\overline\rho,\infty,\leq h}^{\square,\psi',\underline\kappa}|_U:=\left(\Spa R_{ {\tri},\overline\rho,\infty,\leq h}^{\square,\psi',\underline\kappa}|_U\right)^{\an}\hookrightarrow X_{ {\tri},\overline\rho, {\loc},\leq h}^{\square,\psi',\underline\kappa}[\![x_i]\!]|_U	\]
Furthermore, since $R_{ {\tri},\overline\rho,Q_n',\leq h}^{\square,\psi',\underline\kappa}|_U/(y_i)\cong R_{ {\tri},\overline\rho,\Sigma_p,\leq h}^{\square,\psi',\underline\kappa}|_U$ for all $n$, we see that $R_{ {\tri},\overline\rho,\infty,\leq h}^{\square,\psi',\underline\kappa}|_U/(y_i)\cong R_{ {\tri},\overline\rho,\Sigma_p,\leq h}^{\square,\psi',\underline\kappa}|_U$.

\begin{lemma}
	The patched modules $M_\infty^{<r}$ and $M_\infty$ are finite over $R_0\htimes\Z_p[\![y_i]\!]$.  In particular, they are complete.
	\label{lemma: patched modules finite}
\end{lemma}
\begin{proof}
	The powers of the ideal $(u,y_1,\ldots,y_{k+j})$ are cofinal in the set of open ideals of $R_0\htimes\Z_p[\![y_i]\!]$, and for any open ideals $I\subset I'\subset R_0\htimes\Z_p[\![y_i]\!]$, the natural maps $M_I^{<r}/I'\rightarrow M_{I'}^{<r}$ and $M_I/I'\rightarrow M_{I'}$ are isomorphisms.  Then ~\cite[Tag 09B8]{stacks-project} implies that $M_\infty^{<r}$ and $M_\infty$ are complete and $M_\infty^{<r}/I\cong M_I^{<r}$ and $M_\infty/I\cong M_I$ for all open ideals $I\subset R_0\htimes\Z_p[\![y_i]\!]$. 

	We have 
	\[	M_\infty^{<r}/(u,y_1,\ldots,y_{k+j})\cong M_{(u,y_1,\ldots,y_{k+j})}^{<r}\cong S_{\kappa,\xi}^{<r,\circ}(K)_{\leq h,\mathfrak{m}}/u	\]
	which is $R_0$-finite.  Since the number of generators of $M_n/(u,y_1,\ldots,y_{k+j})$ over $R_0$ is bounded independently of $n$, $M_\infty/(u,y_1,\ldots,y_{k+j})$ is $R_0$-finite.  Hence by ~\cite[Theorem 8.4]{matsumura}, $M_\infty^{<r}$ and $M_\infty$ are $R_0\htimes\Z_p[\![y_i]\!]$-finite.
\end{proof}

\begin{prop}
	If $R=L$ is a field with ring of integers $\O_L$, then $M_\infty^{<r}$ is a finite projective $\O_L[\![y_1,\ldots,y_{k+j}]\!]$-module.
\end{prop}
\begin{proof}
	We claim it is enough to show that for any open ideal $I$, $M_n^{<r}/I$ is a free $\O_L[\![y_i]\!]/I$-module of rank $d$ for all $n\gg0$.  Indeed, because our ultrafilter is non-principal, this implies that $M_I^{<r}$ is also a free $\O_L[\![y_i]\!]/I$-module of rank $d$ (since the localization $\prod_{n\geq 1}\O_L[\![y_i]\!]/I\rightarrow \O_L[\![y_i]\!]/I$ factors through the localization $\prod_{n\geq 1}\O_L[\![y_i]\!]/I\rightarrow\prod_{n\geq n_0}\O_L[\![y_i]\!]/I$ for any $n_0\geq 1$).  Since $M_\infty^{<r}$ is $(u,y_1,\ldots,y_{k+j})$-adically separated, ~\cite[Theorem 22.3]{matsumura} implies that $M_\infty^{<r}$ is flat over $\O_L[\![y_i]\!]$, and hence projective.

	By Proposition~\ref{prop: fin proj diamond operators}, $S_{\kappa,\xi}^{<r,\circ}(K^-(Q_n))_{\leq h}$ is a projective $R_0[\Delta_{Q_n}]$-module of rank $d$ for all $n$.  Then for $n\gg0$ (depending on $I$), $M_n^{<r}/I$ is free over $\O_L[\![y_i]\!]$ of rank $d$, so we are done.
\end{proof}

The modules $M_\infty^{<r}$ and $M_\infty$ behave well under finite base change, in particular, under passage to closed subspaces of $U$:
\begin{lemma}
	Let $f:R_0\rightarrow R_0'$ be a finite morphism, where $R_0'$ is a noetherian ring of definition in a pseudoaffinoid algebra.  Let $\kappa'$ be the weight $f\circ\kappa$, and let ${M_\infty'}^{<r}$ denote the patched module constructed from the modules of modular forms $S_{\kappa',\xi}^{<r,\circ}(K_{Q_n}^-)_{\leq h,\mathfrak{m}_{Q_n}^-}$.  Then the natural maps 
		\[	R_0'\htimes\Z_p[\![y_i]\!]\otimes_{R_0\htimes\Z_p[\![y_i]\!]}M_\infty^{<r}\rightarrow {M_\infty^{<r}}'	\]
		and
		\[	R_0'\htimes\Z_p[\![y_i]\!]\otimes_{R_0\htimes\Z_p[\![y_i]\!]}M_\infty\rightarrow M_\infty'	\]
		are surjections.
	\label{lemma: patched module finite base change}
\end{lemma}
\begin{proof}
	We treat the first map; the second is similar.  Let $M_n':=\Z_p[\![y_{k+1},\ldots,y_{k+j}]\!]\otimes_{\Z_p}S_{\kappa',\xi}^{<r,\circ}(K_{Q_n}^-)_{\leq h,\mathfrak{m}_{Q_n}^-}$.  The open ideals $I\subset R_0\htimes\Z_p[\![y_i]\!]$ generate open ideals of $R_0'\htimes\Z_p[\![y_i]\!]$ and are cofinal, so it suffices to show that we have a surjection
	\[	R_0'\htimes\Z_p[\![y_i]\!]/I\otimes_{R_0\htimes\Z_p[\![y_i]\!]/I}M_I\rightarrow M_I':=\left(R_0'\htimes\Z_p[\![y_i]\!]/I\right)\otimes_{\prod_{n\geq 1}R_0'\htimes\Z_p[\![y_i]\!]/I}\prod_{n\geq 1}M_n'/I	\]
	The left side is isomorphic to $R_0'\otimes_{R_0}M_I$ (because $R_0\htimes\Z_p[\![y_i]\!]/I$ is discrete, by construction).  Since each map $R_0'\otimes_{R_0}M_n\rightarrow M_n'$ is surjective (by Lemma~\ref{lemma: base change Dcirc}, and since the transition maps $\prod_{n=1}^{k+1}M_n\rightarrow \prod_{n=1}^kM_n$ are surjective, the Mittag-Leffler condition implies that the natural map
	\[	R_0'\otimes_{R_0}M_I\rightarrow M_I'	\]
	is surjective.
\end{proof}

We have contructed two coherent $R_0\htimes\Z_p[\![y_i]\!]$-modules, $M_\infty^{<r}$ and $M_\infty$; $M_\infty$ is naturally a $R_{ {\tri},\overline\rho,\infty,\leq h}^{\square,\psi',\underline\kappa}|_U$-module, but $M_\infty^{<r}$ is projective when $U$ is a point, making its support over $R_0\htimes\Z_p[\![y_i]\!]$ easier to analyze. 

We now pass to the loci of the corresponding map 
\[	\Spa R_{{\tri},\overline\rho,{\loc},\leq h}^{\square,\psi',\underline\kappa}[\![x_i]\!]|_U\rightarrow \Spa R_0\htimes\Z_p[\![y_i]\!]	\]
where $u\neq 0$, and we consider the analytification $M_\infty^{\an}$ of $M_\infty$ as a coherent sheaf over $X_{{\tri},\overline\rho,{\loc},\leq h}^{\square,\psi',\underline\kappa}|_U\times\Spa\Z_p[\![x_i]\!]$.  

\begin{lemma}
	The support of $M_\infty^{\an}$ is a Zariski-closed subspace of dimension
\[	\dim \Spa R_0\htimes\Z_p[\![y_i]\!]\left[\frac 1 u\right] = \dim U + (g+1)+(4\lvert\Sigma_p\rvert-1) = \dim U + g + 4\lvert\Sigma_p\rvert	\]
	\label{lemma: dim supp M_infty}
\end{lemma}
\begin{proof}
	If $x:R\rightarrow L$ is a maximal point, and $\O_L$ is the ring of integers of $L$, it suffices to show that $\O_L\otimes_{R_0}M_\infty$ is supported on all of $\Spec \O_L[\![y_i]\!]$.  We set $\kappa':=x\circ\kappa$ and we let ${M_\infty'}^{<r}$ and $M_\infty'$ be the patched modules constructed from the modules $S_{\kappa',\xi}^{<r,\circ}(K^-(Q_n))_{\leq h,\mathfrak{m}_{Q_n}^-}$ and $S_{\kappa',\xi}^\circ(K^-(Q_n))_{\leq h,\mathfrak{m}_{Q_n}^-}$, respectively.  Since the natural map
	\[	\O_L\otimes_{R_0}M_\infty\rightarrow M_\infty'	\]
	is surjective, it suffices to show that $M_\infty'$ is supported on all of $\Spec \O_L[\![y_i]\!]$.  To see this, we consider the natural morphism ${M_\infty'}^{<r}\rightarrow M_\infty'$.

	We will show that ${M_\infty'}^{<r}\rightarrow M_\infty'$ is an isomorphism over a dense open subspace of $\Spec \O_L[\![y_i]\!]$.  Let $P_n$ be the cokernel of ${M_n'}^{<r}\rightarrow M_n'$, and let $P$ be the cokernel of ${M_\infty'}^{<r}\rightarrow M_\infty'$.  Since the cokernel of
	\[	S_{\kappa,\xi}^{<r,\circ}(K^-(Q_n))_{\leq h,\mathfrak{m}_{Q_n}^-}\rightarrow S_{\kappa,\xi}^\circ(K^-(Q_n))_{\leq h,\mathfrak{m}_{Q_n}^-}	\]
	is finite and $u$-power-torsion, $P_n$ is also $u$-power-torsion.

	There is some integer $k_0\geq 0$ such that 
	\[	u^{k_0}S_{\kappa',\xi}^\circ(K_0(Q_n))_{\leq h,\mathfrak{m}_{0,Q_n}}\subset S_{\kappa',\xi}^{<r.\circ}(K_0(Q_n))_{\leq h,\mathfrak{m}_{0,Q_n}}	\]
	and by Corollary~\ref{cor: control new lattice} the kernel of 
	\[	\left( S_{\kappa,\xi}^\circ(K^-(Q_n))_{\leq h,\mathfrak{m}_{Q_n}^-} \right)_{\Delta_{Q_n}}\rightarrow S_{\kappa',\xi}^\circ(K_0(Q_n))_{\leq h,\mathfrak{m}_{0,Q_n}}	\]
	is annhilated by $u^{(d-1)a}$.  Hence there is some $N\gg 0$ such that $u^NP_n\subset (y_i)P_n$ for all $n$, and by devissage, the modules $P_n/(y_i)^kP_n$ are annhilated by $u^{kN}$ for all $k, n\geq 0$.

	Next, we observe that we have exact sequences
	\[	0\rightarrow M_n^{<r}/(y_i)^k\rightarrow M_n/(y_i)^k\rightarrow P_n/(y_i)^k\rightarrow 0	\]
	for all $n$.  Indeed, there are surjections
	\[	\Tor_1^{\O_L[\![y_i]\!]}(\O_L[\![y_i]\!]/(y_i)^k, P_n)\twoheadrightarrow \ker\left( M_n^{<r}/(y_i)^k\rightarrow M_n/(y_i)^k \right)	\]
	But $M_n^{<r}$ is projective over $\O_L[\![y_1]\!]$, so $M_n^{<r}/(y_i)^k$ has no $u$-torsion, whereas $\Tor_1^{\O_L[\![y_i]\!]}(\O_L[\![y_i]\!]/(y_i)^k, P_n)$ is entirely $u$-power-torsion, because $P_n$ is.

	Let $J_k\subset \O_L[\![y_i]\!]$ be the ideal generated by $u^{Nk}$ and $(y_i)^k$.  Then the Tor long exact sequence gives us exact sequences
	\[	\Tor_1^{\O_L[\![y_i]\!]/(y_i)^k}\left(P_n/(y_i)^k, \O_L[\![y_i]\!]/J_k\right)\rightarrow {M_n'}^{<r}/J_k\rightarrow M_n'/J_k\rightarrow P_n/J_k\rightarrow 0	\]
	Moreover,
	\[	\Tor_1^{\O_L[\![y_i]\!]/(y_i)^k}\left(P_n/(y_i)^k, \O_L[\![y_i]\!]/J_k\right) = (P_n/(y_i)^k)[u^{Nk}] = P_n = P_n/J_k	\]

	Let $P_{J_k}$ denote the localization of $\prod_{n\geq 1}P_n/J_k$ at the ideal corresponding to our chosen ultrafilter.  We have an exact sequence
	\[	P_{J_k}\rightarrow {M_{J_k}'}^{<r}\rightarrow M_{J_k}'\rightarrow P_{J_k}\rightarrow 0	\]
	and since the set $\{J_k\}$ is cofinal in the set of open ideals of $\O_L[\![y_i]\!]$, an exact sequence
	\[	P\rightarrow {M_\infty'}^{<r}\rightarrow M_\infty'\rightarrow P\rightarrow 0	\]
	But since $u^NP\subset (y_i)P$, $P$ is supported on a proper closed subscheme of $\Spf \O_L[\![y_i]\!]$;  away from the support of $P$, the map ${M_\infty'}^{<r}\rightarrow M_\infty'$ is an isomorphism, as desired.
\end{proof}

The support of $M_\infty^{\an}$ over $X_{ {\tri},\overline\rho,{\loc},\leq h}^{\square,\psi',\underline\kappa}|_U\times\Spa\Z_p[\![x_i]\!]$ is a Zariski-closed subspace, whose dimension must therefore be
\[	\dim \Spa R_0\htimes\Z_p[\![y_i]\!]\left[\frac 1 u\right] = \dim U + g + 4\lvert\Sigma_p\rvert	\]
But the morphism $X_{{\tri},\overline\rho,\loc}^{\square,\psi',\underline\kappa}|_U\rightarrow U$ has relative dimension $4\lvert\Sigma_p\rvert$ over an open subspace of $U$ by Proposition~\ref{prop: rel dim loc tri}, so any non-empty irreducible components have total dimension $\dim U+4\lvert\Sigma_p\rvert$.  It follows that the support of $M_\infty^{\an}$ on $X_{ {\tri},\overline\rho,{\loc},\leq h}^{\square,\psi',\underline\kappa}|_U\times\Spa\Z_p[\![x_i]\!]$ is the union of irreducible components.

Finally, since we have a closed embedding
\[	X_{ {\tri},\overline\rho,\infty,\leq h}^{\square,\psi',\underline\kappa}|_U\hookrightarrow X_{{\tri},\overline\rho,{\loc},\leq h}^{\square,\psi',\underline\kappa}|_U\times\Spa\Z_p[\![x_i]\!]	\]
we conclude that the support of $M_\infty$ on $X_{ {\tri},\overline\rho,\infty,\leq h}^{\square,\psi,\underline\kappa}|_U$ is also a union of irreducible components, which we denote $\mathscr{X}_{\underline D^\times,U,\leq h}^{\infty,\psi,\underline\kappa}$.

We have a sequence of morphisms
\[	X_{ {\tri},\overline\rho,\infty}^{\square,\psi',\underline\kappa}|_U\hookrightarrow X_{{\tri},\overline\rho,{\loc}}^{\square,\psi',\underline\kappa}|_U\times\Spa\Z_p[\![x_i]\!]\rightarrow \G_{m,U}^{\ad}\times\Spa\Z_p[\![y_i]\!]	\]
(where we send the product of the factors of $\G_m^{\ad}$ in the definition of the trianguline varieties to the factor of $\G_m^{\ad}$ on the right, corresponding to the action of $U_p^{-1}$); $M_\infty$ is a finite module on $X_{ {\tri},\overline\rho,\infty}^{\square,\psi',\underline\kappa}|_U$ whose pushforward to $\G_{m,U}^{\ad}\times\Spa\Z_p[\![y_i]\!]$ is also finite.  

We summarize this discussion:
\begin{thm}
	There is a space $\mathscr{X}_{\underline D^\times,U,\leq h}^{\infty,\psi,\underline\kappa}$ (which we call the patched eigenvariety over $U$), a finite module $M_\infty$ supported on $\mathscr{X}_{\underline D^\times,U,\leq h}^{\infty,\psi,\underline\kappa}$, (which we call the patched module) and a morphism
	\[	\mathscr{X}_{\underline D^\times,U,\leq h}^{\infty,\psi,\underline\kappa}\rightarrow \Spa X_{{\tri},\overline\rho,{\loc},\leq h}^{\square,\psi',\underline\kappa}|_U\times\Spa \Z_p[\![x_i]\!]	\]
	whose image is the union of irreducible components.
\end{thm}

Since this morphism factors through the global trianguline variety, we also deduce the following corollary:
\begin{cor}\label{cor: patched module support}
	The support of $M_\infty/(y_1,\ldots,y_{k})$ in the trianguline variety over $X_{ {\tri},\overline\rho,\infty,\leq h}^{\square,\psi',\underline\kappa}|_U/(y_1,\ldots,y_{k})\cong X_{{\tri},\Sigma_p,\overline\rho,\leq h}^{\square,\psi',\underline\kappa}|_U$ is a union of irreducible components.
\end{cor}

\begin{remark}
	We carried out this construction locally, because it is difficult to study the behavior of $M_\infty^{<r}$ and $M_\infty$ under rational localization; we have not checked that the analytic patched modules $M_\infty^{\an}$ form a coherent sheaf. However, because specialization maps induce surjections on patched modules, as $(U,h)$ varies over slope data, the supports of patched modules glue to a global patched eigenvariety $\mathscr{X}_{\underline D^\times}^{\infty,\psi,\underline\kappa}$.
\end{remark}

\subsection{Modularity}

We are now in a position to prove Theorem~\ref{thm: main thm precise}.  We will say that a Galois representation $\rho$ is \emph{modular} if it comes from a point on the extended eigenvariety.

\begin{prop}
	Let $F/\Q$ be a real quadratic extension split at $p$, such that the image of $\overline\rho|_{\Gal_F}$ contains $\SL_2(\F_p)$.  Then $\rho:\Gal_{\Q}\rightarrow \GL_2(L)$ is modular if and only if $\rho|_{\Gal_F}$ is modular.
	\label{prop: modularity cyclic base change}
\end{prop}
\begin{proof}
	We have the cyclic base-change morphism $\mathscr{X}_{\GL_2/\Q,\mathrm{cusp}}\rightarrow \mathscr{X}_{\GL_2/F,\mathrm{mid}}$ from \textsection \ref{section: cyclic base change}, so if $\rho$ corresponds to $x\in \mathscr{X}_{\GL_2/\Q,\mathrm{cusp}}$, then $\rho|_{\Gal_F}$ corresponds to the image of $x$ in $\mathscr{X}_{\GL_2/F,\mathrm{mid}}$.  To show the other direction, we note that if $\rho|_{\Gal_F}$ is associated to $x'\in \mathscr{X}_{\GL_2/F}$, then the corresponding eigenvalues are fixed by $\Gal(F/\Q)$.  Since we assumed that the image of $\overline\rho|_{\Gal_F}$ contains $\SL_2(\F_p)$, by ~\cite[Theorem B.0.1]{bergdall-hansen} we may apply Corollary ~\ref{cor: image cyclic base change} to conclude that $x'$ is in the image of $\mathscr{X}_{\GL_2/\Q,\mathrm{cusp}}$.
\end{proof}

Choose $F/\Q$ a real quadratic extension split at $p$.  We may additionally choose $F$ such that the image of $\overline\rho|_{\Gal_F}$ contains $\SL_2(\F_p)$, by requiring that $\ell$ splits in $F$ for $\ell$ in some finite set of primes $S$ of $\Q$ such that $\{\overline\rho(\Frob_\ell)\}_{\ell\in S}$ generate $\SL_2(\F_p)$.  Maintaining the notation of the previous section, we let $D/F$ be a totally definite quaternion algebra, split at all finite places, and we let $R:=\O_E[\![T_0/\overline{Z(K)}]\!]$.  The Jacquet--Langlands correspondence gives us a morphism of eigenvarieties $\mathscr{X}_{\underline D^\times}\rightarrow \mathscr{X}_{\GL_2/F}$, so it suffices to show that $\rho|_{\Gal_F}$ corresponds to a point on $\mathscr{X}_{\underline D^\times}$.

\begin{thm}
	$\rho|_{\Gal_F}$ corresponds to a point on $\mathscr{X}_{\underline D^\times}$.
\end{thm}
\begin{remark}
	There is some $h\in\Q_{\geq 0}$ such that $\rho|_{\Gal_{F_v}}$ is trianguline with parameter of slope-$\leq h$ for each $v\mid p$, and there is some open affinoid $U\subset \mathscr{W}_F$ containing the weight of $\rho|_{\Gal_{F_v}}$ such that $(U,h)$ is a slope datum for $\mathscr{X}_{\underline D^\times}$.  In the following proof, we will work with a patched eigenvariety $X_{ {\tri},\overline\rho,\infty,\leq h}^{\square,\psi',\underline\kappa}|_U$.  However, since our arguments only require working sufficiently close to the boundary, we are always free to shrink $U$ (or increase $h$).  For compactness of exposition and notation, we therefore suppress $(U,h)$ from the proof.

	Note that we may also arrange $(U,h)$ so that the slope-$\leq h$ part of $\mathscr{X}_{\underline D^\times}|_U$ is non-empty.  Indeed, we have assumed that $\overline\rho$ is modular, so there are components of $\mathscr{X}_{\underline D^\times}$ with residualGalois representation $\overline\rho|_{\Gal_F}$.  Then ~\cite[Corollary 4.2.4]{johansson-newton19} shows that such components contain classical points with parallel weight $2$, and the proof shows that we may find such points arbitrarily close to the boundary by varying the $p$-nebentypus.  Such classical points have slope-$\leq 1$, so we may ensure non-emptiness by increasing $h$ (and shrinking $U$, if necessary).
\end{remark}
\begin{proof}
	Let $\rho_0:=\rho|_{\Gal_F}$.   We have assumed that $\rho|_{\Gal_{\Q_p}}$ is trianguline, so we may write $D_{\rig}(\rho|_{\Gal_{\Q_p}})$ as an extension of rank-$1$ $(\varphi,\Gamma)$-modules:
	\[	0\rightarrow \Lambda_{L,\rig}(\delta_1)\rightarrow D_{\rig}(\rho|_{\Gal_{\Q_p}})\rightarrow \Lambda_{L,\rig}(\delta_2)\rightarrow 0	\]
	for characters $\delta_1,\delta_2:\Q_p^\times\rightrightarrows L^\times$.  After twisting, we may assume that $\delta_1|_{\Z_p^\times}$ is trivial.  We fix a weight $\kappa_0$ according to $\delta_1|_{\Z_p^\times}$ and $\delta_2|_{\Z_p^\times}$,
	and we fix an unramified character $\psi_0:\Gal_F\rightarrow \O_E[\![T_0/\overline{Z(K)}]\!]^\times$ deforming $\chi_{\cyc}\kappa_{0,1}\kappa_{0,2}\det\rho|_{\Gal_F}$.

	It is enough to show that the point $x_0\in X_{{\tri},\overline\rho_0}^{\square,\psi',\underline\kappa}$ corresponding to $\rho_0$ is in the support of the patched module ${M}_\infty$.  To see this, we treat separately the cases when $\rho$ is ordinary or non-ordinary at $p$.

	We first assume that $\rho$ is ordinary at $p$.  Then $\rho_0|_{\Gal_{F_v}}$ has the form
	\[	0\rightarrow \chi_{1,v}\rightarrow \rho_0|_{\Gal_{F_v}}\rightarrow \chi_{2,v}\rightarrow 0	\]
	for each $v\mid p$, where $\chi_{i,v}:\Gal_{F_v}\rightrightarrows \O_E^\times$ are characters.

	We wish to consider the slope-$0$ trianguline variety, which is the same as considering ordinary deformations of these extensions (with determinant fixed).  The characters $\chi_{i,v}|_{\O_{F_v}}$ deform over weight space, and the unramified characters specified by $\chi_{1,v}(p)$ correspond to a point on $\Spf \O_E[\![\{t_v\}_{v\mid p}]\!]$.  Note that since we specify a determinant at every point of weight space, deforming $\chi_{1,v}(p)$ also uniquely deforms $\chi_{2,v}(p)$.

	Thus, we see that the space
	\[	\Spf \O_E[\![T_0/\overline{Z(K)},\{t_v\}_{v\mid p}]\!]	\]
	is the moduli space of pairs of characters.  Moreover, away from a Zariski-closed subspace of this space, the space of extensions of the universal characters is a rank-$1$ vector bundle.  In particular, since $\kappa_0$ is regular, over an open neighborhood of $\kappa_0$ in $\mathscr{W}_F$, this moduli space is irreducible (and contains the point corresponding to $\rho_0$).  Adding framing variables and passing from extensions to Galois representations preserve irreducibility of the moduli space.

	Thus, it suffices to show that the ordinary patched module for some characteristic $0$ weight $\kappa_1$ sufficiently near the boundary is supported on the fiber of this moduli space.  We choose $\kappa_1$ so that it is parallel.  But the ordinary part of the Coleman--Mazur eigencurve is finite flat and surjective over weight space, so we may choose ordinary overconvergent eigenforms of appropriate weight and transfer them to $\underline D^\times$, where they contribute to the support of the patchd module, as desired.

	We now assume that $\rho$ is not ordinary at $p$, so that $\rho_0$ is not ordinary at either place above $p$.
	Since the parameters of $D_{\rig}(\rho)$ were assumed regular, $x_0$ is a smooth point of $X_{ {\tri},\overline\rho_0,{\loc}}^{\square,\psi',\underline\kappa}$.  Therefore, $x_0$ is contained in a unique irreducible component $V$ of $X_{ {\tri},\overline\rho_0,{\loc}}^{\square,\psi',\underline\kappa}$, and we can find an open affinoid neighborhood $V'\subset V$ of $x_0$ so that $V'$ contains no ordinary parameters.  

	It follows that $\rho_0$ can be analytically deformed to characteristic $0$ (as in Example~\ref{example: tri var over wt space residual representation}).  In fact, the morphism $X_{ {\tri},\overline\rho_0,{\loc}}^{\square,\psi',\underline\kappa}\rightarrow \mathscr{W}_F$ is smooth in a neighborhood of $x_0$; in particular, it is open, and we may assume that the weights corresponding to the first term in the triangulation remain trivial as we deform, at both places above $p$.

	Recall that for any $p$-adic field $K/\Q_p$, given a character $\delta:K^\times\rightarrow \overline\Q_p^\times$, its \emph{weight} $(\mathrm{wt}_\sigma(\delta))_{\sigma:K\hookrightarrow \overline\Q_p}$ is the tuple such that
	\[	\lim_{a\rightarrow 0}\frac{\lvert \delta(1+a)-1 + \sum_{\sigma}\mathrm{wt}_\sigma(\delta)\sigma(a)\rvert}{\lvert a\rvert} = 0	\]
	We say that $\delta$ is \emph{locally algebraic of weight $(k_\sigma)_\sigma$} if $\mathrm{wt}_\sigma(\delta)=k_\sigma\in\Z$ for all $\sigma$; equivalently, the restriction of $\delta$ to some open subgroup of $\O_K^\times$ is $\chi_{\underline k}:=x\mapsto\prod_\sigma \sigma(x)^{-k_\sigma}$.  If $\underline\delta$ is the parameter of a trianguline $(\varphi,\Gamma)$-module, we say that it is \emph{locally algebraic of strongly dominant weight} if $\delta_{i,\sigma}$ is locally algebraic of weight $(k_{i,\sigma})$ and $k_{i,\sigma}<k_{i+1,\sigma}$, for all $i$ and $\sigma$.

	We claim that there is some locally algebraic strongly dominant $\kappa_1\in\mathscr{W}_F^{\rig}$ such that the fiber $V^{\rig}|_{\kappa_1}$ is equidimensional of dimension $8$.  Indeed, since we chose a deformation of $(\rho_0,\underline\delta)$ keeping $\delta_1|_{\O_{F,p}^\times}$ trivial, it is enough to see that there are points of $\Spf \Z_p[\![\Z_p^\times]\!]^{\rig}$ which are locally algebraic of positive weight arbitrarily close to the boundary, which follows from the calculations of e.g. ~\cite[\textsection 2.7]{johansson-newton19}.  
	In fact, we can choose $\kappa_1$ to be locally algebraic of weights $\{0,1\}$ at both places above $p$.

	We further claim that we may assume that $V^{',\rig}|_{\kappa_1}$ consists of points corresponding to Galois representations which are potentially Barsotti--Tate at both places above $p$.  Indeed, $2$-dimensional (characteristic $0$) $(\varphi,\Gamma_{\Q_p})$-modules are classified in ~\cite[\textsection 3.3]{colmez2008a}, and after possibly replacing $\kappa_1$ with a weight closer to the boundary which is locally algbraic of weights $\{0,1\}$, the corresponding $(\varphi,\Gamma_{\Q_p})$-modules are crystabelline.

	Now it suffices to show that the patched module for weight $\kappa_1$ is supported on all of $V^{',\rig}|_{\kappa_1}$.  But a dense open subspace of $V^{',\rig}|_{\kappa_1}$ is a subspace of the generic fiber of one of the potentially Barsotti--Tate deformation rings constructed in ~\cite{kisin2009pst}.  Then we may apply ~\cite[Theorem 3.4.11]{kisin2009moduli} (with some hypotheses relaxed in ~\cite{gee2006modularity}).
\end{proof}

\appendix

\section{Extensions of Zariski-closed subsets}

The paper ~\cite{lourenco} proves Riemann extension theorems for functions on normal pseudorigid spaces and normal excellent formal schemes; in this appendix we use those results to extend certain Zariski-closed adic subsets (in the sense of ~\cite[\textsection 2.1]{johansson-newton17}) of pseudorigid spaces over missing subsets of codimension at least $2$.

\begin{prop}\label{prop: formal model of subspace}
	Let $\mathfrak{X}$ be a normal excellent formal scheme, which is nowhere discrete.  If $Z\subset X:=\mathfrak{X}^{\an}$ is a Zariski-closed adic subset, then there is a closed formal subscheme $\mathfrak{Z}\subset\mathfrak{X}$ such that $Z=\mathfrak{Z}^{\an}$.
\end{prop}
\begin{proof}[Proof of Proposition~\ref{prop: formal model of subspace}]
	We may assume that $\mathfrak{X}=\Spf R$, where $R$ is a normal excellent domain with ideal of definition $J=(f_1,\ldots,f_n)$.  Then by the definitions of ~\cite[\textsection 2.1]{johansson-newton17}, there is a coherent sheaf $\mathcal{I}\subset\O_X$ of ideals such that $Z=\{x\in X\mid \mathcal{I}_x\neq \O_{X,x}\}$.  We need to show that there is an ideal $I\subset R$ whose associated sheaf agrees with $\mathcal{I}$ on $X$.

	We define $\mathcal{I}^+:=\mathcal{I}\cap\O_X^+$, and we set $I:=\Gamma(X,\mathcal{I}^+)$; by ~\cite[Proposition 6.2]{lourenco}, $R=\Gamma(X,\O_X^+)$, so we may view $I$ as an ideal of $R$.  It remains to show that for each affinoid open subspace $\Spa R'\subset X$, $R'\otimes_RI=\mathcal{I}(\Spa R')$.  To see this, we observe that we have a finite cover $X=\cup_i\Spa R\left\langle \frac{J}{f_i}\right\rangle$, so it suffices to check this with $R'=\Spa R\left\langle \frac{J}{f_i}\right\rangle$.

	Setting $R_i:=R\left\langle \frac{J}{f_i}\right\rangle$ and $U_i=\Spa R\left\langle \frac{J}{f_i}\right\rangle$, we have an exact sequence of $R$-modules
	\[	0\rightarrow I\rightarrow \prod_i\mathcal{I}^+(U_i)\rightrightarrows \prod_{i,j}\mathcal{I}^+(U_i\cap U_j)	\]
	For any fixed index $i_0$, we may tensor with $R_{i_0}^\circ$ and complete $f_{i_0}$-adically; as $R$ is noetherian, our sequence
	\[	0\rightarrow R_{i_0}^\circ\htimes_RI\rightarrow \prod_i\left(R_{i_0}^\circ\htimes_R\mathcal{I}^+(U_i)\right)\rightrightarrows \prod_{i,j}\left(R_{i_0}^\circ\htimes_R\mathcal{I}^+(U_i\cap U_j)\right)	\]
	remains exact.  But $R_{i_0}^\circ\htimes_R\mathcal{I}^+(U_i)$ generates $\mathcal{I}(U_{i_0}\cap U_i)$ and $R_{i_0}^\circ\htimes_R\mathcal{I}^+(U_i\cap U_j)$ generates $\mathcal{I}(U_{i_0}\cap U_i\cap U_j)$ after inverting a pseudouniformizer $u_{i_0}$ of $R_{i_0}$ for all $i,j$, and $\{U_{i_0}\cap U_i\}_i$ is a cover of $U_{i_0}$, so in fact $R_{i_0}\htimes_RI=\mathcal{I}(U_{i_0})$, as desired.
\end{proof}

\begin{cor}
	\label{cor: integral model in product}
	Let $E$ be a $p$-adic field, let $\mathfrak{X}=\Spa R_1$, where $R_1:=\O_E[\![x_1,\ldots,x_{n_1}]\!]\left\langle y_1,\ldots,y_{n_2}\right\rangle/I$, and let $\mathfrak{Y}:=\Spa R_2$, where $R_2=\O_E[\![z_1,\ldots,z_{m_1}]\!]\left\langle w_1,\ldots,w_{m_2}\right\rangle/J$.  Suppose that $R_1$ has dimension at least $2$.  
	Suppose that $Z\subset \mathfrak{X}^{\an}\times_{\O_E}\mathfrak{Y}$ is a Zariski-closed subset and that there is some integer $N\geq 1$ such that $Z\cap \Spa R_1\left\langle\frac p u, \left\{\frac{x_i}{u}\right\}\right\rangle\left[\frac 1 u\right]\times\mathfrak{Y}$ is contained in the rational domain $\{z_j^N\leq u\neq 0 \text{ for all } j=1,\ldots,m\}$ for each $u\in \{\varpi_E,x_1,\ldots,x_{n_1}\}$.  Then there is a closed formal subscheme $\mathfrak{Z}\subset \mathfrak{X}\htimes_{\O_E}\mathfrak{Y}$ such that $\mathfrak{Z}^{\an}\cap (\Spa R_1)^{\an}\times_{\O_E}\mathfrak{Y}=Z$ (where the intersection is taken inside $\left(\Spa R_1\htimes_{\O_E}R_2\right)^{\an}$).
\end{cor}
\begin{proof}
	Replacing $Z$ with $Z\cap V(I)^{\an}\cap V(J)^{\an}$, we may assume that $I=J=(0)$, so that $R_1=\O_E[\![\{x_i\}]\!]\left\langle \{y_k\}\right\rangle$ and $R_2=\O_E[\![\{z_j\}]\!]\left\langle \{w_\ell\}\right\rangle$.  Then by Proposition~\ref{prop: formal model of subspace} it suffices to extend $Z$ to a Zariski-closed subset of $\left(\Spa R_1\htimes_{\O_E}R_2\right)^{\an}$.  This analytic locus, in turn, is covered by the affinoid pseudorigid spaces 
	\[	\mathcal{V}_u:=\Spa \O_E[\![\{x_i\}_i,\{z_j\}_j]\!]\left\langle \frac{\varpi_E}{u},\left\{\frac{x_i}{u}\right\}_i,\{y_k\}_k, \left\{\frac{z_j}{u}\right\}_j,\{w_\ell\}_\ell\right\rangle\left[\frac 1 u\right]	\]
		for $u\in \{\varpi_E,x_1,\ldots,x_{n_1}\}$
and
\[	\mathcal{V}_{z_{j_0}}:=\Spa \O_E[\![\{x_i\}_i,\{z_j\}_j]\!]\left\langle\frac{\varpi_E}{z_{j_0}}, \left\{\frac{x_i}{z_{j_0}}\right\}_i,\{y_k\}_k,\left\{\frac{z_j}{z_{j_0}}\right\}_j,\{w_\ell\}_\ell\right\rangle\left[\frac{1}{z_{j_0}}\right]	\]
for $j_0=1,\ldots,m$.  Since the $\mathcal{V}_u$ are contained in $(\Spa R_1)^{\an}\times_{\O_E}R_2$, we only need to extend the intersections $Z\cap \mathcal{V}_{z_{j_0}}$ to Zariski-closed subsets of $\mathcal{V}_{z_{j_0}}$.

We can cover each $\mathcal{V}_{z_{j_0}}$ by its open subspaces defined by inequalities
\[	\mathcal{V}_{z_{j_0},\varpi_E,1}:=\{\lvert z_{j_0}^{N+1}\rvert\leq \lvert \varpi_E\rvert\},\qquad \mathcal{V}_{z_{j_0},i,1}:= \{\lvert z_{j_0}^{N+1}\rvert\leq \lvert x_{i}\rvert\}	\]
for $i=1,\ldots, n_1$ and
\[	\mathcal{V}_{z_{j_0},2}:=\{ \lvert\varpi_E\rvert\leq \lvert z_{j_0}^{N+1}\rvert\neq 0\text{ and } \lvert x_i\rvert\leq \lvert z_{j_0}^{N+1}\rvert\neq 0\text{ for all }i\}	\]

so it suffices to find suitable Zariski-closed subsets of each of these spaces.  

By assumption, $Z\cap \mathcal{V}_{z_{j_0},2}$ is empty.  Moreover,
\[	\mathcal{V}_{z_{j_0},p,1}\subset (\Spa R_1)^{\an}\times_{\O_E}R_2	\]
and
\[	\mathcal{V}_{z_{j_0},i,1}\subset (\Spa R_1)^{\an}\times_{\O_E}R_2 \]
since the conditions $\lvert z_{j_0}^{N+1}\rvert\leq \lvert \varpi_E\rvert$ and $z_{j_0}\neq 0$ imply $\varpi_E\neq 0$ (and similarly for $\{\lvert z_{j_0}^{N+1}\rvert\leq \lvert x_{i}\rvert\}$ and $z_{j_0}\neq 0$). Thus, $Z\cap \mathcal{V}_{z_{j_0},p,2}$, $Z\cap \mathcal{V}_{z_{j_0},i,2}$ are defined by sheaves of ideals which agree on intersections $\mathcal{V}_{z_{j_0},i,1}\cap\mathcal{V}_{z_{j_0},i',1}$.   

By construction, these sheaves agree on the overlaps $\mathcal{V}_{z_{j_0},p,1}\cap\mathcal{V}_{z_{j_0},p,2}=\{\lvert z_{j_0}^{N+1}\rvert=\lvert \varpi_E\rvert\}$ and $\mathcal{V}_{z_{j_0},i,1}\cap\mathcal{V}_{z_{j_0},i,2}=\{\lvert z_{j_0}^{N+1}\rvert=\lvert x_i\rvert\}$.  We have therefore extended the sheaf of ideals defining $Z\cap \mathcal{V}_{z_{j_0}}\cap (\Spa R_1)^{\an}\times_{\O_E}R_2$ to a sheaf of ideals on all of $\mathcal{V}_{z_{j_0}}$, as desired.
\end{proof}

\bibliographystyle{amsalpha}
\bibliography{bc-tate}

\end{document}